\documentclass[english]{article}
\usepackage[T1]{fontenc}
\usepackage[latin9]{inputenc}
\usepackage[a4paper]{geometry}
\usepackage{babel}
\usepackage{array}
\usepackage{float}
\usepackage{booktabs}
\usepackage{mathrsfs}
\usepackage{mathtools}
\usepackage{multirow}
\usepackage{amsmath}
\usepackage{amsthm}
\usepackage{amssymb}
\usepackage{graphicx}
\usepackage{wasysym}
\usepackage{nomencl}
\providecommand{\printnomenclature}{\printglossary}
\providecommand{\makenomenclature}{\makeglossary}
\makenomenclature
\usepackage[unicode=true,pdfusetitle,
 bookmarks=true,bookmarksnumbered=false,bookmarksopen=false,
 breaklinks=false,pdfborder={0 0 1},backref=false,colorlinks=false]
 {hyperref}

\makeatletter

\providecommand{\tabularnewline}{\\}

\numberwithin{equation}{section}
\numberwithin{figure}{section}
\numberwithin{table}{section}
\theoremstyle{plain}
\newtheorem{thm}{\protect\theoremname}[section]
\theoremstyle{definition}
\newtheorem{defn}[thm]{\protect\definitionname}
\theoremstyle{plain}
\newtheorem{conjecture}[thm]{\protect\conjecturename}
\theoremstyle{plain}
\newtheorem{lem}[thm]{\protect\lemmaname}
\theoremstyle{plain}
\newtheorem{prop}[thm]{\protect\propositionname}
\theoremstyle{remark}
\newtheorem{rem}[thm]{\protect\remarkname}

\usepackage{color}
\usepackage{diagbox}
\usepackage{graphicx}
\usepackage{titling}
\predate{}
\postdate{}
\date{}

\DeclareMathOperator{\aff}{aff}

\DeclareMathOperator{\area}{area}

\DeclareMathOperator{\conv}{conv}

\DeclareMathOperator{\dist}{dist}

\DeclareMathOperator{\inter}{int}

\DeclareMathOperator{\lin}{lin}

\DeclareMathOperator{\pos}{pos}

\DeclareMathOperator{\relint}{relint}

\DeclareMathOperator{\vol}{vol}

\makeatother

\providecommand{\conjecturename}{Conjecture}
\providecommand{\definitionname}{Definition}
\providecommand{\lemmaname}{Lemma}
\providecommand{\propositionname}{Proposition}
\providecommand{\remarkname}{Remark}
\providecommand{\theoremname}{Theorem}

\begin{document}
\title{On the Sausage Catastrophe in $4$ Dimensions }
\author{Ji Hoon Chun\thanks{The research of the author was supported by the Deutsche Forschungsgemeinschaft
(DFG) Graduiertenkolleg \textquotedbl Facets of Complexity/Facetten
der Komplexität\textquotedbl{} (GRK 2434). } }
\maketitle
\begin{abstract}
The Sausage Catastrophe of J. Wills (1983) is the observation that
in $d=3$ and $d=4$, the densest packing of $n$ spheres in $\mathbb{R}^{d}$
is a sausage for small values of $n$ and jumps to a full-dimensional
packing for large $n$ without passing through any intermediate dimensions.
Let $n_{d}^{*}$ be the smallest value of $n$ for which the densest
packing of $n$ spheres in $\mathbb{R}^{d}$ is full-dimensional and
$N_{d}^{*}$ be the smallest value of $N$ for which the densest packing
of $N$ spheres in $\mathbb{R}^{d}$ is full-dimensional for all $N\geq N_{d}^{*}$.
We extend the work of Gandini and Zucco (1992) to obtain new upper
bounds of $n_{4}^{*}\leq338,\!196$ and $N_{4}^{*}\leq516,\!946$.
Some lengthy and repetitive components of the proof of the latter
result were obtained using interval arithmetic. 
\end{abstract}

\section{\protect\label{sec: Background and history of the Sausage Catastrophe}Introduction }

Let $B^{d}:=\left\{ \mathbf{x}\in\mathbb{R}^{d}\,\middle|\,\left|\mathbf{x}\right|\leq1\right\} $\nomenclature[A01 B^d]{$B^{d}$}{The $d$-dimensional unit ball}
be the \textbf{unit ball}\index{ball} (informally referred to as
a ``sphere'') in $\mathbb{R}^{d}$, where $\left|\mathbf{v}\right|:=\sqrt{v_{1}^{2}+\cdots+v_{d}^{2}}$\nomenclature[A02 abs v]{$\left|\mathbf{v}\right|$, $\mathbf{v}$ is a vector}{The Euclidean ($2$-norm) length of $\mathbf{v}$}
for a vector $\mathbf{v}=\left(v_{1},\ldots,v_{d}\right)^{\mathsf{T}}\in\mathbb{R}^{d}$. 
\begin{defn}
A set $C\subset\mathbb{R}^{d}$\nomenclature[A03 C]{$C$}{A sphere packing (finite or infinite)}
is a \textbf{packing set}\index{packing set} of the closed unit ball
$B^{d}$ (colloquially called a ``sphere packing'') if $\left(\mathbf{x}+\inter B^{d}\right)\cap\left(\mathbf{y}+\inter B^{d}\right)=\emptyset$
for all distinct $\mathbf{x},\mathbf{y}\in C$, where $\inter X$
is the interior of $X$\nomenclature[A04 int X]{$\text{int}\,X$}{The interior of $X$}. 
\end{defn}

Extensive information on packings can be found in Conway and Sloane
(1999) \cite{ConwaySloane1999} for infinite packings and Böröczky
Jr. (2004) \cite{Boeroeczky2004} for finite packings. Additionally,
Zong (1999) \cite{Zong1999} presents detailed summaries of several
major results and problems in both the infinite and finite settings.
For $d\in\mathbb{N}$, we denote the set of all sphere packings with
$n\in\mathbb{N}=\left\{ 1,2,\ldots\right\} $ points in $\mathbb{R}^{d}$
by $\mathscr{P}_{n}^{d}$\nomenclature[A05 P_{n}^{d}(B^d)]{$\mathscr{P}_{n}^{d}$}{The set of all sphere packings of $B^{d}$ in $\mathbb{R}^{d}$ with $n$ points}
and the set of all sphere packings with infinitely many points in
$\mathbb{R}^{d}$ by $\mathscr{P}^{d}$\nomenclature[A06 P^{d}(B^d)]{$\mathscr{P}^{d}$}{The set of all infinite sphere packings of $B^{d}$ in $\mathbb{R}^{d}$}.
Let $\vol_{k}$\nomenclature[A07 vol_k]{$\text{vol}_{k}$}{The $k$-dimensional Lebesgue measure ($k \leq d$)}
denote the $k$-dimensional Lebesgue measure and $\kappa_{d}:=\vol_{d}\left(B^{d}\right)$\nomenclature[A08 kappa_d]{$\kappa_{d}$}{The volume of the $d$-dimensional unit ball}.
The notations $\vol$\nomenclature[A09 vol]{$\text{vol}$}{The $d$-dimensional Lebesgue measure},
$\area$\nomenclature[A10 area]{$\text{area}$}{The $2$-dimensional Lebesgue measure., i.e., the area},
and $\ell$\nomenclature[A11 length]{$\ell$}{The $1$-dimensional Lebesgue measure., i.e., the length}
are used for $k=d$, $k=2$, and $k=1$ respectively. Throughout this
paper we use $n$\nomenclature[A12 n]{$n$}{The number of points in a packing}
for the number of points in a packing and $d$\nomenclature[A13 d]{$d$}{The dimension of the Euclidean space $\mathbb{R}^{d}$ under consideration}
for the dimension of the underlying Euclidean space. 
\begin{defn}
\label{def: Finite and infinite packing densities}Let $C$ be a packing
set. If $\left|C\right|<\infty$, where $\left|C\right|$\nomenclature[A14 abs X]{$\left|X\right|$, $X$ is a set}{The cardinality of $X$}
is its cardinality, then the \textbf{finite packing density}\index{density, finite packing}
of $C$ is defined by\nomenclature[A15 delta(C),  C finite]{$\delta\left(C\right)$,  $C$ finite}{The finite packing density of $C$}
\begin{equation}
\delta\left(C\right):=\frac{\left|C\right|\kappa_{d}}{\vol\left(\conv C+B^{d}\right)}\label{eq: Finite packing density}
\end{equation}
 If $\left|C\right|=\infty$ then the \textbf{(infinite) packing density}\index{density, infinite packing}
of $C$ is defined by (see \cite{Zong1999}, Section 1.1)\nomenclature[A16 delta(C),  C infinite]{$\delta\left(C\right)$,  $C$ infinite}{The infinite packing density of $C$}
\[
\delta\left(C\right):=\limsup_{r\rightarrow\infty}\frac{\left|\left\{ \mathbf{x}\in C\,\middle|\,\mathbf{x}+B^{d}\subseteq r\left[-1,1\right]^{d}\right\} \right|\kappa_{d}}{\vol\left(r\left[-1,1\right]^{d}\right)}.
\]
 For a given $d$ and $n$, the \textbf{density of the densest finite
(sphere) packing}\index{density of the densest finite packing} $\delta\left(d,n\right)$
of $n$ spheres and the \textbf{density of the densest infinite (sphere)
packing}\index{density of the densest infinite packing} $\delta\left(d\right)$
are $\delta\left(d,n\right):=\sup\left\{ \delta\left(C_{n}\right)\,\middle|\,C_{n}\in\mathscr{P}_{n}^{d}\right\} $\nomenclature[A17 delta(d,  n)]{$\delta\left(d,  n\right)$}{The density of the densest finite sphere packing of $B^{d}$ in $\mathbb{R}^{d}$ with $n$ points}
and $\delta\left(d\right):=\sup\left\{ \delta\left(C\right)\,\middle|\,C\in\mathscr{P}^{d}\right\} $\nomenclature[A18 delta(d)]{$\delta\left(d\right)$}{The density of the densest infinite sphere packing in $\mathbb{R}^{d}$}
respectively. 

A packing $C$ in $\mathbb{R}^{d}$ is called a \textbf{sausage arrangement}\index{sausage arrangement}
$S_{n}^{d}$\nomenclature[A19 S_{n}^{d}]{$S_{n}^{d}$}{A sausage arrangement with $n$ spheres in $\mathbb{R}^{d}$}
of cardinality $n$ (``sausage'' for short) if there exists a unit
vector $\mathbf{u}\in\mathbb{R}^{d}$, a vector $\mathbf{v}\in\mathbb{R}^{d}$,
and an $n\in\mathbb{N}$ such that $C=\left\{ \mathbf{v}+2i\mathbf{u}\mid i\in\left\{ 0,\ldots,n-1\right\} \right\} $,
that is, the points of $C$ are all on a single line and as close
as possible. 
\end{defn}

The density of a sausage is 
\begin{equation}
\delta\left(S_{n}^{d}\right)=\frac{n\kappa_{d}}{\kappa_{d}+2\left(n-1\right)\kappa_{d-1}},\label{eq: Sausage density}
\end{equation}
 and the \textbf{Sausage Conjecture}\index{Sausage Conjecture} of
L. Fejes Tóth \cite{FejesToth1975} states that in dimensions $d\geq5$,
the optimal finite packing is reached by a sausage. 
\begin{conjecture}[Sausage Conjecture (L. Fejes Tóth, 1975)]
Let $d\geq5$ and $n\in\mathbb{N}$, then $\delta\left(S_{n}^{d}\right)=\delta\left(d,n\right)$,
and the maximum density $\delta\left(d,n\right)$ is only obtained
with a sausage arrangement. Equivalently, $\vol\left(S_{n}^{d}+B^{d}\right)\leq\vol\left(C+B^{d}\right)$
for all $C\in\mathscr{P}_{n}^{d}$, with equality if and only if $C=S_{n}^{d}$.
\end{conjecture}

\subsection{\protect\label{subsec: Polytopes and Steiner's formula}Polytopes
and Steiner's formula }

To make this paper more self-contained, in this subsection we briefly
state the relevant basic notions of discrete and convex geometry regarding
polytopes, hyperplanes, and lattices. This information is available
among the books by Gruber \cite{Gruber2007} and Grünbaum \cite{Gruenbaum2003},
along with Conway and Sloane (1999) \cite{ConwaySloane1999}, the
first volume \cite{GruberWills1993A} of a two-volume compendium edited
by Gruber and Wills, and Henk, Richter-Gebert, and Ziegler's chapter
\cite{GoodmanORourkeToth2018HenkRichterGebertZiegler15} on polytopes
in the handbook \cite{GoodmanORourkeToth2018} edited by Goodman,
O'Rourke, and Tóth. Here we mainly follow the presentations of Gruber
and Henk, Richter-Gebert, and Ziegler. 

A set $P\subseteq\mathbb{R}^{d}$ is a (convex) \textbf{polytope}\index{polytope}
in $\mathbb{R}^{d}$ if $P=\conv\left\{ \mathbf{x}^{1},\ldots,\mathbf{x}^{n}\right\} $
for $\mathbf{x}^{1},\ldots,\mathbf{x}^{n}\in\mathbb{R}^{d}$ and a
finite $n\geq0$. Let $P\subset\mathbb{R}^{d}$ be a convex polytope
and $F_{k}\in\mathscr{F}_{k}\left(P\right)$. The \textbf{normal cone}\index{normal cone}
$N\left(P,F_{k}\right)$\nomenclature[B01 N]{$N\left(P,F_{k}\right)$}{The normal cone at the face $F_{k}$ of $P$}
is the set of all $\mathbf{x}\in\mathbb{R}^{d}$ with the property
that there exists a $\lambda\geq0$ such that $F_{k}\subseteq P\cap\left\{ \mathbf{y}\in\mathbb{R}^{d}\,\middle|\,\mathbf{x}\cdot\mathbf{y}=\lambda\right\} $
and $\mathbf{x}\cdot\mathbf{y}\leq\lambda$ for all $\mathbf{y}\in P$,
and the \textbf{external angle}\index{external angle} $\theta\left(P,F_{k}\right)$\nomenclature[B02 theta]{$\theta\left(P,F_{k}\right)$}{The external angle of the $k$-dimensional face $F_{k}$ of $P$}
is 
\[
\theta\left(P,F_{k}\right):=\frac{\vol\left(\left(N\left(P,F_{k}\right)+\lin\left(F_{k}-\mathbf{x}\right)\right)\cap B^{d}\right)}{\vol\left(B^{d}\right)},
\]
 where $\mathbf{x}\in\relint\left(F_{k}\right)$. For two convex bodies
$K,L\subset\mathbb{R}^{d}$, the \textbf{Minkowski addition}\index{addition, Minkowski}
$K+L$ of $K$ and $L$ is defined by $K+L:=\left\{ \mathbf{x}+\mathbf{y}\in\mathbb{R}^{d}\,\middle|\,\mathbf{x}\in K,\ \mathbf{y}\in L\right\} $,
and here the sum of two sets always denotes the Minkowski addition.
Steiner \cite{Steiner1840} proved a formula which expresses $\vol\left(K+\lambda B^{d}\right)$,
$\lambda\geq0$, as a polynomial in $\lambda$. For our purposes it
is convenient to use the following representation of Steiner's formula
for a convex polytope: 
\begin{equation}
\vol\left(P+B^{d}\right)=\vol\left(P\right)+\sum_{k=1}^{d-1}\sum_{F_{k}\in\mathscr{F}_{k}\left(P\right)}\vol_{k}\left(F_{k}\right)\theta\left(P,F_{k}\right)\kappa_{d-k}+\kappa_{d}\label{eq: Steiner's formula for a convex polytope}
\end{equation}
 (see \cite{GruberWills1993Sangwine-Yager1.2}, Section 3, and \cite{GoodmanORourkeToth2018HenkRichterGebertZiegler15}),
where $\mathscr{F}_{k}\left(P\right)$\nomenclature[B03 F_k(P)]{$\mathscr{F}_{k}\left(P\right)$}{The set of all $k$-dimensional faces of the polytope $P$}
is the set of all $k$-dimensional faces of $P$. Let $\mathbf{x}\in\mathbb{R}^{d}$
and $\lambda\in\mathbb{R}$, then define the \textbf{hyperplane}\index{hyperplane}
$H\left(\mathbf{x},\lambda\right):=\left\{ \mathbf{y}\in\mathbb{R}^{d}\,\middle|\,\mathbf{x}\cdot\mathbf{y}=\lambda\right\} $\nomenclature[B05 H(x, lambda)]{$H\left(\mathbf{x},\lambda\right)$}{The hyperplane $H\left(\mathbf{x},\lambda\right)=\left\{ \mathbf{y}\in\mathbb{R}^{d}\,\middle|\,\mathbf{x}\cdot\mathbf{y}=\lambda\right\}$}
and the closed \textbf{half-space}\index{half-space}\nomenclature[B06 H^{-}(x, lambda)]{$H^{-}\left(\mathbf{x},\lambda\right)$}{The closed half-space $H^{-}\left(\mathbf{x},\lambda\right):=\left\{ \mathbf{y}\in\mathbb{R}^{d}\,\middle|\,\mathbf{x}\cdot\mathbf{y}\leq\lambda\right\}$}
$H^{-}\left(\mathbf{x},\lambda\right):=\left\{ \mathbf{y}\in\mathbb{R}^{d}\,\middle|\,\mathbf{x}\cdot\mathbf{y}\leq\lambda\right\} $.
We say that $H\left(\mathbf{x},\lambda\right)$ is a \textbf{support
hyperplane}\index{support hyperplane} of a closed convex set $K$
at $\mathbf{x}$ if $K\cap H\left(\mathbf{x},\lambda\right)\neq\emptyset$
and $K\subseteq H^{-}\left(\mathbf{x},\lambda\right)$. A \textbf{lattice}
$\Lambda$\nomenclature[B04 Lambda]{$\Lambda$}{A lattice in $\mathbb{R}^{d}$}
in $\mathbb{R}^{d}$ is a discrete subgroup of $\mathbb{R}^{d}$,
and we assume that $\Lambda$ is of \textbf{full rank}\index{rank, full},
that is, $\dim\Lambda=d$. For a given lattice $\Lambda$ and subset
$S\subset\mathbb{R}^{d}$, we define $G\left(S\right):=\left|S\cap\Lambda\right|$\nomenclature[B07 G(S)]{$G\left(S\right)$}{The number of points of a lattice $\Lambda$ in the set $S$}. 

\subsection{\protect\label{subsec: Sausages in dimension 2}Sausages in dimension
$2$ }

The Sausage Conjecture of Fejes Tóth claims that in all dimensions
$d\geq5$, the sausage arrangement gives the densest packing for any
spheres in $\mathbb{R}^{d}$. In general, this statement does not
hold true for $d<5$ due to the presence of infinite packings obtained
from the lattices $A_{2}$, $D_{3}$, and $D_{4}$ (see, for example,
Conway and Sloane \cite{ConwaySloane1999}, Chapter 4) with greater
infinite packing density than the finite packing densities of almost
all sausages in dimensions $2$, $3$, and $4$. Hence it follows
from the definition of the infinite packing density that there exist
finite subsets of $A_{2}$, $D_{3}$, and $D_{4}$ which are denser
than equinumerous sausages in $\mathbb{R}^{2}$, $\mathbb{R}^{3}$,
and $\mathbb{R}^{4}$. 
\begin{center}
\begin{table}[H]
\noindent \begin{centering}
\begin{tabular}{cccrrcr}
\toprule 
 &  & \multicolumn{2}{c}{\textbf{Infinite packings}} &  & \multicolumn{2}{c}{\textbf{Finite packings}}\tabularnewline
Dimension &  & Packing & Density &  & Sausage & Density\tabularnewline
\midrule
\midrule 
$1$ &  & $\mathbb{Z}$ & ${\displaystyle 1.00000}$ &  & $S_{1}^{1}$ & ${\displaystyle 1.00000}$\tabularnewline
\midrule 
\multirow{2}{*}{$2$} & \multirow{2}{*}{} & \multirow{2}{*}{$A_{2}$} & \multirow{2}{*}{${\displaystyle \frac{\pi}{2\sqrt{3}}\approx0.90690}$} & \multirow{2}{*}{} & $S_{1}^{2}$ & ${\displaystyle 1.00000}$\tabularnewline
 &  &  &  &  & $S_{2}^{2}$ & ${\displaystyle \frac{2\pi}{\pi+4}\approx0.87980}$\tabularnewline
\midrule 
\multirow{2}{*}{$3$} & \multirow{2}{*}{} & \multirow{2}{*}{$D_{3}$} & \multirow{2}{*}{${\displaystyle \frac{\pi}{3\sqrt{2}}\approx0.74048}$} & \multirow{2}{*}{} & $S_{3}^{3}$ & ${\displaystyle \frac{3}{4}\approx0.75000}$\tabularnewline
 &  &  &  &  & $S_{4}^{3}$ & ${\displaystyle \frac{16}{22}\approx0.72727}$\tabularnewline
\midrule 
\multirow{2}{*}{$4$} & \multirow{2}{*}{} & \multirow{2}{*}{$D_{4}$} & \multirow{2}{*}{${\displaystyle \frac{\pi^{2}}{16}\approx0.61685}$} & \multirow{2}{*}{} & $S_{9}^{4}$ & ${\displaystyle \frac{27\pi^{2}}{3\pi^{2}+128\pi}\approx0.61723}$\tabularnewline
 &  &  &  &  & $S_{10}^{4}$ & ${\displaystyle \frac{10\pi^{2}}{\pi^{2}+48\pi}\approx0.61429}$\tabularnewline
\midrule 
$5$ &  & $D_{5}$ & ${\displaystyle \frac{\pi^{2}}{15\sqrt{2}}\approx0.46526}$ &  & ${\displaystyle \lim_{n\rightarrow\infty}S_{n}^{5}}$ & ${\displaystyle \frac{8}{15}\approx0.53333}$\tabularnewline
\bottomrule
\end{tabular}
\par\end{centering}
\caption{\protect\label{tab: Densest infinite packings in small dimensions}The
densest known packings in dimensions $1$ through $5$ (see \cite{ConwaySloane1999},
Chapter 1, Table 1.2) compared with sausage packings.}
\end{table}
\par\end{center}

The sausage is never the optimal packing for all nontrivial numbers
(i.e. more than two) of circles \cite{Wegner1986,FejesToth1975}---also
see the discussions in \cite{Boeroeczky2004}, Sections 4.1--4.3,
but the fact that the sausage arrangement is the densest one-dimensional
packing, however, allows us to conclude that the densest finite packing
in $\mathbb{R}^{2}$ is \emph{necessarily} two-dimensional for $n\geq3$,
even if its specific nature is unknown. 

\subsection{\protect\label{subsec: Sausages in dimensions 3 and 4}Sausages in
dimensions $3$ and $4$ }

The situation in three and four dimensions is more complicated. The
sausage is optimal for small numbers of spheres while the best known
packings for large numbers of spheres are full-dimensional. Curiously,
the best known packings are never in-between, a phenomenon known as
the \textbf{Sausage Catastrophe}\index{Sausage Catastrophe}, coined
by Jörg Wills \cite{Wills1983} in 1983. The wide-ranging survey of
Henk and Wills \cite{HenkWills2020} summarizes the progress and results
on the Sausage Catastrophe, including as it relates to general convex
bodies. For convenience we will use phrases such as $C$ ``is denser
than the sausage'' if $\left|C\right|=n$ and $\delta\left(C\right)>\delta\left(S_{n}^{d}\right)$. 

The sausage is trivially optimal for $n\leq2$ or $d=1$. The simplest
nontrivial case is $d=2$, for which the packing $C_{n}^{\hexagon}$
is denser than the sausage for all $n\geq3$. Zong (\cite{Zong1999},
Section 13.4, Example 13.1) provides a concise proof of the optimality
of the sausage for $n=3$ in any dimension $d\geq3$. For the general
case a couple of definitions are needed---see Wills (1983, 1985)
\cite{Wills1983,Wills1985}. 
\begin{defn}
\label{def: n*_d and N*_d}Let $d\geq2$, $n\geq1$, $\mathscr{C}_{n}^{d}$
be the subset of $\mathscr{P}_{n}^{d}$ consisting of all packings
with greatest density in $\mathscr{P}_{n}^{d}$, and $k\left(d,n\right):=\max\left\{ \dim C\,\middle|\,C\in\mathscr{C}_{n}^{d}\right\} $.
Then define\nomenclature[C01 n*_d]{$n_{d}^{*}$}{For a given $d$, $n_{d}^{*}$ is the smallest number such that the sausage is not the densest packing of $n$ points in $\mathbb{R}^{d}$}\nomenclature[C02 N*_d]{$N_{d}^{*}$}{For a given $d$, $N_{d}^{*}$ is the smallest number such that the sausage is never the densest packing of $n$ points in $\mathbb{R}^{d}$ for all $N\geq N_{d}^{*}$}
\begin{align*}
n_{d}^{*} & :=\min\left\{ n\in\mathbb{N}\,\middle|\,k\left(d,n\right)>1\right\} ,\\
N_{d}^{*} & :=\min\left\{ N\in\mathbb{N}\,\middle|\,k\left(d,N'\right)>1\text{ for all }N'\geq N\right\} .
\end{align*}
\end{defn}

These are thresholds for the crossover from sausage to non-sausage
($\dim>1$) packings. For a given $d$, $n_{d}^{*}$ is the smallest
number such that the sausage is not the densest packing of $n$ points
in $\mathbb{R}^{d}$, while the threshold $N_{d}^{*}$ is the smallest
number such that the sausage is never the densest packing of $N\geq N_{d}^{*}$
points in $\mathbb{R}^{d}$. It follows from the discussion in Subsection
\ref{subsec: Sausages in dimension 2} that $n_{2}^{*}=N_{2}^{*}=3$. 

Originally, Wills (1983) \cite{Wills1983} provided the upper bounds
of $n_{3}^{*}\leq56$ and $n_{4}^{*}\leq5,\!900,\!000$, which were
obtained by the general method of taking a large finite subset of
the densest infinite lattice in each dimension. Since $\delta\left(D_{3}\right)>\lim_{n\rightarrow\infty}\delta\left(S_{n}^{3}\right)$
and $\delta\left(D_{4}\right)>\lim_{n\rightarrow\infty}\delta\left(S_{n}^{4}\right)$,
it is always possible to choose appropriate finite subsets of the
$D_{3}$ and $D_{4}$ lattices that exceed the densities of the corresponding
sausage packings. The first nontrivial lower bounds of $n_{3}^{*}\geq4$
and $n_{4}^{*}\geq5$ arose only a year later as a consequence of
a general result from Betke and Gritzmann (1984) \cite{BetkeGritzmann1984}
on the Sausage Conjecture. The latter inequality remains the best
known lower bound today, but the former inequality was improved by
Böröczky Jr. (1993) \cite{Boeroeczky1993} to $n_{3}^{*}\geq5$. Note
that the presence of a full-dimensional packing of $n$ spheres denser
than the sausage does not a priori indicate that such a full-dimensional
packing also exists for any particular $n'>n$. Gandini and Wills
(1992) \cite{GandiniWills1992} constructed three-dimensional packings
in $\mathbb{R}^{3}$ for $n\in\left\{ 56,59,60,61,62\right\} $ and
$n\geq65$ that are denser than the corresponding sausages, and not
long afterwards, Scholl (2000) \cite{Scholl2000} showed that three-dimensional
packings are optimal for $n\in\left\{ 58,63,64\right\} $. Gandini
and Wills also conjectured that the sausage is the best possible packing
for $n=57$ and all $n<56$. 

In dimension $4$, Gandini and Zucco (1992) \cite{GandiniZucco1992}
constructed of a four-dimensional packing with $375,\!769$ spheres
that is denser than $S_{375,769}^{4}$, and also stated that $n_{4}^{*}<367,\!300$.
To our knowledge we are unaware of any specific upper bound for $N_{4}^{*}$. 
\noindent \begin{center}
\begin{table}[H]
\noindent \begin{centering}
\begin{tabular}{cr@{}lr@{}l}
\toprule 
Dimension & \multicolumn{2}{c}{$n_{d}^{*}$} & \multicolumn{2}{c}{$N_{d}^{*}$}\tabularnewline
\midrule
\midrule 
$2$ &  & $n_{2}^{*}=3$ &  & $N_{2}^{*}=3$\tabularnewline
\midrule 
\multirow{2}{*}{$3$} & $4\leq\:$ & $n_{3}^{*}\leq56$ & $56\leq\:$ & $N_{3}^{*}\leq58$\tabularnewline
 &  & $n_{3}^{*}=56$ (conj.) &  & $N_{3}^{*}=58$ (conj.)\tabularnewline
\midrule 
\multirow{2}{*}{$4$} & $5\leq\:$ & $n_{4}^{*}<367,\!300$ & $5\leq\:$ & $N_{4}^{*}<\infty$\tabularnewline
 & $300,\!000<\:$ & $n_{4}^{*}<367,\!300$ (conj.) &  & \tabularnewline
\midrule 
$5\leq d\leq41$ & \multicolumn{4}{c}{The sausage is conjectured to be optimal}\tabularnewline
\midrule 
$42\leq d$ & \multicolumn{4}{c}{The sausage is known to be optimal}\tabularnewline
\bottomrule
\end{tabular}
\par\end{centering}
\caption{The best known and conjectured lower and upper bounds for $n_{d}^{*}$
and $N_{d}^{*}$. }
\end{table}
\par\end{center}

\subsection{\protect\label{subsec: Examples of dense finite packings in 3 and 4 dimensions}Dense
finite packings in $3$ and $4$ dimensions }

In this section we briefly summarize the packings used to show some
upper bounds for $n_{3}^{*}$ and $n_{4}^{*}$. To prove the upper
bound of $n_{3}^{*}\leq56$, Wills (1985) \cite{Wills1985} constructed
a packing of $56$ spheres by truncating the four vertices of a large
tetrahedron and intersecting the result with $D_{3}$. Gandini and
Zucco \cite{GandiniZucco1992} used the same basic ideas as Wills
to obtain the upper bound $n_{4}^{*}\leq375,\!769$. The first major
difference is the choice of polytope; they defined a sequence of $24$-cells
$\left(Y_{m}\right)_{m\in\mathbb{N}}$ by\nomenclature[C03 Y_m]{$Y_{m}$}{Gandini and Zucco's sequence of $24$-cells}
\begin{equation}
Y_{m}:=m\conv\left\{ \begin{pmatrix}\pm2\\
\hphantom{+}0\\
\hphantom{+}0\\
\hphantom{+}0
\end{pmatrix},\begin{pmatrix}\hphantom{+}0\\
\pm2\\
\hphantom{+}0\\
\hphantom{+}0
\end{pmatrix},\begin{pmatrix}\hphantom{+}0\\
\hphantom{+}0\\
\pm2\\
\hphantom{+}0
\end{pmatrix},\begin{pmatrix}\hphantom{+}0\\
\hphantom{+}0\\
\hphantom{+}0\\
\pm2
\end{pmatrix},\begin{pmatrix}\pm1\\
\pm1\\
\pm1\\
\pm1
\end{pmatrix}\right\} \label{eq: Y_m vectors}
\end{equation}
 for all $m\in\mathbb{N}$, where all ``$\pm$'' signs are independent.
In this paper we reserve $m$ for the index of a sequence of polytopes
or packings. A diagram of $Y_{1}$ and the coordinates of all of its
vertices can be found in Figure \ref{fig: Three disjoint facets of the 24-cell}
at the start of Section \ref{sec: Truncation of three facets}. Gandini
and Zucco used the scaling and orientation of the $D_{4}$ lattice
as generated by the vectors inside the braces of (\ref{eq: Y_m vectors}),
which is a packing set of $B^{4}$ and contains all the vertices of
each $Y_{m}$. The use of the $24$-cell for a high-density finite
packing is a natural choice due to an appropriately scaled and oriented
$24$-cell being the Dirichlet-Voronoi cell of the $D_{4}$ lattice
(see \cite{ConwaySloane1999}, Chapter 21, Subsection 3.2). They calculated
\begin{align}
G\left(Y_{m}\right) & =4m^{4}+8m^{3}+8m^{2}+4m+1,\label{eq: Gandini and Zucco G(Y_m)}\\
\vol\left(Y_{m}+B^{4}\right) & =32m^{4}+64\sqrt{2}m^{3}+16\sqrt{3}\pi m^{2}+192\left(\arccos\left(\frac{1}{3}\right)-\frac{\pi}{3}\right)m+\frac{1}{2}\pi^{2};\label{eq: Gandini and Zucco vol(Y_m + B^4)}
\end{align}
 the latter using Steiner's formula. Then $\vol\left(Y_{m}+B^{4}\right)<\vol\left(\conv\left(S_{G\left(Y_{m}\right)}^{4}\right)+B^{4}\right)$
for the first time when $m=17$ and $G\left(Y_{m}\right)=375,\!769$,
hence $n_{4}^{*}\leq375,\!769$. Gandini and Zucco obtained the upper
bound $n_{4}^{*}<367,\!300$ via ``suitable truncations of the $24$-cell,''
along with a conjectural lower bound of $n_{4}^{*}>300,\!000$ derived
from the sequence $\left(Y_{m}\right)$. However, in their paper they
explained neither the nature of these truncations nor the exact number
of points which they obtained for the upper bound. 

In this paper we prove new upper bounds for $n_{4}^{*}$ and $N_{4}^{*}$: 
\begin{thm}
\label{thm: MAIN THEOREM: n*_4 upper bound}$n_{4}^{*}\leq338,\!196$. 
\end{thm}

\begin{thm}
\label{thm: MAIN THEOREM: N*_4 upper bound}$N_{4}^{*}\leq516,\!946$. 
\end{thm}

The remainder of this paper is organized into three main sections.
In Section \ref{sec: The truncation of a single facet} we introduce
the precise definitions and constructions for a large family of packings
obtained by truncating facets of $Y_{m}$ in the same way that Wills
(1985) \cite{Wills1985} truncated edges. The result is a collection
of polytopes $t_{h}^{3}\left(Y_{m}\right)$, $h\in\left\{ 0,\ldots,\left\lfloor \frac{m-1}{2}\right\rfloor \right\} $
(this notation is defined in Subsection \ref{subsec: Steiner polynomial for the truncation of a single facet})
which are truncations of a single facet from the polytopes $Y_{m}$
of Gandini and Zucco's sequence, and we obtain analogues of (\ref{eq: Gandini and Zucco G(Y_m)})
and (\ref{eq: Gandini and Zucco vol(Y_m + B^4)}) for $t_{h}^{3}\left(Y_{m}\right)$.
In Section \ref{sec: Truncation of three facets}, we truncate three
facets from $Y_{m}$ to obtain a greater supply of polytopes $t_{\mathbf{h}}^{3}\left(Y_{m}\right)$.
We then compute the number of points $G\left(t_{\mathbf{h}}^{3}\left(Y_{m}\right)\right)$
of the packing and $\vol\left(t_{\mathbf{h}}^{3}\left(Y_{m}\right)+B^{d}\right)$,
the latter using Steiner's formula, and prove that $n_{3}^{*}\leq338,\!224$
via a particular packing of the form $t_{\mathbf{h}}^{3}\left(Y_{17}\right)\cap D_{4}$.
At the end of this section we show that it is possible to remove $28$
points from this packing while remaining denser than the sausage,
resulting in Theorem \ref{thm: MAIN THEOREM: n*_4 upper bound}. Finally,
Section \ref{sec: N*_4 upper bound} provides a more careful examination
of the packings $t_{\mathbf{h}}^{3}\left(Y_{m}\right)$ to obtain
a proof of Theorem \ref{thm: MAIN THEOREM: N*_4 upper bound} using
the same basic ideas as Gandini and Wills (1992) \cite{GandiniWills1992}. 

\section{\protect\label{sec: The truncation of a single facet}The truncation
of a single facet }

First we introduce notation for the $24$ facets of the $24$-cell
$Y_{m}$, $m\in\mathbb{N}$. 
\begin{defn}
\label{def: Facets of Y_m}Let $m\in\mathbb{N}$ and denote the facets
of $Y_{m}$ by $X_{m}^{1},\ldots,X_{m}^{24}$\nomenclature[D01 X_{m}^{i}]{$X_{m}^{1},\ldots,X_{m}^{24}$}{The facets of the $24$-cell $Y_{m}$};
we number the specific facets 
\begin{alignat}{7}
X_{m}^{1} & :=m\conv\left\{ \vphantom{\begin{pmatrix}2\\
0\\
0\\
0
\end{pmatrix}}\right.\! & \begin{pmatrix}2\\
0\\
0\\
0
\end{pmatrix} & \,,\, & \begin{pmatrix}1\\
1\\
1\\
1
\end{pmatrix} & \,,\, & \begin{pmatrix}\hphantom{+}1\\
\hphantom{+}1\\
\hphantom{+}1\\
-1
\end{pmatrix} & \,,\, & \begin{pmatrix}\hphantom{+}1\\
\hphantom{+}1\\
-1\\
\hphantom{+}1
\end{pmatrix} & \,,\, & \begin{pmatrix}\hphantom{+}1\\
\hphantom{+}1\\
-1\\
-1
\end{pmatrix} & \,,\, & \begin{pmatrix}0\\
2\\
0\\
0
\end{pmatrix} & \!\left.\vphantom{\begin{pmatrix}2\\
0\\
0\\
0
\end{pmatrix}}\right\} ,\label{eq: X_m^1}\\
X_{m}^{2} & :=m\conv\left\{ \vphantom{\begin{pmatrix}2\\
0\\
0\\
0
\end{pmatrix}}\right.\! & \begin{pmatrix}\hphantom{+}1\\
\hphantom{+}1\\
-1\\
\hphantom{+}1
\end{pmatrix} & \,,\, & \begin{pmatrix}1\\
1\\
1\\
1
\end{pmatrix} & \,,\, & \begin{pmatrix}2\\
0\\
0\\
0
\end{pmatrix} & \,,\, & \begin{pmatrix}\hphantom{+}1\\
-1\\
-1\\
\hphantom{+}1
\end{pmatrix} & \,,\, & \begin{pmatrix}0\\
0\\
0\\
2
\end{pmatrix} & \,,\, & \begin{pmatrix}\hphantom{+}1\\
-1\\
\hphantom{+}1\\
\hphantom{+}1
\end{pmatrix} & \!\left.\vphantom{\begin{pmatrix}2\\
0\\
0\\
0
\end{pmatrix}}\right\} ,\label{eq: X_m^2}\\
X_{m}^{3} & :=m\conv\left\{ \vphantom{\begin{pmatrix}2\\
0\\
0\\
0
\end{pmatrix}}\right.\! & \begin{pmatrix}\hphantom{+}1\\
-1\\
\hphantom{+}1\\
-1
\end{pmatrix} & \,,\, & \begin{pmatrix}0\\
0\\
2\\
0
\end{pmatrix} & \,,\, & \begin{pmatrix}\hphantom{+}1\\
-1\\
\hphantom{+}1\\
\hphantom{+}1
\end{pmatrix} & \,,\, & \begin{pmatrix}2\\
0\\
0\\
0
\end{pmatrix} & \,,\, & \begin{pmatrix}\hphantom{+}1\\
\hphantom{+}1\\
\hphantom{+}1\\
-1
\end{pmatrix} & \,,\, & \begin{pmatrix}1\\
1\\
1\\
1
\end{pmatrix} & \!\left.\vphantom{\begin{pmatrix}2\\
0\\
0\\
0
\end{pmatrix}}\right\} ,\label{eq: X_m^3}\\
X_{m}^{10} & :=m\conv\left\{ \vphantom{\begin{pmatrix}2\\
0\\
0\\
0
\end{pmatrix}}\right.\! & \begin{pmatrix}\hphantom{+}0\\
-2\\
\hphantom{+}0\\
\hphantom{+}0
\end{pmatrix} & \,,\, & \begin{pmatrix}\hphantom{+}1\\
-1\\
\hphantom{+}1\\
-1
\end{pmatrix} & \,,\, & \begin{pmatrix}\hphantom{+}1\\
-1\\
-1\\
-1
\end{pmatrix} & \,,\, & \begin{pmatrix}\hphantom{+}1\\
-1\\
-1\\
\hphantom{+}1
\end{pmatrix} & \,,\, & \begin{pmatrix}\hphantom{+}1\\
-1\\
\hphantom{+}1\\
\hphantom{+}1
\end{pmatrix} & \,,\, & \begin{pmatrix}2\\
0\\
0\\
0
\end{pmatrix} & \!\left.\vphantom{\begin{pmatrix}2\\
0\\
0\\
0
\end{pmatrix}}\right\} ,\label{eq: X_m^10}\\
X_{m}^{16} & :=m\conv\left\{ \vphantom{\begin{pmatrix}2\\
0\\
0\\
0
\end{pmatrix}}\right.\! & \begin{pmatrix}-2\\
\hphantom{+}0\\
\hphantom{+}0\\
\hphantom{+}0
\end{pmatrix} & \,,\, & \begin{pmatrix}-1\\
\hphantom{+}1\\
-1\\
\hphantom{+}1
\end{pmatrix} & \,,\, & \begin{pmatrix}-1\\
\hphantom{+}1\\
-1\\
-1
\end{pmatrix} & \,,\, & \begin{pmatrix}-1\\
-1\\
-1\\
\hphantom{+}1
\end{pmatrix} & \,,\, & \begin{pmatrix}-1\\
-1\\
-1\\
-1
\end{pmatrix} & \,,\, & \begin{pmatrix}\hphantom{+}0\\
\hphantom{+}0\\
-2\\
\hphantom{+}0
\end{pmatrix} & \!\left.\vphantom{\begin{pmatrix}2\\
0\\
0\\
0
\end{pmatrix}}\right\} ,\label{eq: X_m^16}\\
X_{m}^{17} & :=m\conv\left\{ \vphantom{\begin{pmatrix}2\\
0\\
0\\
0
\end{pmatrix}}\right.\! & \begin{pmatrix}\hphantom{+}0\\
\hphantom{+}0\\
\hphantom{+}2\\
\hphantom{+}0
\end{pmatrix} & \,,\, & \begin{pmatrix}\hphantom{+}1\\
-1\\
\hphantom{+}1\\
\hphantom{+}1
\end{pmatrix} & \,,\, & \begin{pmatrix}\hphantom{+}1\\
-1\\
\hphantom{+}1\\
-1
\end{pmatrix} & \,,\, & \begin{pmatrix}-1\\
-1\\
\hphantom{+}1\\
-1
\end{pmatrix} & \,,\, & \begin{pmatrix}-1\\
-1\\
\hphantom{+}1\\
\hphantom{+}1
\end{pmatrix} & \,,\, & \begin{pmatrix}\hphantom{+}0\\
-2\\
\hphantom{+}0\\
\hphantom{+}0
\end{pmatrix} & \!\left.\vphantom{\begin{pmatrix}2\\
0\\
0\\
0
\end{pmatrix}}\right\} ,\label{eq: X_m^17}
\end{alignat}
 and also in Definition \ref{def: Truncation of the facets of Y_m};
the remaining facets may be numbered arbitrarily. For each $i\in\left\{ 1,\ldots,24\right\} $,
let $\mathbf{u}^{i}$\nomenclature[D02 u^i]{$\mathbf{u}^{i}$}{The unit outward normal vector to $X_{m}^{i}$}
be the unit outward normal to $X_{m}^{i}$, and for $\lambda\geq0$
let $H_{i}\left(\lambda\right):=H\left(\mathbf{u}^{i},\sqrt{2}\lambda\right)$\nomenclature[D03 H_i-(lambda)]{$H_{i}^{-}\left(\lambda\right)$}{The half-space $H_{i}^{-}\left(\lambda\right)=H^{-}\left(\mathbf{u}^{i},\sqrt{2}\lambda\right)$ whose boundary is orthogonal to the unit normal $\mathbf{u}^{i}$}
and $H_{i}^{-}\left(\lambda\right):=H^{-}\left(\mathbf{u}^{i},\sqrt{2}\lambda\right)$\nomenclature[D04 H_i(lambda)]{$H_{i}\left(\lambda\right)$}{The hyperplane $H_{i}\left(\lambda\right)=H\left(\mathbf{u}^{i},\sqrt{2}\lambda\right)$ that is orthogonal to the unit normal $\mathbf{u}^{i}$}. 
\end{defn}

Note that $H_{i}\left(m\right)$ is a supporting hyperplane of $Y_{m}$
and $Y_{m}=\bigcap_{i=1}^{24}H_{i}^{-}\left(m\right)$. For any $i\in\left\{ 1,\ldots,24\right\} $,
$\mathbf{u}^{i}$ is $\frac{1}{m\sqrt{2}}$ times the centroid of
$X_{m}^{i}$ (see \cite{Coxeter1947}, page 292, Table I (ii)), $X_{m}^{i}\subset H_{i}\left(m\right)$,
and $X_{m}=\bigcap_{i=1}^{24}H_{i}^{-}\left(\lambda\right)$. For
the rest of this section we will assume that $m\in\mathbb{N}$ and
$h\in\left\{ 0,\ldots,m\right\} $ unless otherwise indicated. The
following definition makes precise our usage of facet truncation when
it comes to the $24$-cell. Initially we truncate a single facet of
$Y_{m}$. Due to the symmetry of the $24$-cell, the specific choice
of facet is irrelevant, so we may choose the regular octahedron $X_{m}^{1}$
of edge length $2m$. Color-coded physical models of various sections
of the $24$-cell can be viewed at \cite{24CellOrtiz2018}. 
\begin{defn}
\label{def: Truncation of a single facet from X_m}Let $m\in\mathbb{N}$.
Define the polytope obtained from the single-facet truncation of $Y_{m}$
by $t_{h}^{3}\left(Y_{m}\right):=Y_{m}\cap H_{1}^{-}\left(m-h\right)$\nomenclature[D05 t_{h}^{3}\left(Y_{m}\right)]{$t_{h}^{3}\left(Y_{m}\right)$}{The truncation of $h$ layers of $D_{4}$ from a single facet of $Y_{m}$}
and denote its facets $X_{m,h}^{1},\ldots,X_{m,h}^{24}$\nomenclature[D06 X_{m, h}^{i}]{$X_{m,h}^{1},\ldots,X_{m,h}^{24}$}{The facets of the truncated $24$-cell $t_{-h}^{3}\left(Y_{m}\right)$}
by 
\[
X_{m,h}^{i}:=\begin{cases}
Y_{m}\cap H_{1}\left(m-h\right) & i=1\\
X_{m}^{i}\cap H_{i}^{-}\left(m-h\right) & i\in\left\{ 2,\ldots,24\right\} 
\end{cases}.
\]
\end{defn}

Since $\dist\left(\mathbf{0},X_{m}^{1}\right)=\sqrt{2}m$ (cf. the
hyperplanes mentioned in \cite{GandiniZucco1992}), the truncation
$t_{h}^{3}$ removes a $\sqrt{2}h$-thick ``slice'' containing the
facet $X_{m}^{1}$ of $Y_{m}$. The polytope $t_{h}^{3}\left(Y_{m}\right)$
is a truncation of a single facet $X_{m}^{1}$ of $Y_{m}$, where
$h$ controls the ``amount'' of truncation, and for all $i\in\left\{ 1,2,\ldots,24\right\} $,
the facet $X_{m,h}^{i}\subset t_{h}^{3}\left(Y_{m}\right)$ is parallel
to the corresponding facet $X_{m}^{i}\subset Y_{m}$. The packing
density of $t_{h}^{3}\left(Y_{m}\right)\cap D_{4}$ is (see (\ref{def: Finite and infinite packing densities}))
\[
\delta\left(t_{h}^{3}\left(Y_{m}\right)\cap D_{4}\right)=\frac{G\left(t_{h}^{3}\left(Y_{m}\right)\right)}{\vol\left(\conv\left(t_{h}^{3}\left(Y_{m}\right)\cap D_{4}\right)+B^{4}\right)}.
\]
 Tedious but elementary calculations (see the Appendix---Subsection
\ref{subsec: The vertices of t_=00007Bh=00007D^=00007B3=00007D(Y_m) are at points of D_4})
show that the vertices of $t_{h}^{3}\left(Y_{m}\right)$ are located
at points of $D_{4}$ for all $m\in\mathbb{N}$ and $h\in\left\{ 0,\ldots,m\right\} $,
which implies that $\conv\left(t_{h}^{3}\left(Y_{m}\right)\cap D_{4}\right)=t_{h}^{3}\left(Y_{m}\right)$
for all $m\in\mathbb{N}$ and $h\in\left\{ 0,\ldots,m\right\} $,
therefore 
\[
\delta\left(t_{h}^{3}\left(Y_{m}\right)\cap D_{4}\right)=\frac{G\left(t_{h}^{3}\left(Y_{m}\right)\right)}{\vol\left(t_{h}^{3}\left(Y_{m}\right)+B^{4}\right)}.
\]
In Subsections \ref{subsec: The number of points in the truncated 24-cell}
and \ref{subsec: Steiner polynomial for the truncation of a single facet}
we obtain formulas for $G\left(t_{h}^{3}\left(Y_{m}\right)\right)$
and $\vol\left(t_{h}^{3}\left(Y_{m}\right)+B^{4}\right)$ respectively
in terms of only $m$ and $h$; the latter using Steiner's formula. 

\subsection{\protect\label{subsec: Steiner polynomial for the truncation of a single facet}The
Steiner polynomial for the single-facet truncation }

In this section we state the basic properties of $t_{h}^{3}$ that
are necessary to compute the exact values of $G\left(t_{h}^{3}\left(Y_{m}\right)\right)$
and $\vol\left(t_{h}^{3}\left(Y_{m}\right)+B^{4}\right)$. 
\begin{lem}
\label{lem: G((t_h)^3(Y_m)), one facet} Let $m\in\mathbb{N}$ and
$h\in\left\{ 0,\ldots,m\right\} $. Then 
\begin{multline}
G\left(t_{\mathbf{h}}^{3}\left(Y_{m}\right)\right)=4m^{4}+\left(8-\frac{2h}{3}\right)m^{3}+\left(8-\left(h^{2}+h\right)\right)m^{2}\\
\,+\left(4-\frac{2h^{3}+3h^{2}+2h}{3}\right)m+\left(1-\frac{-2h^{4}+2h^{3}+5h^{2}+h}{6}\right)\label{eq: G((t_h)^3(Y_m)), one facet}
\end{multline}
\end{lem}

\begin{proof}
See Subsection \ref{subsec: The number of points in the truncated 24-cell}. 
\end{proof}
For the volume calculation, we know from Steiner's formula that 
\begin{equation}
\vol\left(t_{h}^{3}\left(Y_{m}\right)+B^{4}\right)=\vol\left(t_{h}^{3}\left(Y_{m}\right)\right)+c\left(t_{h}^{3}\left(Y_{m}\right)\right)+f\left(t_{h}^{3}\left(Y_{m}\right)\right)+e\left(t_{h}^{3}\left(Y_{m}\right)\right)+\frac{1}{2}\pi^{2},\label{eq: Steiner's formula for the single facet truncation}
\end{equation}
 where $\vol\left(t_{h}^{3}\left(Y_{m}\right)\right)$, $c\left(t_{h}^{3}\left(Y_{m}\right)\right)$,
$f\left(t_{h}^{3}\left(Y_{m}\right)\right)$, and $e\left(t_{h}^{3}\left(Y_{m}\right)\right)$
are the $4$-, $3$-, $2$-, and $1$-dimensional volume components
of (\ref{eq: Steiner's formula for a convex polytope}) respectively. 
\begin{lem}
\label{lem: Components of the Steiner polynomial}Let $m\in\mathbb{N}$
and $h\in\left\{ 0,\ldots,m\right\} $, then 
\begin{align}
\vol\left(t_{h}^{3}\left(Y_{m}\right)\right) & =\vol\left(Y_{m}\right)-\frac{4}{3}\left(\left(m+h\right)^{4}-m^{4}-3h^{4}\right),\label{eq: Steiner polynomial for (t_h)^3(Y_m) + B^4, 4-dim volume part}\\
c\left(t_{h}^{3}\left(Y_{m}\right)\right) & =\frac{8\sqrt{2}}{3}\left(25m^{3}-\left(m+h\right)^{3}\right),\label{eq: Steiner polynomial for (t_h)^3(Y_m) + B^4, facet part}\\
f\left(t_{h}^{3}\left(Y_{m}\right)\right) & =\frac{2\sqrt{3}\pi}{3}\left(24m^{2}-2mh+\left(3\sqrt{3}-7\right)h^{2}\right),\label{eq: Steiner polynomial for (t_h)^3(Y_m) + B^4, face part}\\
e\left(t_{h}^{3}\left(Y_{m}\right)\right) & =\left(64m-24h\right)\left(3\arccos\left(\frac{1}{3}\right)-\pi\right)+64h\arctan\left(3-2\sqrt{2}\right).\label{eq: Steiner polynomial for (t_h)^3(Y_m) + B^4, edge part}
\end{align}
\end{lem}

\begin{proof}
See Subsubsections \ref{subsec: 4-dimensional volume}, \ref{subsec: 3-dimensional facets of the truncated 24-cell},
\ref{subsec: 2-dimensional faces of the truncated 24-cell}, and \ref{subsec: 1-dimensional edges of the truncated 24-cell}
respectively. 
\end{proof}
\begin{lem}
\label{lem: Steiner's formula (polynomial) for (t_h)^3(Y_m) + B^4, one facet}
Let $m\in\mathbb{N}$ and $h\in\left\{ 0,\ldots,m\right\} $, then
\begin{eqnarray}
\vol\left(t_{h}^{3}\left(Y_{m}\right)+B^{4}\right) & = & 32m^{4}+\left[64\sqrt{2}-\frac{16}{3}h\right]m^{3}+\left[16\sqrt{3}\pi-8\sqrt{2}h-8h^{2}\right]m^{2}\nonumber \\
 &  & \,+\left[64\left(3\arccos\left(\frac{1}{3}\right)-\pi\right)-\frac{4\sqrt{3}\pi}{3}h-8\sqrt{2}h^{2}-\frac{16}{3}h^{3}\right]m\nonumber \\
 &  & \,+\,\frac{1}{2}\pi^{2}+\left[\left(64\arctan\left(3-2\sqrt{2}\right)-24\left(3\arccos\left(\frac{1}{3}\right)-\pi\right)\right)h\vphantom{\frac{\left(18-14\sqrt{3}\right)\pi}{3}}\right.\nonumber \\
 &  & \left.\hphantom{\,+\,\frac{1}{2}\pi^{2}}\qquad+\frac{\left(18-14\sqrt{3}\right)\pi}{3}h^{2}-\frac{8\sqrt{2}}{3}h^{3}+\frac{8}{3}h^{4}\right].\label{eq: Steiner polynomial for (t_h)^3(Y_m) + B^4 in terms of m and h}
\end{eqnarray}
\end{lem}

\begin{proof}
This result follows from (\ref{eq: Steiner's formula for the single facet truncation})
and Lemma \ref{lem: Components of the Steiner polynomial}. 
\end{proof}

\subsection{\protect\label{subsec: The number of points in the truncated 24-cell}The
number of points in $t_{h}^{3}\left(Y_{m}\right)$ }

We will obtain the quantity $G\left(Y_{m}\right)-G\left(t_{h}^{3}\left(Y_{m}\right)\right)$
by counting the points of $D_{4}$ in each truncated octahedron $X_{m,k}^{1}$
for $k\in\left\{ 0,\ldots,h-1\right\} $ and summing them up. 
\begin{prop}
\label{prop: Formula for the number of points in the truncated 24-cell}Let
$m\in\mathbb{N}$ and $h\in\left\{ 0,\ldots,m\right\} $. Then 
\begin{equation}
G\left(Y_{m}\right)-G\left(t_{h}^{3}\left(Y_{m}\right)\right)=\sum_{k=0}^{h-1}G\left(X_{m,k}^{1}\right).\label{eq: Slice formula for the truncated-off portion of Y_m}
\end{equation}
\end{prop}

\begin{proof}
This method of adding up the $h$ distinct layers $X_{m,k}^{1}$ neither
omits nor double counts any points of $D_{4}$ in $Y_{m}\!\left\backslash t_{h}^{3}\left(Y_{m}\right)\right.$.
To see this fact, note that the distance between $X_{m,k}^{1}$ and
its adjacent layer $X_{m,k+1}^{1}$ is $\dist\left(X_{m,k}^{1},X_{m,k+1}^{1}\right)=\sqrt{2}$
for any $k\in\left\{ 0,\ldots,h-1\right\} $, which is equal to the
distance between adjacent hyperplanes containing translates of $D_{3}$
in $D_{4}$. These quantities are equal, so the sets $X_{m,k}^{1}$,
$k\in\left\{ 0,\ldots,h-1\right\} $, coincide with individual layers
of $D_{4}$, from which (\ref{eq: Slice formula for the truncated-off portion of Y_m})
follows. 
\end{proof}
In Section \ref{sec: Truncation of three facets} it will be convenient
to express $G\left(Y_{m}\right)-G\left(t_{h}^{3}\left(Y_{m}\right)\right)$
as a polynomial in $m$ whose coefficients are themselves polynomials
in $h$. Let $\widehat{X}_{m,h}^{1}$\nomenclature[D07 X_hat_{m,h}^{1}]{$\widehat{X}_{m,h}^{1}$}{The regular octahedron whose faces have the same centroids as the faces of $X_{m,h}^{1}$}
be the regular octahedron whose faces have the same centroids as the
faces of $X_{m,h}^{1}$, in other words, $X_{m,h}^{1}$ is obtained
from $\widehat{X}_{m,h}^{1}$ via vertex truncations that remove six
square pyramids of appropriate height from $\widehat{X}_{m,h}^{1}$.
Basic calculations along with the linearity of the edge length of
$\widehat{X}_{m,h}^{1}$ in the variable $h$ show that $\widehat{X}_{m,h}^{1}$
has edge length $2\left(m+h\right)$. 
\begin{proof}[Proof of (\ref{eq: G((t_h)^3(Y_m)), one facet}) in Lemma \ref{lem: G((t_h)^3(Y_m)), one facet}]
 Let $k\in\left\{ 0,\ldots,h-1\right\} $. The intersection $\widehat{X}_{m,k}^{1}\cap D_{3}$
consists of a middle layer with $\left(\left(m+k\right)+1\right)^{2}$
points and for each $i\in\left\{ 1,\ldots,m+k\right\} $, two layers
with $i^{2}$ points, one above the middle layer and one below, so
\begin{equation}
G\left(\widehat{X}_{m,k}^{1}\right)=\left(m+k+1\right)^{2}+2\sum_{i=1}^{m+k}i^{2}=\frac{2\left(m+k\right)^{3}+6\left(m+k\right)^{2}+7\left(m+k\right)+3}{3}.\label{eq: Number of points in the big octahedron}
\end{equation}
 Due to the vertex truncations, $G\left(X_{m,k}^{1}\right)=G\left(\widehat{X}_{m,k}^{1}\right)-6\sum_{i=1}^{k}i^{2}$,
so from this equation, (\ref{eq: Slice formula for the truncated-off portion of Y_m}),
and (\ref{eq: Number of points in the big octahedron}) we conclude
that 
\begin{multline*}
G\left(Y_{m}\right)-G\left(t_{h}^{3}\left(Y_{m}\right)\right)=\sum_{k=0}^{h-1}\left(\frac{2\left(m+k\right)^{3}+6\left(m+k\right)^{2}+7\left(m+k\right)+3}{3}-6\sum_{i=1}^{k}i^{2}\right)\\
=\frac{2h}{3}m^{3}+\left(h^{2}+h\right)m^{2}+\frac{2h^{3}+3h^{2}+2h}{3}m+\frac{-2h^{4}+2h^{3}+5h^{2}+h}{6}.
\end{multline*}
\end{proof}

\subsection{\protect\label{subsec: Vertices, edges, faces, and facets of (t_-h)^3(Y_m)}The
vertices, edges, faces, and facets of $t_{h}^{3}\left(Y_{m}\right)$ }

In this subsection we prove the four equations (\ref{eq: Steiner polynomial for (t_h)^3(Y_m) + B^4, 4-dim volume part}),
(\ref{eq: Steiner polynomial for (t_h)^3(Y_m) + B^4, facet part}),
(\ref{eq: Steiner polynomial for (t_h)^3(Y_m) + B^4, face part}),
and (\ref{eq: Steiner polynomial for (t_h)^3(Y_m) + B^4, edge part})
of Lemma \ref{lem: Components of the Steiner polynomial}. For each
$i\in\left\{ 1,2,3\right\} $, we classify all the $i$-faces of $t_{h}^{3}\left(Y_{m}\right)$
into different types based on their interactions with the half-space
$H_{1}^{-}\left(m-h\right)$, then find the $i$-volume and $\left(4-i\right)$-dimensional
external angle of each $i$-face. Due to the symmetries of $Y_{m}$
and $t_{h}^{3}\left(Y_{m}\right)$, all faces within a given type
have the same volume and external angle, so it is not necessary to
explicitly write down all of these faces. Instead, for each type we
will provide a description of a single ``representative face'' with
which the calculations will be done. Diagrams of these faces are shown
in the Appendix, Figures \ref{fig: Truncated facets of t_=00007Bh=00007D^=00007B3=00007D(Y_m)}
and \ref{fig: Truncated faces and edges of t_=00007Bh=00007D^=00007B3=00007D(Y_m)}. 

\subsubsection{\protect\label{subsec: 4-dimensional volume}The $4$-volume of $Y_{m}\!\left\backslash t_{h}^{3}\left(Y_{m}\right)\right.$ }

Let $\widehat{Y}_{m,h}:=Y_{m}\!\left\backslash t_{h}^{3}\left(Y_{m}\right)\right.$\nomenclature[D08 Y_hat_{m,h}]{$\widehat{Y}_{m,h}$}{The remainder of the $24$-cell $Y_{m}$ after removing the truncated $24$-cell $t_{h}^{3}\left(Y_{m}\right)$ from it}.
In this subsubsection we compute the four-dimensional volume $\vol\left(\widehat{Y}_{m,h}\right)$,
then we can obtain $\vol\left(t_{h}^{3}\left(Y_{m}\right)\right)=\vol\left(Y_{m}\right)-\vol\left(\widehat{Y}_{m,-h}\right)$. 
\begin{proof}[Proof of (\ref{eq: Steiner polynomial for (t_h)^3(Y_m) + B^4, facet part})
in Lemma \ref{lem: Components of the Steiner polynomial}]
 We calculate the ($4$-dimensional) volume of $\widehat{Y}_{m,h}$
by finding the $3$-dimensional volume of each cross-section $Y_{m}\cap H_{1}\left(\lambda\right)\subset\widehat{Y}_{m,h}$
and computing the integral 
\begin{equation}
\vol\left(\widehat{Y}_{m,h}\right)=\sqrt{2}\int_{m-h}^{m}\vol_{3}\left(Y_{m}\cap H_{1}\left(\lambda\right)\right)\,d\lambda.\label{eq: Integral formula for the truncated volume of X_(m, h)}
\end{equation}
 (The factor of $\sqrt{2}$ occurs because $\dist\left(\mathbf{0},H_{m}\left(\lambda\right)\right)=\sqrt{2}\lambda$.)
We wish to find the integrand $\vol_{3}\left(Y_{m}\cap H_{1}\left(\lambda\right)\right)$.
The intersection $Y_{m}\cap H_{1}\left(\lambda\right)$ is a truncated
octahedron for all $\lambda\in\left(0,m\right)$ and is contained
inside $\widehat{X}_{m,h}^{1}$, which has edge length $2\left(m+h\right)=2\left(2m-\lambda\right)$.
Its volume can be found by subtracting the volume of six square pyramids
with edge lengths $2\left(m-\lambda\right)$ from $\vol_{3}\left(\widehat{X}_{m,h}^{1}\right)$,
so 
\begin{equation}
\vol_{3}\left(Y_{m}\cap H_{1}\left(\lambda\right)\right)=\frac{8\sqrt{2}}{3}\left(\left(m+h\right)^{3}-3h^{3}\right).\label{eq: 3-volume of truncated facet}
\end{equation}
 Using this expression, evaluating the integral (\ref{eq: Integral formula for the truncated volume of X_(m, h)})
gives 
\[
\vol\left(\widehat{Y}_{m,h}\right)=\sqrt{2}\int_{m-h}^{m}\vol_{3}\left(Y_{m}\cap H_{m}\left(\lambda\right)\right)\,d\lambda=\frac{4}{3}\left(\left(m+h\right)^{4}-m^{4}-3h^{4}\right),
\]
 from which (\ref{eq: Steiner polynomial for (t_h)^3(Y_m) + B^4, 4-dim volume part})
follows. 
\end{proof}

\subsubsection{\protect\label{subsec: 3-dimensional facets of the truncated 24-cell}The
$3$-dimensional facets }

We wish to find the three-dimensional volume $\vol_{3}\left(t_{h}^{3}\left(Y_{m}\right)\right)$
obtained from the facets of $t_{h}^{3}\left(Y_{m}\right)$, and to
do so we investigate the facets of $t_{h}^{3}\left(Y_{m}\right)$. 
\begin{defn}
\label{def: Truncation of the facets of Y_m}Order the facets $X_{m}^{1},\ldots,X_{m}^{24}$
of $Y_{m}$ such that $X_{m}^{2},\ldots,X_{m}^{9}$ share a face with
$X_{m}^{1}$ and $X_{m}^{10},\ldots,X_{m}^{15}$ share a vertex with
but not a face of $X_{m}^{1}$. For $h\in\left\{ 0,\ldots,m\right\} $,
define $t_{h}^{2}$ and $t_{h}^{0}$ on these facets by $t_{h}^{2}\left(X_{m}^{i}\right):=X_{m}^{i}\cap H_{1}^{-}\left(m-h\right)$\nomenclature[D09 t_{h}^{2}\left(X_{m}^{i}\right)]{$t_{h}^{2}\left(X_{m}^{i}\right)$}{The face truncation of the facet $X_{m}^{i}$ of $Y_{m}$, where $i\in\left\{ 2,\ldots,9\right\}$}
for $i\in\left\{ 2,\ldots,9\right\} $ and $t_{h}^{0}\left(X_{m}^{k}\right):=X_{m}^{k}\cap H_{1}^{-}\left(m-h\right)$\nomenclature[D10 t_{h}^{0}\left(X_{m}^{i}\right)]{$t_{h}^{0}\left(X_{m}^{k}\right)$}{The vertex truncation of the facet $X_{m}^{i}$ of $Y_{m}$, where $i\in\left\{ 10,\ldots,15\right\}$}
for $k\in\left\{ 10,\ldots,15\right\} $. 
\end{defn}

Note that $X_{m}^{2}$, $X_{m}^{3}$, and $X_{m}^{10}$ were already
defined in Definition \ref{def: Facets of Y_m}, but we can choose
an ordering of the facets so that these two definitions are consistent.
As this notation implies, the adjacent facets $t_{h}^{2}\left(X_{m}^{i}\right)$
and $t_{h}^{0}\left(X_{m}^{k}\right)$ of $Y_{m}$ are obtained by
truncating one face and one vertex of the octahedron respectively.
When $m\in\mathbb{N}$ and $0<h<m$ the facets $X_{m}^{1},\ldots,X_{m}^{24}$
of $Y_{m}$, and the corresponding facets $X_{m,h}^{i}$ of $t_{h}^{3}\left(Y_{m}\right)$,
can be categorized into four types based on their intersection with
the half-space $H_{1}^{-}\left(m-h\right)$. 
\begin{enumerate}
\item The intersection $X_{m}^{1}\cap H_{1}^{-}\left(m-h\right)$ is the
empty set, so $X_{m}^{1}\subset Y_{m}$ corresponds to a single new
facet $X_{m,h}^{1}\subset t_{h}^{3}\left(Y_{m}\right)$ that is parallel
to $X_{m}^{1}$. Since $X_{m,h}^{1}$ is a cross-section parallel
to a facet of $Y_{m}$, it is a truncated octahedron (see \cite{BooleStott1900,PoloBlanco2014,24CellOrtiz2018})
with six square faces of edge length $2h$ and eight hexagonal faces
with edge lengths alternating between $2h$ and $2\left(m-h\right)$. 
\item Eight facets of $Y_{m}$, which after reordering we refer to as $X_{m}^{2},\ldots,X_{m}^{9}$,
are truncated at the face where the facet intersects $H_{1}\left(m-h\right)$.
The original facet $X_{m}^{i}$ shares three of its six vertices with
its corresponding face-truncated facet $X_{m}^{i}\cap H_{1}^{-}\left(m-h\right)\subset t_{h}^{3}\left(Y_{m}\right)$.
Hence the facets $t_{h}^{2}\left(X_{m}^{2}\right),\ldots,t_{h}^{2}\left(X_{m}^{9}\right)$
of $t_{h}^{3}\left(Y_{m}\right)$ are face-truncations of the octahedron
and are rigid motions of the representative facet 
\begin{eqnarray*}
t_{h}^{2}\left(X_{m}^{\hexagon}\right) & = & \conv\left\{ \begin{pmatrix}2m-h\\
-h\\
\hphantom{+}h\\
\hphantom{+}h
\end{pmatrix},\begin{pmatrix}2m-h\\
-h\\
\hphantom{+}h\\
-h
\end{pmatrix},\begin{pmatrix}\hphantom{+}m\\
m-2h\\
\hphantom{+}m\\
-m
\end{pmatrix},\begin{pmatrix}\hphantom{+}m-h\\
\hphantom{+}m-h\\
\hphantom{+}m+h\\
-m+h
\end{pmatrix},\right.\\
 &  & \hphantom{\conv\left\{ \vphantom{\begin{pmatrix}2m-h\\
-h\\
\hphantom{+}h\\
\hphantom{+}h
\end{pmatrix}}\right.}\!\!\left.\begin{pmatrix}m-h\\
m-h\\
m+h\\
m-h
\end{pmatrix},\begin{pmatrix}m\\
m-2h\\
m\\
m
\end{pmatrix},\begin{pmatrix}\hphantom{+}m\\
-m\\
\hphantom{+}m\\
\hphantom{+}m
\end{pmatrix},\begin{pmatrix}0\\
0\\
2m\\
0
\end{pmatrix},\begin{pmatrix}\hphantom{+}m\\
-m\\
\hphantom{+}m\\
-m
\end{pmatrix}\right\} ,
\end{eqnarray*}
 where\nomenclature[D11 X_{m}^{\hexagon}]{$X_{m}^{\hexagon}$}{The facet of $Y_{m}$ which is face-truncated to form the representative facet $t_{h}^{2}\left(X_{m}^{\hexagon}\right)$}
\[
X_{m}^{\hexagon}:=\conv\left\{ \begin{pmatrix}2m\\
0\\
0\\
0
\end{pmatrix},\begin{pmatrix}m\\
m\\
m\\
m
\end{pmatrix},\begin{pmatrix}\hphantom{+}m\\
\hphantom{+}m\\
\hphantom{+}m\\
-m
\end{pmatrix},\begin{pmatrix}\hphantom{+}m\\
-m\\
\hphantom{+}m\\
\hphantom{+}m
\end{pmatrix},\begin{pmatrix}0\\
0\\
2m\\
0
\end{pmatrix},\begin{pmatrix}\hphantom{+}m\\
-m\\
\hphantom{+}m\\
-m
\end{pmatrix}\right\} .
\]
\item Six facets of $Y_{m}$, which after reordering we refer to as $X_{m}^{10},\ldots,X_{m}^{15}$,
are truncated at the vertex where the facet intersects $H_{1}\left(m-h\right)$.
The original facet $X_{m}^{k}$ shares five of its six vertices with
its corresponding vertex-truncated facet $X_{m}^{k}\cap H_{1}^{-}\left(m-h\right)\subset t_{h}^{3}\left(Y_{m}\right)$.
So the facets $t_{h}^{0}\left(X_{m}^{10}\right),\ldots,t_{h}^{0}\left(X_{m}^{15}\right)$
of $t_{h}^{3}\left(Y_{m}\right)$ are vertex-truncations of the octahedron
and are rigid motions of the representative facet 
\begin{eqnarray*}
t_{h}^{0}\left(X_{m}^{\square}\right) & = & \conv\left\{ \begin{pmatrix}2m-h\\
-h\\
\hphantom{+}h\\
\hphantom{+}h
\end{pmatrix},\begin{pmatrix}2m-h\\
-h\\
\hphantom{+}h\\
-h
\end{pmatrix},\begin{pmatrix}2m-h\\
-h\\
-h\\
-h
\end{pmatrix},\begin{pmatrix}2m-h\\
-h\\
-h\\
\hphantom{+}h
\end{pmatrix},\right.\\
 &  & \hphantom{\conv\left\{ \vphantom{\begin{pmatrix}2m-h\\
-h\\
\hphantom{+}h\\
\hphantom{+}h
\end{pmatrix}}\right.}\!\!\left.\begin{pmatrix}\hphantom{+}1\\
-1\\
\hphantom{+}1\\
\hphantom{+}1
\end{pmatrix},\begin{pmatrix}\hphantom{+}1\\
-1\\
\hphantom{+}1\\
-1
\end{pmatrix},\begin{pmatrix}\hphantom{+}1\\
-1\\
-1\\
-1
\end{pmatrix},\begin{pmatrix}\hphantom{+}1\\
-1\\
-1\\
\hphantom{+}1
\end{pmatrix},\begin{pmatrix}\hphantom{+}0\\
-2\\
\hphantom{+}0\\
\hphantom{+}0
\end{pmatrix}\right\} ,
\end{eqnarray*}
 where\nomenclature[D12 X_{m}^{\square}]{$X_{m}^{\square}$}{The facet of $Y_{m}$ which is vertex-truncated to form the representative facet $t_{h}^{0}\left(X_{m}^{\square}\right)$}
\[
X_{m}^{\square}:=\conv\left\{ \begin{pmatrix}2\\
0\\
0\\
0
\end{pmatrix},\begin{pmatrix}\hphantom{+}1\\
-1\\
\hphantom{+}1\\
\hphantom{+}1
\end{pmatrix},\begin{pmatrix}\hphantom{+}1\\
-1\\
\hphantom{+}1\\
-1
\end{pmatrix},\begin{pmatrix}\hphantom{+}1\\
-1\\
-1\\
-1
\end{pmatrix},\begin{pmatrix}\hphantom{+}1\\
-1\\
-1\\
\hphantom{+}1
\end{pmatrix},\begin{pmatrix}\hphantom{+}0\\
-2\\
\hphantom{+}0\\
\hphantom{+}0
\end{pmatrix}\right\} .
\]
\item The remaining nine facets $X_{m}^{16},\ldots,X_{m}^{24}\subset Y_{m}$
are unmodified by the facet truncation because they lie entirely in
the closed half-space $H_{1}^{-}\left(m-h\right)$, hence $X_{m}^{l}=X_{m}^{l}\cap H_{1}^{-}\left(m-h\right)\subset t_{h}^{3}\left(Y_{m}\right)$. 
\end{enumerate}
To show (\ref{eq: Steiner polynomial for (t_h)^3(Y_m) + B^4, facet part})
we establish the following statements. The external angle $\theta\left(Y_{m},F\right)$
of any facet $F$ is $\frac{1}{2}$, so we do not need to perform
external angle computations for these facets. 
\begin{lem}
\label{lem: 3-volumes of the facets of the truncated 24-cell}Let
$m\in\mathbb{N}$ and $h\in\left\{ 0,\ldots,m\right\} $. The four
kinds of representative facets of $t_{h}^{3}\left(Y_{m}\right)$ described
above have the following $3$-volumes: 
\begin{align}
\vol_{3}\left(X_{m,h}^{1}\right) & =\frac{8\sqrt{2}}{3}\left(\left(m+h\right)^{3}-3h^{3}\right),\label{eq: 3-volume of X_=00007Bm,h=00007D^=00007B1=00007D}\\
\vol_{3}\left(t_{h}^{2}\left(X_{m}^{i}\right)\right) & =\frac{2\sqrt{2}}{3}\left(5m^{3}+3h^{3}-\left(m+h\right)^{3}\right),\label{eq: 3-volume of face-truncated octahedron}\\
\vol_{3}\left(t_{h}^{0}\left(X_{m}^{k}\right)\right) & =\frac{4\sqrt{2}}{3}\left(2m^{3}-h^{3}\right),\label{eq: 3-volume of vertex-truncated octahedron}\\
\vol_{3}\left(X_{m}^{l}\right) & =\frac{8\sqrt{2}}{3}m^{3},\label{eq: 3-volume of octahedron}
\end{align}
 where $i\in\left\{ 2,\ldots,9\right\} $, $k\in\left\{ 10,\ldots,15\right\} $,
and $l\in\left\{ 16,\ldots,24\right\} $. 
\end{lem}

\begin{proof}
Equation (\ref{eq: 3-volume of octahedron}) follows directly from
the formula for the volume of a regular octahedron and (\ref{eq: 3-volume of X_=00007Bm,h=00007D^=00007B1=00007D})
was already obtained as (\ref{eq: 3-volume of truncated facet}) in
the previous subsubsection. The cross-section of $X_{m}^{i}$ parallel
to $H^{i}\left(\lambda\right)$ is $X_{m}^{i}\cap H^{i}\left(\lambda\right)$,
a large equilateral triangle of edge length $2\left(2m-\lambda\right)$
with its vertices truncated via removing an equilateral triangle of
edge length $2\left(m-\lambda\right)$ from each vertex. So $\area\left(X_{m}^{i}\cap H^{i}\left(\lambda\right)\right)=\sqrt{3}\left(2m-\lambda\right)^{2}-3\sqrt{3}\left(m-\lambda\right)^{2}$
and 
\begin{eqnarray*}
\vol_{3}\left(t_{h}^{2}\left(X_{m}^{i}\right)\right) & = & \frac{2\sqrt{6}}{3}\int_{0}^{m-h}\area\left(X_{m}^{i}\cap H^{i}\left(\lambda\right)\right)\,d\lambda\\
 & = & \frac{2\sqrt{2}}{3}\left(5m^{3}+3h^{3}-\left(m+h\right)^{3}\right),
\end{eqnarray*}
 showing (\ref{eq: 3-volume of face-truncated octahedron}). (The
distance between two opposite faces of $X_{m}^{i}$ is $\frac{2\sqrt{6}}{3}m$.)
For (\ref{eq: 3-volume of vertex-truncated octahedron}), let $\mathbf{v}_{0}=X_{m}^{i}\cap Y_{m}$,
$\mathbf{v}$ be any of the four neighboring vertices of $X_{m}^{i}$,
and let $E=\conv\left\{ \mathbf{v}_{0},\mathbf{v}\right\} $ be the
edge of $Y_{m}$ containing these vertices. The hyperplane $H^{1}\left(m-h\right)$
intersects $E$ at $\frac{h}{m}$ of the distance from $\mathbf{v}_{0}$
to $\mathbf{v}$. Since $E$ has length $2m$, $X_{m}^{i}\!\left\backslash \left(t_{h}^{0}\left(X_{m}^{i}\right)\right)\right.$
is a square pyramid of edge length $2h$. So 
\begin{eqnarray*}
\vol_{3}\left(t_{h}^{0}\left(X_{m}^{i}\right)\right) & = & \vol_{3}\left(X_{m}^{i}\right)-\vol_{3}\left(X_{m}^{i}\left\backslash \,t_{h}^{0}\left(X_{m}^{i}\right)\right.\right)\\
 & = & \frac{4\sqrt{2}}{3}\left(2m^{3}-h^{3}\right).
\end{eqnarray*}
\end{proof}
\begin{proof}[Proof of (\ref{eq: Steiner polynomial for (t_h)^3(Y_m) + B^4, facet part})
in Lemma \ref{lem: Components of the Steiner polynomial}]
 The discussion at the beginning of this subsubsection, along with
Lemma \ref{lem: 3-volumes of the facets of the truncated 24-cell},
imply that 
\begin{eqnarray*}
c\left(t_{h}^{3}\left(Y_{m}\right)\right) & = & \sum_{i=1}^{24}\vol_{3}\left(X_{m,h}^{i}\right)\\
 & = & 1\cdot\vol_{3}\left(X_{m,h}^{1}\right)+8\cdot\vol_{3}\left(t_{h}^{2}\left(X_{m}^{i}\right)\right)+6\cdot\vol_{3}\left(t_{h}^{0}\left(X_{m}^{k}\right)\right)+9\cdot\vol_{3}\left(X_{m}^{l}\right)\\
 & = & \frac{8\sqrt{2}}{3}\left(25m^{3}-\left(m+h\right)^{3}\right).
\end{eqnarray*}
\end{proof}

\subsubsection{\protect\label{subsec: 2-dimensional faces of the truncated 24-cell}The
$2$-dimensional faces }

Denote the $96$ triangular faces of $Y_{m}$ by $T_{m}^{1},\ldots,T_{m}^{96}$\nomenclature[D13 T_{m}^{1},\ldots,T_{m}^{96}]{$T_{m}^{1},\ldots,T_{m}^{96}$}{The faces of $Y_{m}$}.
The $102$ faces of $t_{h}^{3}\left(Y_{m}\right)$ can also be categorized
into five types based on their relationship to $H_{1}^{-}\left(m-h\right)$. 
\begin{defn}
Let $T_{m}^{i}$ and $T_{m}^{j}$ be faces of $Y_{m}$ such that $T_{m}^{j}\cap X_{m}^{1}$
is an edge and $T_{m}^{k}\cap X_{m}^{1}$ is a vertex. Then define
$t_{h}^{1}\left(T_{m}^{j}\right):=T_{m}^{j}\cap H_{i}^{-}\left(m-h\right)$\nomenclature[D14 t_{h}^{1}\left(T_{m}^{j}\right)]{$t_{h}^{1}\left(T_{m}^{j}\right)$}{The edge truncation of a face of $Y_{m}$, $t_{h}^{0}\left(E_{m}\right)=E_{m}\cap H_{1}^{-}\left(m-h\right)$, where $T_{m}^{j}\cap X_{m}^{1}$ is an edge}
and $t_{h}^{0}\left(T_{m}^{k}\right):=T_{m}^{k}\cap H_{i}^{-}\left(m-h\right)$\nomenclature[D15 t_{h}^{0}\left(T_{m}^{k}\right)]{$t_{h}^{0}\left(T_{m}^{k}\right)$}{The vertex truncation of a face of $Y_{m}$, $t_{h}^{0}\left(E_{m}\right)=E_{m}\cap H_{1}^{-}\left(m-h\right)$, where $T_{m}^{k}\cap X_{m}^{1}$ is a vertex}. 
\end{defn}

$t_{h}^{1}\left(T_{m}^{i}\right)$ is an edge truncation of the triangular
face $T_{m}^{i}$, hence is just a smaller triangle with the same
shape. In contrast, $t_{h}^{0}\left(T_{m}^{j}\right)$ is a vertex
truncation of $T_{m}^{j}$ and is an isosceles trapezium. As in the
previous subsubsection, the faces can be reordered as follows. 
\begin{enumerate}
\item $6$ square faces $Q_{m,h}^{1},\ldots,Q_{m,h}^{6}$\nomenclature[D16 Q_{m,h}^{i}]{$Q_{m,h}^{1},\ldots,Q_{m,h}^{6}$}{The square faces of $t_{h}^{3}\left(Y_{m}\right)$}
which are rigid motions of\nomenclature[D17 Q_{m,h}]{$Q_{m,h}$}{The representative face for the square faces $Q_{m,h}^{1},\ldots,Q_{m,h}^{6}$ of $t_{h}^{3}\left(Y_{m}\right)$}
\[
Q_{m,h}:=\conv\left\{ \begin{pmatrix}2m-h\\
-h\\
\hphantom{+}h\\
\hphantom{+}h
\end{pmatrix},\begin{pmatrix}2m-h\\
-h\\
\hphantom{+}h\\
-h
\end{pmatrix},\begin{pmatrix}2m-h\\
-h\\
-h\\
-h
\end{pmatrix},\begin{pmatrix}2m-h\\
-h\\
-h\\
\hphantom{+}h
\end{pmatrix}\right\} .
\]
 These faces are new and do not correspond to any face of $Y_{m}$. 
\item $8$ hexagonal faces $H_{m,h}^{1},\ldots,H_{m,h}^{8}$\nomenclature[D18 H_{m,h}^{i}]{$H_{m,h}^{1},\ldots,H_{m,h}^{8}$}{The hexagonal faces of $t_{h}^{3}\left(Y_{m}\right)$}
which are rigid motions of\nomenclature[D19 H_{m,h}]{$H_{m,h}$}{The representative face for the hexagonal faces $H_{m,h}^{1},\ldots,H_{m,h}^{8}$ of $t_{h}^{3}\left(Y_{m}\right)$}
\[
H_{m,h}:=\conv\left\{ \begin{pmatrix}2m-h\\
-h\\
-h\\
\hphantom{+}h
\end{pmatrix},\begin{pmatrix}2m-h\\
-h\\
\hphantom{+}h\\
\hphantom{+}h
\end{pmatrix},\begin{pmatrix}m\\
m-2h\\
m\\
m
\end{pmatrix},\begin{pmatrix}m-h\\
m-h\\
m-h\\
m+h
\end{pmatrix},\begin{pmatrix}\hphantom{+}m-h\\
\hphantom{+}m-h\\
-m+h\\
\hphantom{+}m+h
\end{pmatrix},\begin{pmatrix}\hphantom{+}m\\
m-2h\\
-m\\
\hphantom{+}m
\end{pmatrix}\right\} .
\]
\item $12$ triangular faces $t_{h}^{1}\left(T_{m}^{9}\right),\ldots,t_{h}^{1}\left(T_{m}^{20}\right)$
which are rigid motions of 
\[
t_{h}^{1}\left(T_{m}^{\triangledown}\right)=\conv\left\{ \begin{pmatrix}2m-h\\
-h\\
\hphantom{+}h\\
\hphantom{+}h
\end{pmatrix},\begin{pmatrix}\hphantom{+}m\\
-m\\
\hphantom{+}m\\
\hphantom{+}m
\end{pmatrix},\begin{pmatrix}m\\
m-2h\\
m\\
m
\end{pmatrix}\right\} ,
\]
 where\nomenclature[D20 T_{m}^{\triangledown}]{$T_{m}^{\triangledown}$}{The face of $Y_{m}$ which is edge-truncated to form the representative facet $t_{h}^{1}\left(T_{m}^{\triangledown}\right)$}
\[
T_{m}^{\triangledown}:=\conv\left\{ \begin{pmatrix}2m\\
0\\
0\\
0
\end{pmatrix},\begin{pmatrix}\hphantom{+}m\\
-m\\
\hphantom{+}m\\
\hphantom{+}m
\end{pmatrix},\begin{pmatrix}m\\
m\\
m\\
m
\end{pmatrix}\right\} .
\]
\item $24$ isosceles trapezium faces $t_{h}^{0}\left(T_{m}^{21}\right),\ldots,t_{h}^{0}\left(T_{m}^{44}\right)$
which are rigid motions of 
\[
t_{h}^{0}\left(T_{m}^{\triangle}\right)=\conv\left\{ \begin{pmatrix}2m-h\\
-h\\
\hphantom{+}h\\
-h
\end{pmatrix},\begin{pmatrix}2m-h\\
-h\\
\hphantom{+}h\\
\hphantom{+}h
\end{pmatrix},\begin{pmatrix}\hphantom{+}m\\
-m\\
\hphantom{+}m\\
\hphantom{+}m
\end{pmatrix},\begin{pmatrix}\hphantom{+}m\\
-m\\
\hphantom{+}m\\
-m
\end{pmatrix}\right\} ,
\]
 where\nomenclature[D21 T_{m}^{\triangle}]{$T_{m}^{\triangle}$}{The face of $Y_{m}$ which is vertex-truncated to form the representative facet $t_{h}^{0}\left(T_{m}^{\triangle}\right)$}
\[
T_{m}^{\triangle}:=\conv\left\{ \begin{pmatrix}2m\\
0\\
0\\
0
\end{pmatrix},\begin{pmatrix}\hphantom{+}m\\
-m\\
\hphantom{+}m\\
\hphantom{+}m
\end{pmatrix},\begin{pmatrix}\hphantom{+}m\\
-m\\
\hphantom{+}m\\
-m
\end{pmatrix}\right\} .
\]
Each of these faces shares a single vertex with $X_{m,h}^{1}$. 
\item $52$ triangular faces $T_{m}^{45},\ldots,T_{m}^{96}$ which are unchanged
from the corresponding faces in $Y_{m}$. 
\end{enumerate}
\begin{lem}
\label{lem: Areas and external angles of the faces of the truncated 24-cell}The
representative faces $Q_{m,h}$, $H_{m,h}$, $t_{h}^{1}\left(T_{m}^{\triangledown}\right)$,
$t_{h}^{0}\left(T_{m}^{\triangle}\right)$, and $T_{m}$ have the
following areas and external angles: 
\begin{align*}
\area\left(Q_{m,h}\right) & =4h^{2}, & \theta\left(t_{h}^{3}\left(Y_{m}\right),Q_{m,h}\right) & =\frac{1}{4},\\
\area\left(H_{m,h}\right) & =\sqrt{3}\left(m+h\right)^{2}-3\sqrt{3}h^{2}, & \theta\left(t_{h}^{3}\left(Y_{m}\right),H_{m,h}\right) & =\frac{1}{6},\\
\area\left(t_{h}^{1}\left(T_{m}^{\triangledown}\right)\right) & =\sqrt{3}\left(m-h\right)^{2}, & \theta\left(t_{h}^{3}\left(Y_{m}\right),t_{h}^{1}\left(T_{m}^{\triangledown}\right)\right) & =\frac{1}{6},\\
\area\left(t_{h}^{0}\left(T_{m}^{\triangle}\right)\right) & =\sqrt{3}m^{2}-\sqrt{3}h^{2}, & \theta\left(t_{h}^{3}\left(Y_{m}\right),t_{h}^{0}\left(T_{m}^{\triangle}\right)\right) & =\frac{1}{6},\\
\area\left(T_{m}\right) & =\sqrt{3}m^{2}, & \theta\left(t_{h}^{3}\left(Y_{m}\right),T_{m}\right) & =\frac{1}{6}.
\end{align*}
\end{lem}

\begin{proof}
The equations for the areas follow from elementary calculations. As
for the external angles, the hexagonal face $H_{m,h}$ is parallel
to the face $\conv\left\{ \left(2m,0,0,0\right)^{\mathsf{T}},\left(m,m,m,m\right)^{\mathsf{T}},\left(m,m,-m,m\right)^{\mathsf{T}}\right\} $
of $Y_{m}$, which is the intersection of the two facets $X_{m}^{1}$
and $X_{m}^{2}$. Their outward unit normals $\mathbf{u}^{1}$ and
$\mathbf{u}^{2}$ meet at the angle of $\arccos\left(\mathbf{u}^{1}\cdot\mathbf{u}^{2}\right)=\frac{\pi}{3}$,
so $\theta\left(t_{h}^{3}\left(Y_{m}\right),H_{m,h}\right)=\frac{\arccos\left(\mathbf{u}^{1}\cdot\mathbf{u}^{2}\right)}{2\pi}=\frac{1}{6}$.
Similarly, $t_{h}^{1}\left(T_{m}^{\triangledown}\right)$, $t_{h}^{0}\left(T_{m}^{\triangle}\right)$,
and $T_{m}$ are also parallel to faces of $Y_{m}$, so almost identical
calculations result in $\theta\left(t_{h}^{3}\left(Y_{m}\right),t_{h}^{1}\left(T_{m}^{\triangledown}\right)\right)=\theta\left(t_{h}^{3}\left(Y_{m}\right),t_{h}^{0}\left(T_{m}^{\triangle}\right)\right)=\theta\left(t_{h}^{3}\left(Y_{m}\right),T_{m}\right)=\frac{1}{6}$.
Finally, the square face $Q_{m,h}$ is the intersection of $X_{m,h}^{1}$
and the vertex-truncation of $X_{m}^{10}$, which has outward unit
normal $\mathbf{u}^{10}$, so $\theta\left(t_{h}^{3}\left(Y_{m}\right),Q_{m,h}\right)=\frac{\arccos\left(\mathbf{u}^{1}\cdot\mathbf{u}^{10}\right)}{2\pi}=\frac{1}{4}$. 
\end{proof}
\begin{proof}[Proof of (\ref{eq: Steiner polynomial for (t_h)^3(Y_m) + B^4, face part})
in Lemma \ref{lem: Components of the Steiner polynomial}]
 The discussion at the beginning of this subsubsection, along with
Lemma \ref{lem: Areas and external angles of the faces of the truncated 24-cell},
imply that 
\begin{eqnarray*}
f\left(t_{h}^{3}\left(Y_{m}\right)\right) & = & \sum_{F\in\mathscr{F}_{2}\left(t_{h}^{3}\left(Y_{m}\right)\right)}\area\left(F\right)\theta\left(t_{h}^{3}\left(Y_{m}\right),F\right)\kappa_{2}\\
 & = & 6\area\left(Q_{m,h}\right)\theta\left(t_{h}^{3}\left(Y_{m}\right),Q_{m,h}\right)\pi+8\area\left(H_{m,h}\right)\theta\left(t_{h}^{3}\left(Y_{m}\right),H_{m,h}\right)\pi\\
 &  & \,+\,12\area\left(t_{h}^{1}\left(T_{m}^{\triangledown}\right)\right)\theta\left(t_{h}^{3}\left(Y_{m}\right),t_{h}^{1}\left(T_{m}^{\triangledown}\right)\right)\pi\\
 &  & \,+\,24\area\left(t_{h}^{0}\left(T_{m}^{\triangle}\right)\right)\theta\left(t_{h}^{3}\left(Y_{m}\right),t_{h}^{0}\left(T_{m}^{\triangle}\right)\right)\pi+52\area\left(T_{m}\right)\theta\left(t_{h}^{3}\left(Y_{m}\right),T_{m}\right)\pi\\
 & = & \frac{2\sqrt{3}\pi}{3}\left(24m^{2}-2mh+\left(3\sqrt{3}-7\right)h^{2}\right).
\end{eqnarray*}
\end{proof}

\subsubsection{\protect\label{subsec: 1-dimensional edges of the truncated 24-cell}The
$1$-dimensional edges }

The $24$-cell $Y_{m}$ has $96$ edges $E_{m}^{1},\ldots,E_{m}^{96}$\nomenclature[D22 E_{m}^{1},\ldots,E_{m}^{96}]{$E_{m}^{1},\ldots,E_{m}^{96}$}{The edges of $Y_{m}$}
but the truncation shrinks some edges and adds four new edges near
each vertex of the facet $X_{m,h}^{1}$, for a total of $120$ edges
of $t_{h}^{3}\left(Y_{m}\right)$. 
\begin{defn}
\label{def: Truncated edges of the 24-cell}Let $E_{m}$ be an edge
of $Y_{m}$ with one vertex in $X_{m}^{1}$ and the other vertex not
in $X_{m}^{1}$. Then define $t_{h}^{0}\left(E_{m}\right):=E_{m}\cap H_{1}^{-}\left(m-h\right)$\nomenclature[D23 t_{h}^{0}\left(E_{m}\right)]{$t_{h}^{0}\left(E_{m}\right)$}{The vertex truncation of an edge of $Y_{m}$, $t_{h}^{0}\left(E_{m}\right)=E_{m}\cap H_{1}^{-}\left(m-h\right)$}. 
\end{defn}

The edges of $t_{h}^{3}\left(Y_{m}\right)$ can be sorted into four
types based on their intersections with $H_{1}^{-}\left(m-h\right)$.
The first two types are new edges (that do not exist in $Y_{m}$)
while the other two are subsets of existing edges in $Y_{m}$. 
\begin{enumerate}
\item $24$ new edges $E_{m,h}^{\square,1},\ldots,E_{m,h}^{\square,24}$
which are appropriate rigid motions of\nomenclature[D24 E_{m,h}^{\square}]{$E_{m,h}^{\square}$}{The representative edge for the edges $E_{m,h}^{25},\ldots,E_{m,h}^{48}$ of $t_{h}^{3}\left(Y_{m}\right)$ that border one square face of $X_{m,h}^{1}$}
\[
E_{m,h}^{\square}:=\conv\left\{ \begin{pmatrix}2m-h\\
-h\\
\hphantom{+}h\\
\hphantom{+}h
\end{pmatrix},\begin{pmatrix}2m-h\\
-h\\
\hphantom{+}h\\
-h
\end{pmatrix}\right\} .
\]
 Each of these edges borders a square face of $X_{m,h}^{1}$. 
\item $12$ edges $E_{m,h}^{\hexagon,1},\ldots,E_{m,h}^{\hexagon,12}$ which
are appropriate rigid motions of the representative edge\nomenclature[D25 E_{m,h}^{\hexagon}]{$E_{m,h}^{\hexagon}$}{The representative edge for the edges $E_{m,h}^{1},\ldots,E_{m,h}^{24}$ of $t_{h}^{3}\left(Y_{m}\right)$ that border two adjacent hexagonal faces of $X_{m,h}^{1}$}
\[
E_{m,h}^{\hexagon}:=\conv\left\{ \begin{pmatrix}2m-h\\
-h\\
\hphantom{+}h\\
\hphantom{+}h
\end{pmatrix},\begin{pmatrix}m\\
m-2h\\
m\\
m
\end{pmatrix}\right\} .
\]
 Each of these edges borders two adjacent hexagonal faces of $X_{m,h}^{1}$
but not a square face. 
\item $24$ edges $t_{h}^{0}\left(E_{m}^{13}\right),\ldots,t_{h}^{0}\left(E_{m}^{36}\right)$
which are appropriate rigid motions of\nomenclature[D26 E_{m}^{\triangle}]{$E_{m}^{\triangle}$}{The edge of $Y_{m}$ which is truncated to form the representative edge $t_{h}^{0}\left(E_{m}^{\triangle}\right)$}
\[
t_{h}^{0}\left(E_{m}^{\triangle}\right)=\conv\left\{ \begin{pmatrix}2m-h\\
-h\\
\hphantom{+}h\\
\hphantom{+}h
\end{pmatrix},\begin{pmatrix}\hphantom{+}m\\
-m\\
\hphantom{+}m\\
\hphantom{+}m
\end{pmatrix}\right\} ,\qquad\text{where}\qquad E_{m}^{\triangle}:=\conv\left\{ \begin{pmatrix}2m\\
0\\
0\\
0
\end{pmatrix},\begin{pmatrix}\hphantom{+}m\\
-m\\
\hphantom{+}m\\
\hphantom{+}m
\end{pmatrix}\right\} \text{.}
\]
\item $60$ edges $E_{m}^{37},\ldots,E_{m}^{96}$ which are unchanged from
the corresponding edges in $Y_{m}$. 
\end{enumerate}
It follows directly from the definition that their lengths are 
\begin{equation}
\ell\left(E_{m}\right)=2m,\qquad\ell\left(E_{m,h}^{\square}\right)=2h,\qquad\ell\left(E_{m,h}^{\hexagon}\right)=\ell\left(t_{h}^{0}\left(E_{m}^{\triangle}\right)\right)=2\left(m-h\right)\text{.}\label{eq: Lengths of the edges of the truncated 24-cell}
\end{equation}
 Each edge $E$ is the intersection of three facets so the intersection
of its normal cone with $S^{2}$ is a spherical triangle, and to calculate
its area we use L'Huilier's Theorem (see \cite{Todhunter1871}, page
70): 
\begin{thm}[L'Huilier's Theorem]
Let $\pos\left\{ \mathbf{v}^{1},\mathbf{v}^{2},\mathbf{v}^{3}\right\} \cap S^{2}$
be a spherical triangle, then 
\[
\tan\left(\frac{a_{1,2,3}}{4}\right)=\sqrt{\begin{array}{l}
{\displaystyle \tan\left(\frac{a_{1,2}+a_{1,3}+a_{2,3}}{4}\right)\tan\left(\frac{-a_{1,2}+a_{1,3}+a_{2,3}}{4}\right)}\\
{\displaystyle \,\cdot\tan\left(\frac{a_{1,2}-a_{1,3}+a_{2,3}}{4}\right)\tan\left(\frac{a_{1,2}+a_{1,3}-a_{2,3}}{4}\right)}
\end{array}},
\]
 where $a_{i,j}$ is the angle between $\mathbf{v}^{i}$ and $\mathbf{v}^{j}$
and $a_{1,2,3}$ is the spherical area of $\pos\left\{ \mathbf{v}^{1},\mathbf{v}^{2},\mathbf{v}^{3}\right\} \cap S^{2}$.
\end{thm}

\begin{lem}
\label{lem: External angles of the edges of the truncated 24-cell}The
representative edges $E_{m,h}^{\hexagon}$, $E_{m,h}^{\square}$,
$t_{h}^{0}\left(E_{m}^{\triangle}\right)$, and $E_{m}$ have the
following external angles: 
\begin{align}
\theta\left(t_{h}^{3}\left(Y_{m}\right),E_{m,h}^{\hexagon}\right) & =\frac{3\arccos\left(\frac{1}{3}\right)-\pi}{4\pi},\label{eq: External angle of hexagon edge}\\
\theta\left(t_{h}^{3}\left(Y_{m}\right),E_{m,h}^{\square}\right) & =\frac{4\arctan\left(3-2\sqrt{2}\right)}{4\pi},\label{eq: External angle of square edge}\\
\theta\left(t_{h}^{3}\left(Y_{m}\right),t_{h}^{0}\left(E_{m}^{\triangle}\right)\right) & =\frac{3\arccos\left(\frac{1}{3}\right)-\pi}{4\pi},\label{eq: External angle of triangle edge}\\
\theta\left(t_{h}^{3}\left(Y_{m}\right),E_{m}\right) & =\frac{3\arccos\left(\frac{1}{3}\right)-\pi}{4\pi}.\label{eq: External angle of unchanged edge}
\end{align}
\end{lem}

\begin{proof}
The spherical triangles formed by the normal cones of $E_{m,h}^{\hexagon}$,
$t_{h}^{0}\left(E_{m}^{\triangle}\right)$, and $E_{m}$ have the
same angles as the spherical triangle mentioned in \cite{GandiniZucco1992},
Part (A), from which (\ref{eq: External angle of hexagon edge}),
(\ref{eq: External angle of triangle edge}), and (\ref{eq: External angle of unchanged edge})
all follow. The spherical triangle formed by the normal cone of $E_{m,h}^{\square}$
has outward normals $\mathbf{u}^{1}$, $\mathbf{u}^{2}$, and $\mathbf{u}^{10}$,
so $\arccos\left(\mathbf{u}^{1}\cdot\mathbf{u}^{2}\right)=\arccos\left(\mathbf{u}^{2}\cdot\mathbf{u}^{10}\right)=\frac{\pi}{3}$
and $\arccos\left(\mathbf{u}^{1}\cdot\mathbf{u}^{10}\right)=\frac{\pi}{2}$,
and so L'Huilier's Theorem gives (\ref{eq: External angle of square edge}):
\begin{eqnarray*}
\area\left(N\left(t_{h}^{3}\left(Y_{m}\right),E_{m,h}^{\square}\right)\cap S^{2}\right) & = & 4\arctan\sqrt{\tan\left(\frac{7\pi}{24}\right)\tan\left(\frac{3\pi}{24}\right)\tan\left(\frac{3\pi}{24}\right)\tan\left(\frac{\pi}{24}\right)}\\
 & = & 4\arctan\left(3-2\sqrt{2}\right).
\end{eqnarray*}
\end{proof}
\begin{proof}[Proof of (\ref{eq: Steiner polynomial for (t_h)^3(Y_m) + B^4, edge part})
in Lemma \ref{lem: Components of the Steiner polynomial}]
 The discussion at the beginning of this subsubsection, along with
(\ref{eq: Lengths of the edges of the truncated 24-cell}) and Lemma
\ref{lem: External angles of the edges of the truncated 24-cell},
imply that 
\begin{eqnarray*}
e\left(t_{h}^{3}\left(Y_{m}\right)\right) & = & \sum_{E\in\mathscr{F}_{1}\left(t_{h}^{3}\left(Y_{m}\right)\right)}\ell\left(E\right)\theta\left(t_{h}^{3}\left(Y_{m}\right),E\right)\kappa_{3}\\
 & = & 12\ell\left(E_{m,h}^{\square}\right)\theta\left(t_{h}^{3}\left(Y_{m}\right),E_{m,h}^{\square}\right)\kappa_{3}+24\ell\left(E_{m,h}^{\hexagon}\right)\theta\left(t_{h}^{3}\left(Y_{m}\right),E_{m,h}^{\square}\right)\kappa_{3}\\
 &  & \,+\,24\ell\left(t_{h}^{0}\left(E_{m}^{\triangle}\right)\right)\theta\left(t_{h}^{3}\left(Y_{m}\right),t_{h}^{0}\left(E_{m}^{\triangle}\right)\right)\kappa_{3}+60\ell\left(E_{m}\right)\theta\left(t_{h}^{3}\left(Y_{m}\right),E_{m}\right)\kappa_{3}\\
 & = & \left(64m-24h\right)\left(3\arccos\left(\frac{1}{3}\right)-\pi\right)+16h\arctan\left(3-2\sqrt{2}\right).
\end{eqnarray*}
\end{proof}

\section{\protect\label{sec: Truncation of three facets}The truncation of
three facets and the upper bound $n_{4}^{*}\protect\leq338,\!196$ }

In this section we truncate three facets of $Y_{m}$. Since the $24$-cell
has three pairwise disjoint facets, in the next subsection we define
the polytope $t_{\mathbf{h}}^{3}\left(Y_{m}\right)$ obtained from
truncating three such facets $X_{m}^{1}$, $X_{m}^{16}$, and $X_{m}^{17}$
(see the start of Section \ref{sec: The truncation of a single facet})
of $Y_{m}$, and obtain the triple facet versions of Lemmas \ref{lem: G((t_h)^3(Y_m)), one facet}
and \ref{lem: Steiner's formula (polynomial) for (t_h)^3(Y_m) + B^4, one facet}. 
\begin{defn}
\label{def: Truncation of three facets from Y_m}Let $m\in\mathbb{N}$
and $\mathbf{h}=\left(h_{1},h_{16},h_{17}\right)\in\left\{ 0,\ldots,m\right\} ^{3}$,
then define\nomenclature[E1 t_{h_vector}^{3}\left(Y_{m}\right)]{$t_{\mathbf{h}}^{3}\left(Y_{m}\right)$}{The truncation of $h_{1}$, $h_{16}$, and $h_{17}$ layers of $D_{4}$ from three facets of $Y_{m}$, where $\mathbf{h}=\left(h_{1},h_{16},h_{17}\right)$}
$t_{\mathbf{h}}^{3}\left(Y_{m}\right):=Y_{m}\cap\left[H_{1}\left(m-h_{1}\right)\cap H_{16}\left(m-h_{16}\right)\cap H_{17}\left(m-h_{17}\right)\right]$. 
\end{defn}

\noindent \begin{center}
\begin{figure}[H]
\noindent \begin{centering}
\includegraphics[width=6.24in]{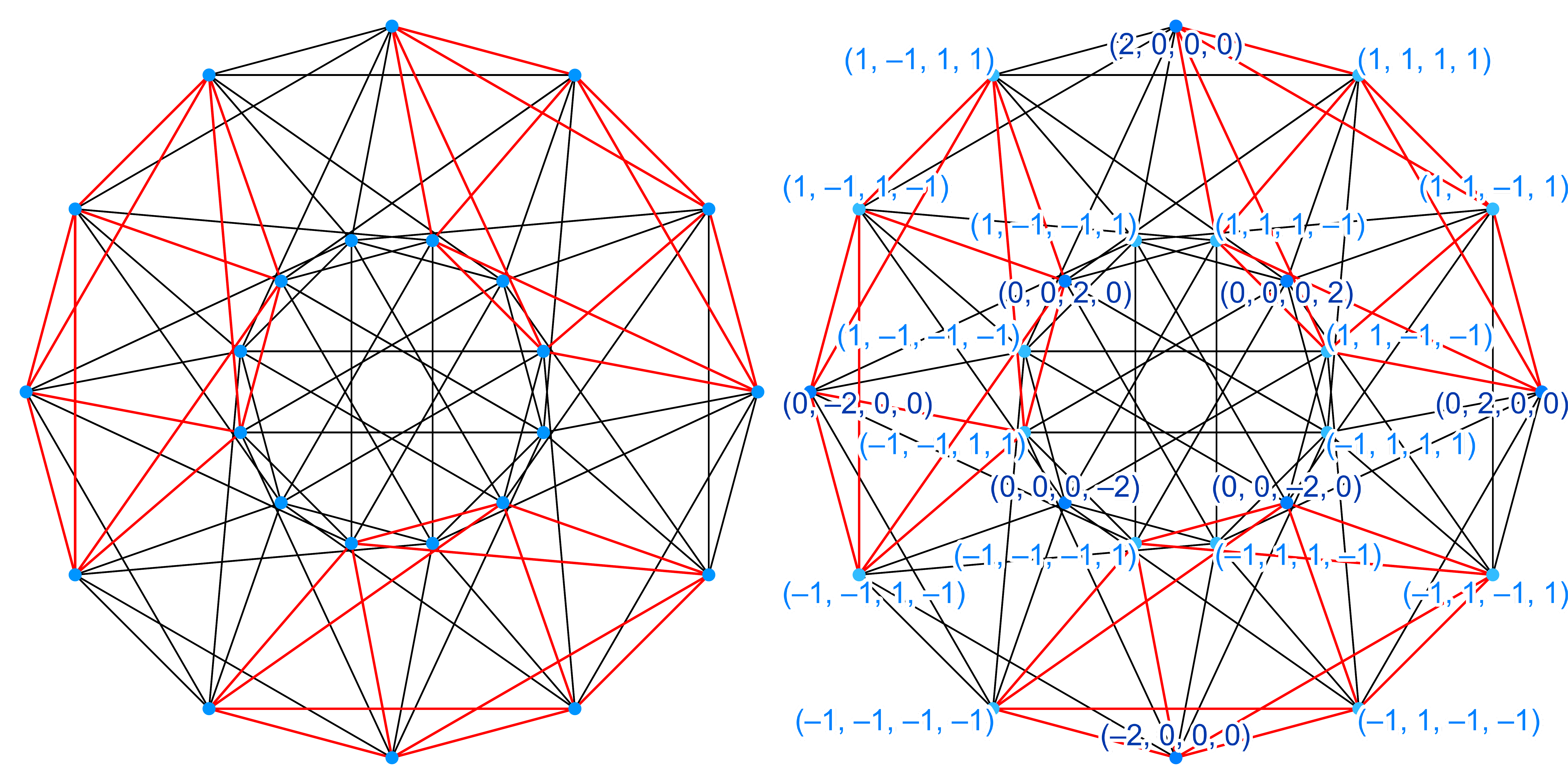}
\par\end{centering}
\caption{\protect\label{fig: Three disjoint facets of the 24-cell}Three disjoint
facets of $Y_{1}$, whose edges are marked in \emph{red}, with and
without coordinates. The $24$-cell layout of this diagram and others
in the Appendix are based on \cite{24CellMathWorld,24CellWikipedia}
and the $24$-cell visualization in the application \emph{HyperSolids}
\cite{24CellHyperSolids}. }
\end{figure}
\par\end{center}

In the same way that $h$ was used to denote the amount of truncation
of $X_{m}^{1}$, we use $\mathbf{h}=\left(h_{1},h_{16},h_{17}\right)$\nomenclature[E2 h_vector]{$\mathbf{h}$}{An ordered triple $\mathbf{h}=\left(h_{1},h_{16},h_{17}\right)$ which denotes the number of layers of $D_{4}$ which are removed by truncating the faces $X_{m}^{1}$, $X_{m}^{16}$, and $X_{m}^{17}$ of $Y_{m}$}
where $h_{i}$ denotes the amount of truncation of $X_{m}^{i}$ for
$i\in\left\{ 1,16,17\right\} $. While Wills \cite{Wills1985} was
able to truncate all four vertices of the tetrahedron to varying degrees,
we limit ourselves to three facets of $Y_{m}$ which were chosen to
be pairwise disjoint. However, when they are truncated they will move
closer to each other and possibly intersect in ways that will change
the shape of the facets. We want to avoid any intersections, otherwise
we may accidentally double count or miss some points or volumes during
our calculations. The following proposition shows that not only are
$X_{m}^{1}$, $X_{m}^{16}$, and $X_{m}^{17}$ disjoint, but their
truncated counterparts $X_{m,h_{1}}^{1}$, $X_{m,h_{16}}^{16}$, and
$X_{m,h_{17}}^{17}$ stay disjoint for all $h_{1},h_{16},h_{17}\in\left\{ 0,\ldots,\left\lfloor \frac{m-1}{2}\right\rfloor \right\} $.
\begin{prop}
\label{prop: The three truncated facets of the 24-cell are still disjoint}Let
$m\in\mathbb{N}$ and $\mathbf{h}\in\left\{ 0,\ldots,\left\lfloor \frac{m-1}{2}\right\rfloor \right\} ^{3}$.
Then the three facets $X_{m,h_{i}}^{i}$, $i\in\left\{ 1,16,17\right\} $,
are pairwise disjoint.
\end{prop}

\begin{proof}
Without loss of generality, we can assume that $i=1$, $j=17$, and
$h=h_{1}=h_{17}$ (otherwise we can set $h=\max\left\{ h_{1},h_{17}\right\} $).
We will show that if $h<\frac{m}{2}$ then the intersection $A_{m,h}:=\aff\left(X_{m,h}^{1}\right)\cap\aff\left(X_{m,h}^{17}\right)$
does not intersect $t_{\mathbf{h}}^{3}\left(Y_{m}\right)$, from which
it follows that $X_{m,h}^{1}$ and $X_{m,h}^{17}$ are disjoint. 

There exists some $\mathbf{x}^{m,h}\in A_{m,h}$ such that the quadrilateral
$\conv\left\{ \mathbf{0},\sqrt{2}\left(m-h\right)\mathbf{u}^{1},\sqrt{2}\left(m-h\right)\mathbf{u}^{17},\mathbf{x}^{m,h}\right\} \subset\lin\left\{ \mathbf{u}^{1},\mathbf{u}^{17}\right\} $
forms a kite whose internal angles are $120{^\circ}$, $90{^\circ}$,
$60{^\circ}$, and $90{^\circ}$. Consider the support hyperplane
$H_{17}\left(m-h\right)$. Since $\mathbf{x}^{m,h}=\sqrt{2}\left(m-h\right)\left(1,0,1,0\right)^{\mathsf{T}}$
is a scalar multiple of $\mathbf{u}^{17}$, the intersection $A_{m,h}$
is parallel to $H_{17}\left(m-h\right)$, hence it suffices to check
only the single point $\mathbf{x}^{m,h}\in A_{m,h}$. The distances
are $\dist\left(\mathbf{0},\mathbf{x}^{m,-k}\right)=2\sqrt{2}\left(m-h\right)$
and $\dist\left(\mathbf{0},H_{17}\left(m-h\right)\right)=\sqrt{2}m$,
so $X_{m,h}^{1}$ and $X_{m,h}^{17}$ are disjoint if $2\sqrt{2}\left(m-h'\right)>\sqrt{2}m$,
which simplifies to $h'<\frac{m}{2}$. This final inequality can be
written as $h'\leq\left\lfloor \frac{m-1}{2}\right\rfloor $ since
$m$ is an integer. 
\end{proof}
With the following lemma, we can finally write the desired quantities
$G\left(t_{\mathbf{h}}^{3}\left(Y_{m}\right)\right)$ and $\vol\left(t_{\mathbf{h}}^{3}\left(Y_{m}\right)+B^{4}\right)$
in terms of $m$ and $\mathbf{h}=\left(h_{1},h_{16},h_{17}\right)$.
For conciseness, we write $\left|\mathbf{h}\right|^{k}:=h_{1}^{k}+h_{16}^{k}+h_{17}^{k}$\nomenclature[E3 h_vector^k]{$\left|\mathbf{h}\right|^{k}$}{The sum $h_{1}^{k}+h_{16}^{k}+h_{17}^{k}$}
for $\mathbf{h}$ in Definition \ref{def: Truncation of three facets from Y_m}
and $k\in\mathbb{N}$. 
\begin{lem}
\label{lem: G((t_h)^3(Y_m)) and Steiner's formula, three facets}
Let $m\in\mathbb{N}$ and $\mathbf{h}\in\left\{ 0,\ldots,\left\lfloor \frac{m-1}{2}\right\rfloor \right\} ^{3}$.
Then 
\begin{multline}
G\left(t_{\mathbf{h}}^{3}\left(Y_{m}\right)\right)=4m^{4}+\left(8-\frac{2\left|\mathbf{h}\right|}{3}\right)m^{3}+\left(8-\left(\left|\mathbf{h}\right|^{2}+\left|\mathbf{h}\right|\right)\right)m^{2}\\
\,+\left(4-\frac{2\left|\mathbf{h}\right|^{3}+3\left|\mathbf{h}\right|^{2}+2\left|\mathbf{h}\right|}{3}\right)m+\left(1-\frac{-2\left|\mathbf{h}\right|^{4}+2\left|\mathbf{h}\right|^{3}+5\left|\mathbf{h}\right|^{2}+\left|\mathbf{h}\right|}{6}\right)\label{eq: G((t_h)^3(X_m)), the number of points in the truncated 24-cell, three facets}
\end{multline}
 and 
\begin{eqnarray}
\vol\left(t_{\mathbf{h}}^{3}\left(Y_{m}\right)+B^{4}\right) & = & 32m^{4}+\left[64\sqrt{2}-\frac{16}{3}\left|\mathbf{h}\right|\right]m^{3}+\left[16\sqrt{3}\pi-8\sqrt{2}\left|\mathbf{h}\right|-8\left|\mathbf{h}\right|^{2}\right]m^{2}\nonumber \\
 &  & \,+\left[64\left(3\arccos\left(\frac{1}{3}\right)-\pi\right)-\frac{4\sqrt{3}\pi}{3}\left|\mathbf{h}\right|-8\sqrt{2}\left|\mathbf{h}\right|^{2}-\frac{16}{3}\left|\mathbf{h}\right|^{3}\right]m\nonumber \\
 &  & \,+\,\frac{1}{2}\pi^{2}+\left[\left(64\arctan\left(3-2\sqrt{2}\right)-24\left(3\arccos\left(\frac{1}{3}\right)-\pi\right)\right)\left|\mathbf{h}\right|\vphantom{\frac{\left(18-14\sqrt{3}\right)\pi}{3}}\right.\nonumber \\
 &  & \left.\hphantom{\,+\,\frac{1}{2}\pi^{2}}\qquad+\frac{\left(18-14\sqrt{3}\right)\pi}{3}\left|\mathbf{h}\right|^{2}-\frac{8\sqrt{2}}{3}\left|\mathbf{h}\right|^{3}+\frac{8}{3}\left|\mathbf{h}\right|^{4}\right].\label{eq: Steiner polynomial for (t_-h)^3(X_m) + B^4, three facets, in terms of m and h}
\end{eqnarray}
\end{lem}

Hence the values of $G\left(t_{\mathbf{h}}^{3}\left(Y_{m}\right)\right)$
and $\vol\left(t_{\mathbf{h}}^{3}\left(Y_{m}\right)+B^{4}\right)$
are straightforward calculations using $m$ and the sums of powers
$\left|\mathbf{h}\right|^{k}$, $k\in\left\{ 1,2,3,4\right\} $. 
\begin{proof}
(\ref{eq: G((t_h)^3(X_m)), the number of points in the truncated 24-cell, three facets})
and (\ref{eq: Steiner polynomial for (t_-h)^3(X_m) + B^4, three facets, in terms of m and h})
follow directly from Lemmas \ref{lem: G((t_h)^3(Y_m)), one facet}
and \ref{lem: Steiner's formula (polynomial) for (t_h)^3(Y_m) + B^4, one facet},
Proposition \ref{prop: The three truncated facets of the 24-cell are still disjoint},
and (\ref{eq: Steiner polynomial for (t_h)^3(Y_m) + B^4 in terms of m and h}). 
\end{proof}
Note that because $\left|\mathbf{h}\right|^{k}$ is linear in each
term $h_{i}^{k}$, (\ref{eq: G((t_h)^3(X_m)), the number of points in the truncated 24-cell, three facets})
and (\ref{eq: Steiner polynomial for (t_-h)^3(X_m) + B^4, three facets, in terms of m and h})
are identical to (\ref{eq: G((t_h)^3(Y_m)), one facet}) and (\ref{eq: Steiner polynomial for (t_h)^3(Y_m) + B^4 in terms of m and h})
except for the use of $\left|\mathbf{h}\right|^{k}$ instead of $h^{k}$. 

Finally, we prove the main theorem about $n_{4}^{*}$ after a couple
of additional definitions. 
\begin{defn}
Let $\mathscr{C}$ be a set of packings in $\mathbb{R}^{d}$. Let
$C\in\mathscr{C}$ and $k\in\left\{ 0,\ldots,\left|C\right|\right\} $,
then define $p_{k}\left(C\right):=C\left\backslash \left\{ k\text{ distinct points of }C\right\} \right.$\nomenclature[E4 p_k(C)]{$p_{k}\left(C\right)$}{A packing $C$ with $k$ of its points removed. The specific points do not matter (for our purposes), only the number of points}. 
\end{defn}

This definition is strictly speaking not well-defined, however, \emph{we
will only use inequalities involving $p_{k}\left(C\right)$ that do
not depend on the specific points that are removed}. 
\begin{defn}
\label{def: Approximate density}Let $C\in\mathscr{C}$ and $n\in\mathbb{N}_{0}$,
then define\nomenclature[E5 delta tilde(C, n)]{$\widetilde{\delta}\left(C,n\right)$}{A lower bound for $\delta\left(C,n\right)$ obtained by the convex hull of the spheres of $C$}
\begin{equation}
\widetilde{\delta}\left(C,n\right):=\frac{n\kappa_{d}}{\vol\left(\conv C+B^{d}\right)}\qquad\text{for }n\in\left\{ 0,\ldots,\left|C\right|\right\} \text{.}\label{eq: Approximate (crude) density}
\end{equation}
\end{defn}

We refer to $\widetilde{\delta}$ as the \textbf{approximate (packing)
density} and write $\widetilde{\delta}\left(p_{k}\left(C\right)\right)$\nomenclature[E6 delta tilde(p_{k}(C))]{$\widetilde{\delta}\left(p_{k}\left(C\right)\right)$}{Another term for $\widetilde{\delta}\left(C,\left|C\right|-k\right)$}
to mean $\widetilde{\delta}\left(C,\left|C\right|-k\right)$. The
definition of $\widetilde{\delta}\left(p_{k}\left(C\right)\right)$
uses $\vol\left(\conv C+B^{d}\right)$ in the denominator instead
of $\vol\left(\conv\left(p_{k}\left(C\right)+B^{d}\right)\right)$,
thereby sidestepping the ambiguity in the definition of $p_{k}\left(C\right)$
at the cost of a slightly larger denominator. In particular, $\delta\left(p_{k}\left(C\right)\right)\geq\widetilde{\delta}\left(p_{k}\left(C\right)\right)$
for any packing set $C$ and $k\in\left\{ 0,\ldots,\left|C\right|\right\} $. 
\begin{proof}[Proof of Theorem \ref{thm: MAIN THEOREM: n*_4 upper bound}]
We calculate $G\left(t_{\mathbf{h}}^{3}\left(Y_{17}\right)\right)$
and $\vol\left(t_{\mathbf{h}}^{3}\left(Y_{17}\right)+B^{4}\right)$
for all triples $\mathbf{h}\in\left\{ 0,\ldots,8\right\} ^{3}$ using
(\ref{eq: G((t_h)^3(X_m)), the number of points in the truncated 24-cell, three facets})
and (\ref{eq: Steiner polynomial for (t_-h)^3(X_m) + B^4, three facets, in terms of m and h}).
The particular $\mathbf{h}$ that yields the packing with fewest spheres
that is still denser than the sausage is $\mathbf{h}=\left(1,3,4\right)$.
The corresponding packing $t_{\left(1,3,4\right)}^{3}\left(Y_{17}\right)\cap D_{4}$
has $G\left(t_{-\left(1,3,4\right)}^{3}\left(Y_{17}\right)\right)=338,\!224$
points and $\delta\left(t_{\left(1,3,4\right)}^{3}\left(Y_{m}\right)\cap D_{4}\right)\approx0.589\,106$
exceeds $\delta\left(S_{338,224}^{4}\right)\approx0.589\,049$, so
$n_{4}^{*}\leq338,\!224$. Furthermore, $n_{4}^{*}\leq338,\!224-32=338,\!192$
is obtained from 
\begin{alignat*}{3}
\widetilde{\delta}\left(p_{33}\left(t_{\left(1,3,4\right)}^{3}\left(Y_{m}\right)\cap D_{4}\right)\right) & \approx0.589\,048\,6 & \qquad & <\qquad & \delta\left(S_{338,191}^{4}\right) & \approx0.589\,049\,3,\\
\widetilde{\delta}\left(p_{32}\left(t_{\left(1,3,4\right)}^{3}\left(Y_{m}\right)\cap D_{4}\right)\right) & \approx0.589\,050\,3 & \qquad & >\qquad & \delta\left(S_{338,192}^{4}\right) & \approx0.589\,049\,3.
\end{alignat*}
\end{proof}
Figures \ref{fig: Densities of the truncated 24-cells} and \ref{fig: Densities of the truncated 24-cells, zoom in}
in the Appendix graph the densities of various packings $t_{\mathbf{h}}^{3}\left(Y_{m}\right)\cap D_{4}$
near $n=338,\!224$. 

\section{\protect\label{sec: N*_4 upper bound}The upper bound $N_{4}^{*}\protect\leq516,\!946$ }

In this section we prove Theorem \ref{thm: MAIN THEOREM: N*_4 upper bound}.
While we use the existing collection of packings $t_{\mathbf{h}}^{3}\left(Y_{m}\right)$,
an important contrast is present for $N_{4}^{*}$ compared to $n_{4}^{*}$.
To establish Theorem \ref{thm: MAIN THEOREM: n*_4 upper bound} it
sufficed to exhibit a single packing of $338,\!192$ points with density
exceeding that of the sausage, however, to establish Theorem \ref{thm: MAIN THEOREM: N*_4 upper bound}
it is required to show that \emph{for every} natural number $N\geq516,\!946$
there exists a finite packing in $\mathbb{R}^{4}$ with $N$ points
that is denser than the corresponding sausage. Gandini and Wills (1992)
\cite{GandiniWills1992} used three different sequences of three different
kinds of polyhedra and performed systems of vertex truncations on
them. With this approach they were able to show that for every $N\geq65$
and $N\in\left\{ 56,59,61,62\right\} $, an appropriate three-dimensional
packing of $N$ spheres is denser than the sausage. Similarly, we
generate a large number of packings from the truncations of $Y_{m}$
using the observation at the end of Section \ref{sec: Truncation of three facets}. 
\begin{defn}
\label{def: Covering intervals of N}Let $\mathscr{C}\subseteq\mathscr{P}_{n}^{d}$
with $n\geq3$. For any $C\in\mathscr{C}$, let\nomenclature[F1 v_tilde(C, d)]{$\widetilde{v}\left(C,d\right)$}{The difference $\widetilde{v}\left(C,d\right)=\vol\left(\conv S_{n}^{d}+B^{d}\right)-\vol\left(\conv C+B^{d}\right)$ between the volume of the convex hull of the sausage and the volume of the convex hull of the packing (is defined as $0$ if $C$ is less dense than the sausage)}\nomenclature[F2 L_tilde(C)]{$\widetilde{L}\left(C\right)$}{The set of all integers from $n-\widetilde{r}\left(C,d\right)$ to $n$, inclusive (is defined as the empty set if $C$ is less dense than the sausage)}\nomenclature[F3 L_tilde(C_script)]{$\widetilde{L}\left(\mathscr{C}\right)$}{The collection of sets $\widetilde{L}\left(C\right)$ for all $C\in\mathscr{C}$}
\begin{align*}
\widetilde{r}\left(C,d\right) & :=\max\left\{ 0,\left\lceil \frac{\vol\left(\conv S_{n}^{d}+B^{d}\right)-\vol\left(\conv C+B^{d}\right)}{2\kappa_{d-1}}-1\right\rceil \right\} ,\\
\widetilde{L}\left(C\right) & :=\begin{cases}
\left\{ n-\widetilde{r}\left(C,d\right),\ldots,n-1,n\right\}  & \delta\left(C\right)>\delta\left(S_{n}^{d}\right)\\
\emptyset & \delta\left(C\right)\leq\delta\left(S_{n}^{d}\right)
\end{cases}.
\end{align*}
\end{defn}

A brief calculation shows that $\widetilde{r}\left(C,d\right)$ is
the maximum number of spheres that can be removed from $C$ so that
the approximate density upper bound $\widetilde{\delta}\left(p_{k}\left(C\right)\right)>\delta\left(S_{n-k}^{d}\right)$
holds. We show that for all $N\geq516,\!946$ there exist $m\in\mathbb{N}$
and $\mathbf{h}\in\left\{ 0,\ldots,\left\lfloor \frac{m-1}{2}\right\rfloor \right\} ^{3}$
such that $N\in\widetilde{L}\left(t_{\mathbf{h}}^{3}\left(Y_{m}\right)\cap D_{4}\right)$,
then there exists a full-dimensional packing (namely $t_{\mathbf{h}}^{3}\left(Y_{m}\right)\cap D_{4}$)
of $N$ points that is denser than $S_{N}^{4}$. The proof of Theorem
\ref{thm: MAIN THEOREM: N*_4 upper bound} is a direct consequence
of the following covering inclusions. 
\begin{lem}
\label{lem: Covering inclusion using t_=00007Bh=00007D^=00007B3=00007D(Y_m), with truncations}Let
$n^{*}=516,\!946$ and $N^{*}:=459,\!118,\!697$, then for all $n^{*}\leq N<N^{*}$
we have 
\begin{equation}
N\subseteq\bigcup_{m=1}^{\infty}\bigcup_{\mathbf{h}\in\left\{ 0,\ldots,\left\lfloor \frac{m-1}{2}\right\rfloor \right\} ^{3}}\widetilde{L}\left(t_{\mathbf{h}}^{3}\left(Y_{m}\right)\cap D_{4}\right).\label{eq: Covering inclusion using t_=00007Bh=00007D^=00007B3=00007D(Y_m), with truncations}
\end{equation}
\end{lem}

\begin{lem}
\label{lem: Covering inclusion using Y_m, no truncations}Let $N^{*}$
be as above, then for all $N\geq N^{*}$ we have 
\begin{equation}
N\subseteq\bigcup_{m=1}^{\infty}\widetilde{L}\left(Y_{m}\cap D_{4}\right).\label{eq: Covering inclusion using Y_m, no truncations}
\end{equation}
\end{lem}

We use interval arithmetic on a computer for the proof of Lemma \ref{lem: Covering inclusion using t_=00007Bh=00007D^=00007B3=00007D(Y_m), with truncations}.
Information about interval arithmetic can be found in \cite{MooreKearfottCloud2009,Tucker2011}
and the documentation \cite{JuliaIntervalArithmetic}. Ordinarily,
computers perform calculations on floating-point numbers, the set
of which---bar some exceptions such as $+\infty$ and $-\infty$
(see \cite{Kneusel2017} for more details)---is a finite subset of
the real numbers and is therefore unsuitable for rigorous conclusions
on $\mathbb{R}$. Interval arithmetic is a solution to this issue
by performing calculations not on individual numbers, but on ranges
$\left[a,b\right]$ known as intervals, where $a$ and $b$ are floating-point
numbers. Functions are defined to work with intervals such that if
$x\in\left[a,b\right]$ then $f\left(x\right)\in\left[c,d\right]$
for certain floating-point numbers $c$ and $d$ that depend on $a$,
$b$, and $f$, whether or not $x$ or $f\left(x\right)$ are themselves
floating-point numbers. We use the Julia programming language (2023,
version 1.8.5) \cite{Julia} together with the Julia Standard Library
package \texttt{LinearAlgebra} and the third-party packages \texttt{IntervalArithmetic}
by Benet and Sanders (2022, version 0.20.8) \cite{JuliaIntervalArithmetic}
and Readables by Sarnoff (2019, version 0.3.3)\cite{JuliaReadables}. 
\begin{proof}[Proof of Lemma \ref{lem: Covering inclusion using t_=00007Bh=00007D^=00007B3=00007D(Y_m), with truncations}]
In (\ref{eq: Covering inclusion using t_=00007Bh=00007D^=00007B3=00007D(Y_m), with truncations})
the outer union runs over all $m\in\mathbb{N}$, hence for computer
purposes we must restrict it to a finite subset. Let $m\in\mathbb{N}$,
$m_{\max}:=\left\lfloor \frac{m-1}{2}\right\rfloor $, $\mathbf{m}_{\max}:=\left(\left\lfloor \frac{m-1}{2}\right\rfloor ,\left\lfloor \frac{m-1}{2}\right\rfloor ,\left\lfloor \frac{m-1}{2}\right\rfloor \right)$,
and $\widetilde{r}_{\max}:=\widetilde{r}\left(t_{\mathbf{m}_{\max}}^{3}\left(Y_{m}\right)\cap D_{4},4\right)$.
Then 
\begin{align*}
\min\bigcup_{\mathbf{h}\in\left\{ 0,\ldots,\left\lfloor \frac{m-1}{2}\right\rfloor \right\} ^{3}}\widetilde{L}\left(t_{\mathbf{h}}^{3}\left(Y_{m}\right)\cap D_{4}\right) & =G\left(t_{\mathbf{m}_{\max}}^{3}\left(Y_{m}\right)\right)-\widetilde{r}_{\max},\\
\max\bigcup_{\mathbf{h}\in\left\{ 0,\ldots,\left\lfloor \frac{m-1}{2}\right\rfloor \right\} ^{3}}\widetilde{L}\left(t_{\mathbf{h}}^{3}\left(Y_{m}\right)\cap D_{4}\right) & =G\left(Y_{m}\right),
\end{align*}
 so for each $N\in\mathbb{N}$ it suffices to check in (\ref{eq: Covering inclusion using t_=00007Bh=00007D^=00007B3=00007D(Y_m), with truncations})
only those values of $m$ for which 
\begin{equation}
G\left(t_{\mathbf{m}_{\max}}^{3}\left(Y_{m}\right)\right)-\widetilde{r}_{\max}\leq N\leq G\left(Y_{m}\right).\label{eq: Range of m needed for a given n}
\end{equation}
 Since $G\left(t_{\mathbf{m}_{\max}}^{3}\left(Y_{m}\right)\right)>G\left(Y_{\left\lfloor \frac{m}{2}\right\rfloor }\right)$
and $G\left(Y_{m}\right)\rightarrow\infty$, for each $N$ there exists
only a finite number of $m$ that satisfy (\ref{eq: Range of m needed for a given n}).
Therefore we use interval arithmetic to show that for each $N\in\left\{ n^{*},n^{*}+1,\ldots,N^{*}-1\right\} $
there exists some $m$ satisfying (\ref{eq: Range of m needed for a given n})
(due to the nature of interval arithmetic, the precise inequalities
used in the code are more complicated) and $\mathbf{h}\in\left\{ 0,\ldots,\left\lfloor \frac{m-1}{2}\right\rfloor \right\} ^{3}$
such that $N\in\widetilde{L}\left(t_{\mathbf{h}}^{3}\left(Y_{m}\right)\cap D_{4}\right)$. 
\end{proof}
For the proof of Lemma \ref{lem: Covering inclusion using Y_m, no truncations},
we first note that for a fixed $m$, $k\mapsto\widetilde{\delta}\left(p_{k}\left(Y_{m}\cap D_{4}\right)\right)$
is decreasing. Then we establish that the ``worst-case packing''
$p_{G\left(Y_{m}\right)-G\left(Y_{m-1}\right)-1}\left(Y_{m}\cap D_{4}\right)$,
in the sense that it has the lowest approximate density among all
packings of the form $p_{k}\left(Y_{m}\cap D_{4}\right)$, $k\in\left\{ 0,\,\ldots,\,G\left(Y_{m}\right)-G\left(Y_{m-1}\right)-1\right\} $,
has density that monotonically increases in $m$ for all $m\geq43$.
Hence the sets $\widetilde{L}\left(p_{k}\left(Y_{m}\cap D_{4}\right)\right)$,
where $m\geq43$ and $k\in\left\{ 0,\,\ldots,\,G\left(Y_{m}\right)-G\left(Y_{m-1}\right)-1\right\} $,
contain all $N\geq G\left(Y_{43}\right)$. 
\begin{proof}[Proof of Lemma \ref{lem: Covering inclusion using Y_m, no truncations}]
From (\ref{eq: Gandini and Zucco G(Y_m)}) and (\ref{eq: Gandini and Zucco vol(Y_m + B^4)}),
we have 
\begin{align*}
G\left(p_{G\left(Y_{m}\right)-G\left(Y_{m-1}\right)-1}\left(Y_{m}\cap D_{4}\right)\right) & =4m^{4}-8m^{3}+8m^{2}-2,\\
\widetilde{\delta}\left(p_{G\left(Y_{m}\right)-G\left(Y_{m-1}\right)-1}\left(Y_{m}\cap D_{4}\right)\right) & =f_{Y}\left(m\right),
\end{align*}
 where 
\[
f_{Y}\left(x\right):=\frac{\left[4x^{4}-8x^{3}+8x^{2}-2\right]\cdot\frac{1}{2}\pi^{2}}{32x^{4}+64\sqrt{2}x^{3}+16\sqrt{3}\pi x^{2}+192\left(\arccos\left(\frac{1}{3}\right)-\frac{\pi}{3}\right)x+\frac{1}{2}\pi^{2}}.
\]
 Then 
\[
f_{Y}'\left(x\right)=\frac{g_{Y}\left(x\right)}{\left(32x^{4}+64\sqrt{2}x^{3}+16\sqrt{3}\pi x^{2}+192\left(\arccos\left(\frac{1}{3}\right)-\frac{\pi}{3}\right)x+\frac{1}{2}\pi^{2}\right)^{2}}\cdot\frac{1}{2}\pi^{2},
\]
 where $g_{Y}\left(x\right)=256\left(\sqrt{2}+1\right)x^{6}+a_{5}x^{5}+a_{4}x^{4}+a_{3}x^{3}+a_{2}x^{2}+a_{1}x+a_{0}$.
Each $a_{i}$, $i\in\left\{ 0,\ldots,5\right\} $, is bounded by $\left|a_{i}\right|\leq64\sqrt{2}\cdot24+8\cdot192\sqrt{2}=4344.46\ldots$,
then $g_{Y}\left(x\right)\geq256\left(\sqrt{2}+1\right)\left(x-42.176\ldots\right)x^{5}>0$
for all $x\geq43$, so $f_{Y}$ is monotonically increasing for all
$m\geq43$. We conclude the proof by observing that 
\begin{align*}
\widetilde{\delta}\left(p_{17,910,584}\left(Y_{104}\cap D_{4}\right)\right)=\widetilde{\delta}\left(Y_{104}\cap D_{4},\,459,\!118,\!697\right) & \approx0.589\,048\,622\,8\\
<\delta\left(S_{459,118,697}^{4}\right) & \approx0.589\,048\,623\,1,\\
\widetilde{\delta}\left(p_{17,910,583}\left(Y_{104}\cap D_{4}\right)\right)=\widetilde{\delta}\left(Y_{104}\cap D_{4},\,459,\!118,\!698\right) & \approx0.589\,048\,624\,1\\
>\delta\left(S_{459,118,698}^{4}\right) & \approx0.589\,048\,623\,1.
\end{align*}
\end{proof}
\begin{rem}
The code used in the proof of Lemma \ref{lem: Covering inclusion using t_=00007Bh=00007D^=00007B3=00007D(Y_m), with truncations}
shows that $516,\!945$ is \emph{not} in any $\widetilde{L}\left(t_{\mathbf{h}}^{3}\left(Y_{m}\right)\cap D_{4}\right)$,
hence $516,\!946$ is the best possible using our specific methods.
The polytope $t_{\mathbf{h}}^{3}\left(Y_{m}\right)$ contains six
vertices that are unmodified by the facet truncations, and by truncating
these vertices it is possible to improve our bounds for $n_{4}^{*}$
and $N_{4}^{*}$. 
\end{rem}

\section{Acknowledgements }

The author wishes to thank Martin Henk for his review and comments
of this paper. Additionally, the author wishes to thank Martin Henk,
Duncan Clark, Ansgar Freyer, and Cemile Kürko\u{g}lu for their review
and comments of portions of my PhD thesis-in-progress, from which
this paper is derived from.

\pagebreak{}

\bibliographystyle{plain}
\bibliography{Chun_Sausage_Catastrophe}

\pagebreak{}

\settowidth{\nomlabelwidth}{$\left|\mathbf{v}\right|$, $\mathbf{v}$ is a vector}
\printnomenclature{}

\pagebreak{}

\section{Appendix }

\subsection{\protect\label{subsec: The vertices of t_=00007Bh=00007D^=00007B3=00007D(Y_m) are at points of D_4}The
vertices of $t_{\mathbf{h}}^{3}\left(Y_{m}\right)$ }

In this subsection we prove the claim which we stated immediately
above Section \ref{subsec: Steiner polynomial for the truncation of a single facet}:
For all $m\in\mathbb{N}$ and $h\in\left\{ 0,\ldots,m\right\} $,
the set of vertices of $t_{h}^{3}\left(Y_{m}\right)$ is a subset
of $D_{4}$. 
\begin{lem}
\label{lem: Intersections of faces of t_=00007Bh=00007D(Y_m) with D_4 yield densest lattices}Let
$m\in\mathbb{N}_{0}$, $\mathbf{h}\in\left\{ 0,\ldots,\left\lfloor \frac{m-1}{2}\right\rfloor \right\} ^{3}$,
and $F$ be any face in $t_{\mathbf{h}}^{3}\left(Y_{m}\right)$. Then
with one exception, 
\[
\aff\left\{ F\cap D_{4}\right\} \cong\begin{cases}
\left\{ \mathbf{0}\right\}  & \dim F=0\\
\mathbb{Z} & \dim F=1\\
A_{2} & \dim F=2\\
D_{3} & \dim F=3\\
D_{4} & \dim F=4
\end{cases},
\]
and all lattices have minimal norm $4$. The exceptional face is any
square face $Q_{m,h}$ of $t_{h}^{3}\left(Y_{m}\right)$ for which
$Q_{m,h}\cap D_{4}\cong\mathbb{Z}^{2}$ with minimal norm $4$. 
\end{lem}

In other words, the intersection of any face $F$ (aside from $Q_{m,h}$)
of $t_{\mathbf{h}}^{3}\left(Y_{m}\right)$ with $D_{4}$ yields an
unscaled copy of a densest lattice in the appropriate dimension, and
the minimal norm of $4$ means that the maximal density can be achieved
by unit balls instead of scaled balls. In particular, points of $D_{4}$
reside at the vertices of these squares, hexagons, and other faces
of $t_{\mathbf{h}}^{3}\left(Y_{m}\right)$ and so we can write $\vol\left(\conv\left(t_{h}^{3}\left(Y_{m}\right)\cap D_{4}\right)+B^{4}\right)=\vol\left(t_{h}^{3}\left(Y_{m}\right)+B^{4}\right)$. 

Some remarks about this lemma and its proof are in order. Due to the
disjointness property (Proposition \ref{prop: The three truncated facets of the 24-cell are still disjoint}),
we can assume that the polytope is $t_{h}^{3}\left(Y_{m}\right)$,
i.e., $t_{\mathbf{h}}^{3}\left(Y_{m}\right)$ when $\mathbf{h}=\left\{ 0,\ldots,\left\lfloor \frac{m-1}{2}\right\rfloor \right\} \times\left\{ 0\right\} \times\left\{ 0\right\} $.
The $\dim F=4$ case is trivial. The case when $F$ is a face of $Y_{m}$
was established by Gandini and Zucco \cite{GandiniZucco1992}; this
case also covers all facets of $t_{h}^{3}\left(Y_{m}\right)$ because
a truncated facet $Y_{m}$ has the same affine hull as the corresponding
facet of $Y_{m}$. Hence we need to show the following dimension $0$,
$1$, and $2$ claims: 
\begin{enumerate}
\item Any vertex $\mathbf{v}$ of $t_{h}^{3}\left(Y_{m}\right)$ is a point
of $D_{4}$, that is, has integer coordinates with all even or all
odd terms. 
\item Any edge $E$ of $t_{h}^{3}\left(Y_{m}\right)$ contains two points
of $D_{4}$ of distance $2$ from each other. 
\item Any face $F$ of $t_{h}^{3}\left(Y_{m}\right)$ contains an equilateral
triangle of side length $2$ and vertices in $D_{4}$, aside from
the exceptional faces. 
\end{enumerate}
Claims 2 and 3 imply that $\aff\left(E\right)\cap D_{4}\cong\mathbb{Z}$
and $\aff\left(F\right)\cap D_{4}\cong A_{2}$ respectively. Moreover,
the three facets $X_{m,h_{1}}^{1}$, $X_{m,h_{16}}^{16}$, and $X_{m,h_{17}}^{17}$
(defined in Lemma \ref{def: Truncation of a single facet from X_m})
of the truncated $24$-cell $t_{\mathbf{h}}^{3}\left(Y_{m}\right)$
are disjoint for all $\mathbf{h}\in\left\{ 0,\ldots,\left\lfloor \frac{m-1}{2}\right\rfloor \right\} $,
so the conclusion of Lemma \ref{lem: Intersections of faces of t_=00007Bh=00007D(Y_m) with D_4 yield densest lattices}
also holds for the triple facet truncation $t_{\mathbf{h}}^{3}\left(Y_{m}\right)$
with $\mathbf{h}\in\left\{ 0,\ldots,\left\lfloor \frac{m-1}{2}\right\rfloor \right\} $. 
\begin{proof}[Proof of Claim 1]
Consider the truncated $24$-cell $t_{-h}^{3}\left(Y_{m}\right)$.
It suffices to show that the $24$ vertices of the truncated octahedron
$X_{m,h}^{1}$ are in $D_{4}$. Each vertex $\mathbf{v}^{m,h}$ of
$X_{m,h}^{1}$ is part of an edge $E_{m}$ of $Y_{m}$ whose endpoints
consist of one vertex of $X_{m}^{1}$ and the other vertex on the
hyperplane passing through $\mathbf{0}$ and perpendicular to the
centroid $\mathbf{c}=m\left(1,1,0,0\right)^{\mathsf{T}}$ of $Y_{m}$.
We wish to show that $\mathbf{v}^{m,h}\in D_{4}$, and by symmetry
of the octahedron and $24$-cell, it suffices to check this statement
for only \emph{one} such edge $E_{m,h}$. Let 
\[
E_{m}=\conv\left\{ m\begin{pmatrix}2\\
0\\
0\\
0
\end{pmatrix},m\begin{pmatrix}\hphantom{+}1\\
-1\\
\hphantom{+}1\\
-1
\end{pmatrix}\right\} ;
\]
 note that $\left(2,0,0,0\right)^{\mathsf{T}}\in X_{m}^{1}$ and $\left(1,-1,1,-1\right)^{\mathsf{T}}\cdot\mathbf{c}=0$.
Since $\dist\left(X_{m,h}^{1},X_{m}^{1}\right)=\sqrt{2}h$ and $\dist\left(\mathbf{0},Y_{m}\right)=\sqrt{2}m$,
the facet $X_{m,h}^{1}$ is $\frac{h}{m}$ of the way from $X_{m}^{1}$
to the origin. Then the point $\mathbf{v}^{m,h}$ is also $\frac{h}{m}$
of the way from $m\left(2,0,0,0\right)^{\mathsf{T}}$ to $m\left(1,-1,1,1\right)^{\mathsf{T}}$,
and so 
\[
\mathbf{v}^{m,h}=\left(1-\frac{h}{m}\right)\cdot m\begin{pmatrix}2\\
0\\
0\\
0
\end{pmatrix}+\frac{h}{m}\cdot m\begin{pmatrix}\hphantom{+}1\\
-1\\
\hphantom{+}1\\
-1
\end{pmatrix}=\begin{pmatrix}2m-h\\
-h\\
\hphantom{+}h\\
-h
\end{pmatrix}\in D_{4}.
\]
\end{proof}
For the second claim, we note that $t_{h}^{3}\left(Y_{m}\right)$
has four kinds of edges (see Subsubsection \ref{subsec: 1-dimensional edges of the truncated 24-cell}
for the details), but two kinds are subsets of existing edges in $Y_{m}$
so do not need to be rechecked. Similarly to the first claim, symmetry
considerations ensure that it is sufficient to show Claim 2 for the
representative edges $E_{m,h}^{\square}$ and $E_{m,h}^{\hexagon}$. 
\begin{proof}[Proof of Claim 2]
Consider the edge $E_{m,h}^{\square}$ with vertices 
\[
\mathbf{v}^{0}\in\conv\left\{ m\begin{pmatrix}2\\
0\\
0\\
0
\end{pmatrix},m\begin{pmatrix}\hphantom{+}1\\
-1\\
\hphantom{+}1\\
-1
\end{pmatrix}\right\} \qquad\text{and}\qquad\mathbf{v}^{1}\in\conv\left\{ m\begin{pmatrix}2\\
0\\
0\\
0
\end{pmatrix},m\begin{pmatrix}\hphantom{+}1\\
-1\\
\hphantom{+}1\\
\hphantom{+}1
\end{pmatrix}\right\} 
\]
 respectively. Then 
\[
\mathbf{v}^{0}=\begin{pmatrix}2m-h\\
-h\\
\hphantom{+}h\\
-h
\end{pmatrix}\qquad\text{and}\qquad\mathbf{v}^{1}=\begin{pmatrix}2m-h\\
-h\\
\hphantom{+}h\\
\hphantom{+}h
\end{pmatrix}\text{,}
\]
 so every point on $E_{m,h}^{\square}$ is of the form 
\[
\mathbf{v}^{\lambda}=\begin{pmatrix}2m-h\\
-h\\
\hphantom{+}h\\
\left(2\lambda-1\right)h
\end{pmatrix}.
\]
 In particular, 
\[
\mathbf{v}^{\frac{1}{h}}=\begin{pmatrix}2m-h\\
-h\\
\hphantom{+}h\\
\left(2\frac{1}{h}-1\right)h
\end{pmatrix}=\begin{pmatrix}2m-h\\
-h\\
\hphantom{+}h\\
\hphantom{m}2-h
\end{pmatrix}\in D_{4}
\]
 and in addition $\left|\mathbf{v}^{\frac{1}{h}}-\mathbf{v}^{0}\right|=2$. 

Now consider the edge $E_{m,h}^{\hexagon}$ with vertices 
\[
\mathbf{v}^{0}\in\conv\left\{ m\begin{pmatrix}2\\
0\\
0\\
0
\end{pmatrix},m\begin{pmatrix}\hphantom{+}1\\
-1\\
\hphantom{+}1\\
-1
\end{pmatrix}\right\} \qquad\text{and}\qquad\mathbf{v}^{1}\in\conv\left\{ m\begin{pmatrix}\hphantom{+}1\\
\hphantom{+}1\\
\hphantom{+}1\\
-1
\end{pmatrix},m\begin{pmatrix}\hphantom{+}1\\
-1\\
\hphantom{+}1\\
-1
\end{pmatrix}\right\} 
\]
 respectively. Then 
\[
\mathbf{v}^{0}=\begin{pmatrix}2m-h\\
-h\\
\hphantom{+}h\\
-h
\end{pmatrix}\qquad\text{and}\qquad\mathbf{v}^{1}=\begin{pmatrix}\hphantom{+}m\\
m-2h\\
\hphantom{+}m\\
-m
\end{pmatrix}\text{,}
\]
 so every point on $E_{m,h}^{\hexagon}$ is of the form 
\[
\mathbf{v}^{\lambda}=\begin{pmatrix}\left(2-\lambda\right)m-\left(1-\lambda\right)h\\
\lambda m-\left(1+\lambda\right)h\\
\lambda m+\left(1-\lambda\right)h\\
-\lambda m-\left(1-\lambda\right)h
\end{pmatrix}.
\]
 Let $\lambda=\frac{1}{m-h}$ for 
\[
\mathbf{v}^{\frac{1}{m-h}}=\begin{pmatrix}2m-\left(h+1\right)\\
-h+1\\
\hphantom{+}h+1\\
-\left(h+1\right)
\end{pmatrix}\in D_{4},
\]
 and finally $\left|\mathbf{v}^{0}-\mathbf{v}^{\frac{1}{m-h}}\right|=2$. 
\end{proof}
For the third claim, we observe that $t_{h}^{3}\left(Y_{m}\right)$
has five kinds of faces (see Subsubsection \ref{subsec: 2-dimensional faces of the truncated 24-cell}),
but only two of them are not subsets of faces in $Y_{m}$. This fact,
along with symmetry considerations, imply that it suffices to show
Claim 3 for the representative faces $Q_{m,h}$ and $H_{m,h}$. 
\begin{proof}[Proof of Claim 3]
Consider the square face $Q_{m,h}$ with vertices 
\[
\mathbf{v}^{1,1}=\begin{pmatrix}2m-h\\
-h\\
\hphantom{+}h\\
\hphantom{+}h
\end{pmatrix},\qquad\mathbf{v}^{-1,1}=\begin{pmatrix}2m-h\\
-h\\
-h\\
\hphantom{+}h
\end{pmatrix},\qquad\mathbf{v}^{1,-1}=\begin{pmatrix}2m-h\\
-h\\
\hphantom{+}h\\
-h
\end{pmatrix},\qquad\mathbf{v}^{-1,-1}=\begin{pmatrix}2m-h\\
-h\\
-h\\
-h
\end{pmatrix}.
\]
 Any point of $Q_{m,h}$ can be written as 
\[
\mathbf{v}^{\lambda_{1},\lambda_{2}}=\begin{pmatrix}2m-h\\
-h\\
\lambda_{1}h\\
\lambda_{2}h
\end{pmatrix}\qquad\text{for }\lambda_{1},\lambda_{2}\in\left[-1,1\right]\text{,}
\]
 then in particular for $\lambda_{1},\lambda_{2}=1-\frac{2}{h}$,
we have 
\[
\mathbf{v}^{1-\frac{2}{h},\,1-\frac{2}{h}}=\begin{pmatrix}2m-h\\
-h\\
\left(1-\frac{2}{h}\right)h\\
\left(1-\frac{2}{h}\right)h
\end{pmatrix}=\begin{pmatrix}2m-h\\
-h\\
h-2\\
h-2
\end{pmatrix}\in D_{4}.
\]
 Then $\conv\left\{ \mathbf{v}^{1,1},\,\mathbf{v}^{1,\,1-\frac{2}{h}},\,\mathbf{v}^{1-\frac{2}{h},\,1},\,\mathbf{v}^{1-\frac{2}{h},\,1-\frac{2}{h}}\right\} $
forms a square of side length $2$ with vertices in $D_{4}$. 

Now consider the hexagonal face $H_{m,h}$ with adjacent vertices
\[
\mathbf{v}^{-1}=\begin{pmatrix}2m-h\\
-h\\
\hphantom{+}h\\
-h
\end{pmatrix},\qquad\mathbf{v}^{0}=\begin{pmatrix}2m-h\\
-h\\
\hphantom{+}h\\
\hphantom{+}h
\end{pmatrix},\qquad\mathbf{v}^{1}=\begin{pmatrix}m\\
m-2h\\
m\\
m
\end{pmatrix}.
\]
 The points of $D_{4}$ on $\conv\left\{ \mathbf{v}^{0},\mathbf{v}^{-1}\right\} $
and $\conv\left\{ \mathbf{v}^{0},\mathbf{v}^{1}\right\} $ closest
to but distinct from $\mathbf{v}^{0}$ are 
\[
\mathbf{v}^{-\frac{1}{h}}=\begin{pmatrix}2m-h\\
-h\\
h\\
h-2
\end{pmatrix}\qquad\text{and}\qquad\mathbf{v}^{\frac{1}{m-h}}=\begin{pmatrix}2m-h-1\\
1-h\\
h+1\\
h+1
\end{pmatrix}.
\]
 Let 
\[
\mathbf{v}=\mathbf{v}^{0}+\left(\mathbf{v}^{\frac{1}{m-h}}-\mathbf{v}^{0}\right)+\left(\mathbf{v}^{-\frac{1}{h}}-\mathbf{v}^{0}\right)=\begin{pmatrix}2m-h-1\\
1-h\\
h+1\\
h-1
\end{pmatrix}\in D_{4},
\]
 then $\left\{ \mathbf{v}^{0},\mathbf{v}^{\frac{1}{m-h}},\mathbf{v}\right\} $
form an equilateral triangle of side length $2$ with vertices in
$D_{4}$. 
\end{proof}
\pagebreak{}

\subsection{Faces of $t_{h}^{3}\left(Y_{m}\right)$ }
\noindent \begin{center}
\begin{figure}[H]
\noindent \begin{centering}
\includegraphics[width=6.24in]{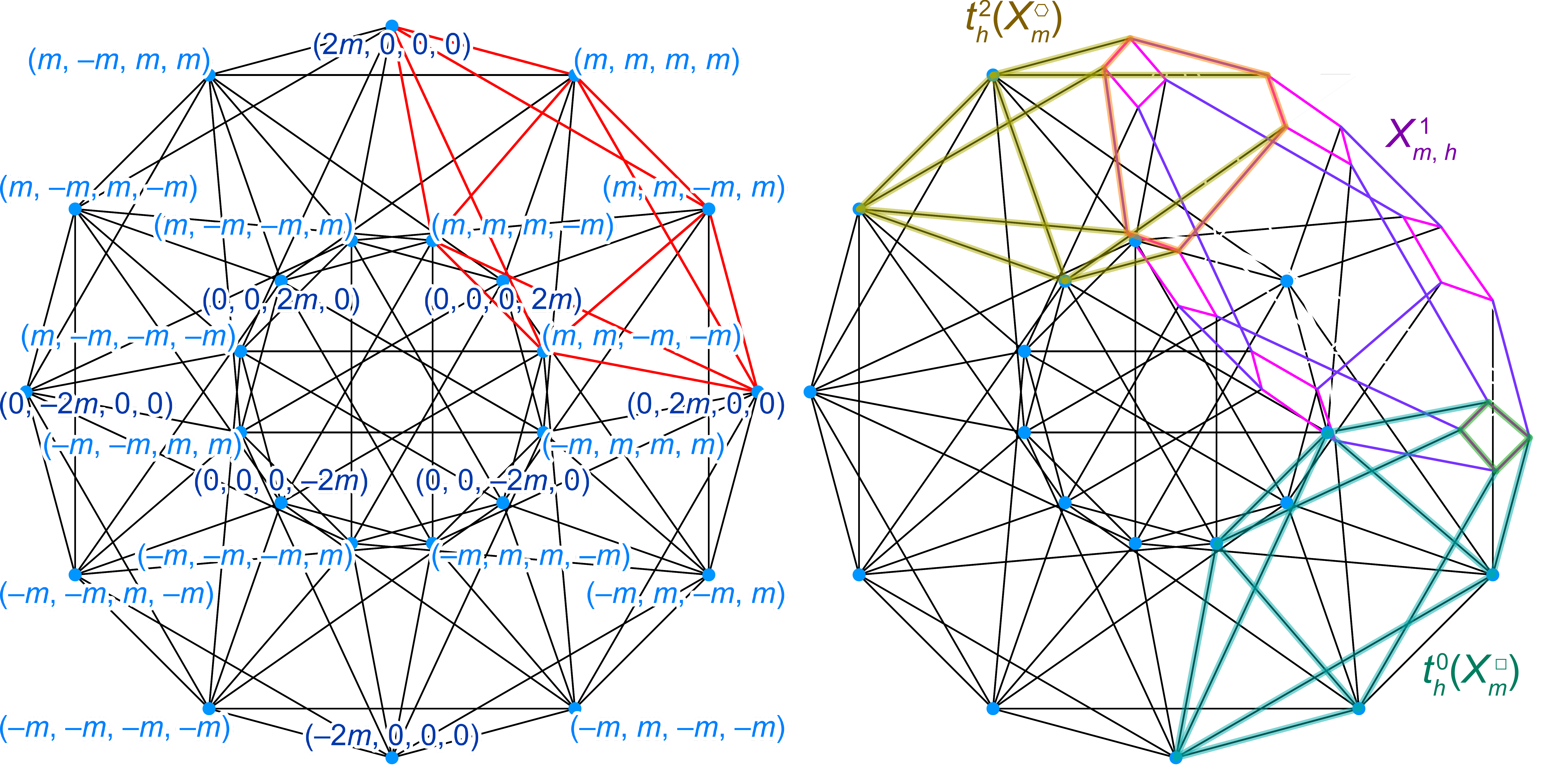}
\par\end{centering}
\caption{\protect\label{fig: Truncated facets of t_=00007Bh=00007D^=00007B3=00007D(Y_m)}\emph{Left:}
A diagram of $Y_{m}$ with the edges of $X_{m}^{1}$ in \emph{red}.
\emph{Right:} A diagram of $t_{h}^{3}\left(Y_{m}\right)$ with the
facet $X_{m,h}^{1}$ in \emph{pink} (square edges) and \emph{purple}
(other edges), the facet $t_{h}^{2}\left(X_{m}^{\hexagon}\right)$
in \emph{orange} (its intersection with $X_{m,h}^{1}$) and \emph{yellow}
(other edges), and the facet $t_{h}^{0}\left(X_{m}^{\square}\right)$
in \emph{green} (its intersection with $X_{m,h}^{1}$) and \emph{teal}
(other edges). }
\end{figure}
\begin{figure}[H]
\noindent \begin{centering}
\includegraphics[width=6.24in]{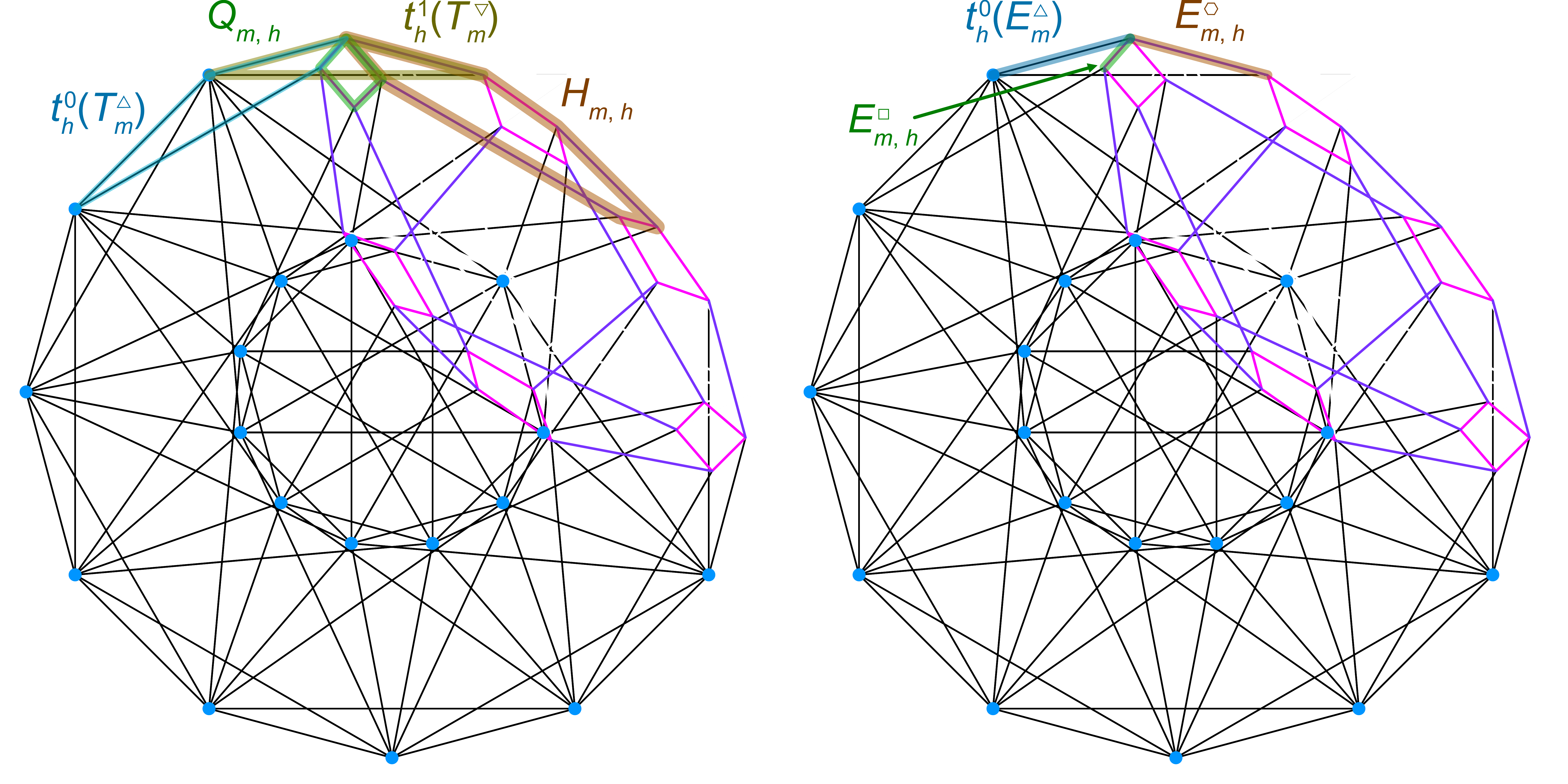}
\par\end{centering}
\caption{\protect\label{fig: Truncated faces and edges of t_=00007Bh=00007D^=00007B3=00007D(Y_m)}Diagrams
of $t_{h}^{3}\left(Y_{m}\right)$ with the facet $X_{m,h}^{1}$ in
\emph{pink} (square edges) and \emph{purple} (other edges). \emph{Left:}
The edges of the faces $Q_{m,h}$, $H_{m,h}$, $t_{h}^{1}\left(T_{m}^{\triangledown}\right)$,
and $t_{h}^{0}\left(T_{m}^{\triangle}\right)$ are colored \emph{green},
\emph{orange}, \emph{yellow}, and \emph{blue} respectively. \emph{Right:}
The edges $E_{m,h}^{\square}$, $E_{m,h}^{\hexagon}$, and $t_{h}^{0}\left(E_{m}^{\triangle}\right)$
are colored \emph{green}, \emph{orange}, and \emph{blue} respectively. }
\end{figure}
\par\end{center}

\pagebreak{}

\subsection{Density graphs for $t_{\mathbf{h}}^{3}\left(Y_{m}\right)$ }

This subsection presents graphs of the densities of various truncated
$24$-cells $t_{\mathbf{h}}^{3}\left(Y_{m}\right)$ . 
\noindent \begin{center}
\begin{figure}[H]
\noindent \begin{centering}
\includegraphics[width=6.24in]{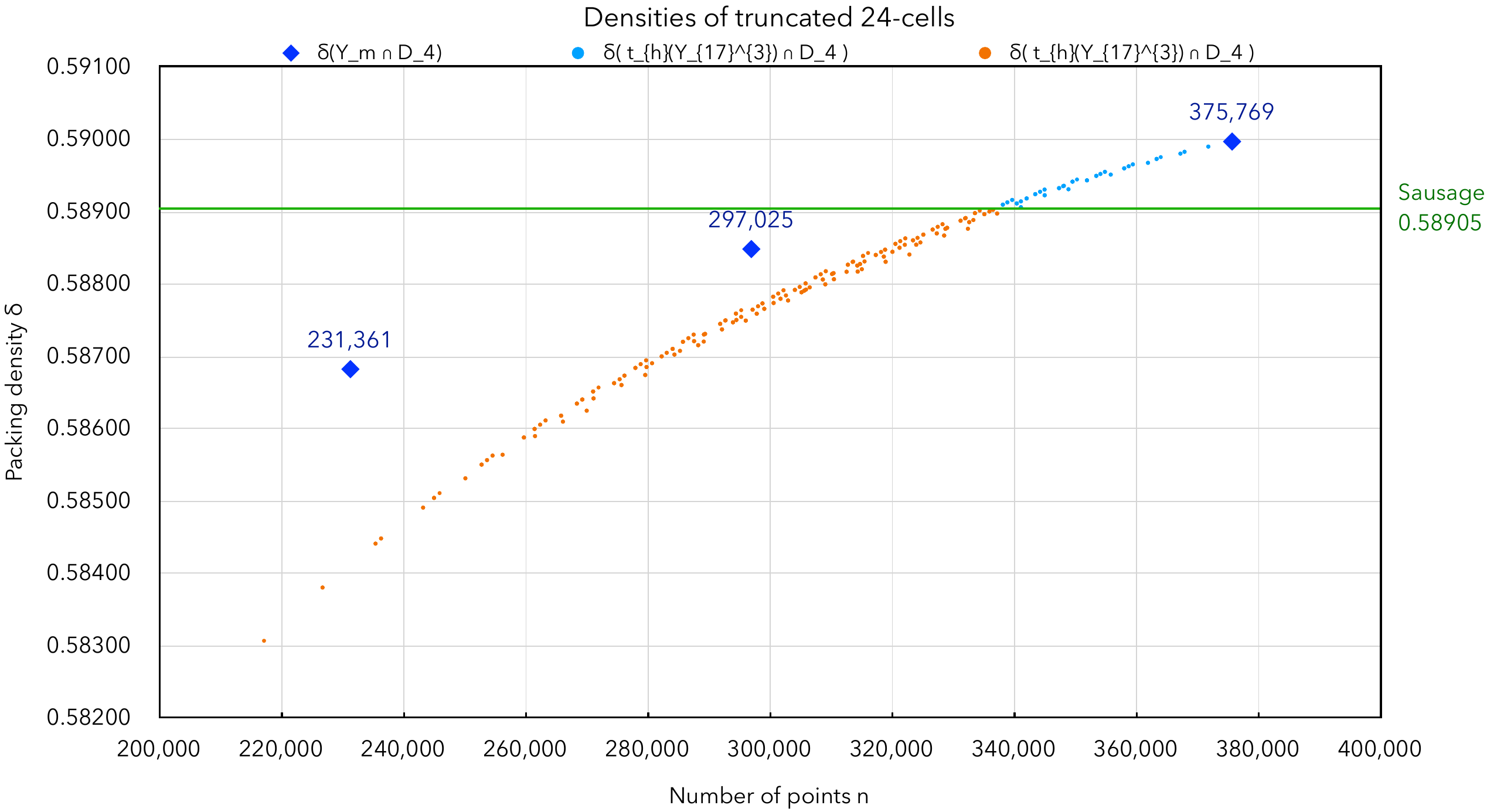}
\par\end{centering}
\caption{\protect\label{fig: Densities of the truncated 24-cells}The densities
of $t_{\mathbf{h}}^{3}\left(Y_{17}\right)\cap D_{4}$ for $\mathbf{h}\in\left\{ 0,\ldots,8\right\} ^{3}$.}

The \emph{green horizontal line} is the density $\delta\left(S_{n}^{d}\right)$
of the sausage, the \emph{dark blue diamonds} are $\delta\left(X_{16}\cap D_{4}\right)$
and $\delta\left(X_{17}\cap D_{4}\right)$, the \emph{light blue circles}
are packings $t_{\mathbf{h}}^{3}\left(Y_{17}\right)\cap D_{4}$ that
are denser than the sausage, and the \emph{orange circles} are packings
$t_{\mathbf{h}}^{3}\left(Y_{17}\right)\cap D_{4}$ that are less dense
than the sausage.
\end{figure}
\par\end{center}

\noindent \begin{center}
\begin{figure}[H]
\noindent \begin{centering}
\includegraphics[width=6.24in]{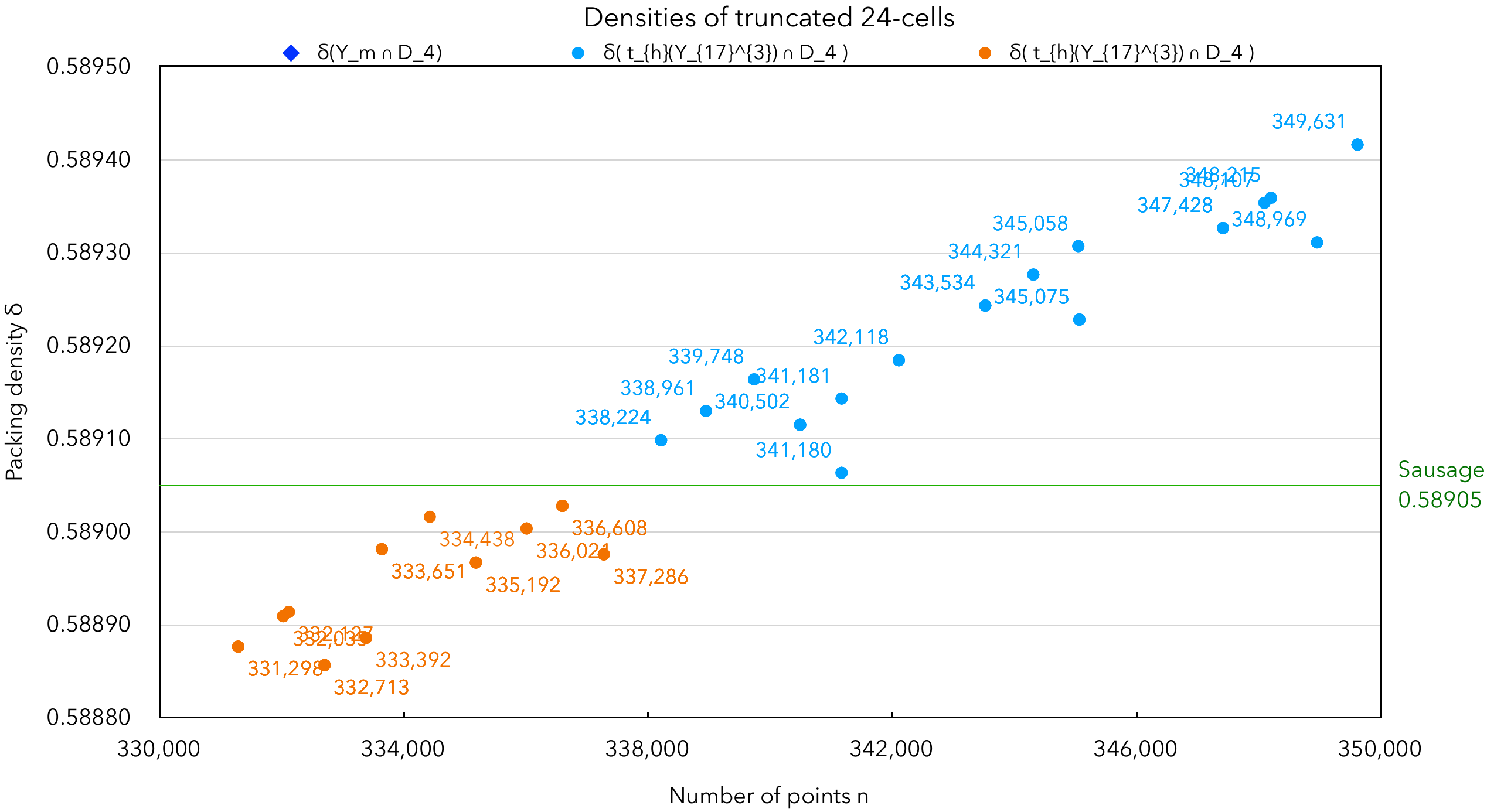}
\par\end{centering}
\caption{\protect\label{fig: Densities of the truncated 24-cells, zoom in}Figure
\ref{fig: Densities of the truncated 24-cells}, zoomed in around
$n=338,\!224$.}
\end{figure}
\par\end{center}

\pagebreak{}

\subsection{\protect\label{subsec: Graphs of 24-cell truncations}Density graphs
for $p_{k}\left(t_{\mathbf{h}}^{3}\left(Y_{m}\right)\cap D_{4}\right)$ }

\subsubsection{\protect\label{subsec: Graphs of 24-cell truncations, descriptions, overview, and formulas}Descriptions
and overview }

In the graphs of this subsection it is helpful to see the packing
densities of the packings $p_{k}\left(t_{\mathbf{h}}^{3}\left(Y_{17}\right)\cap D_{4}\right)$
along with the intervals $\widetilde{L}\left(t_{\mathbf{h}}^{3}\left(Y_{17}\right)\cap D_{4}\right)\subset\mathbb{N}$.
Since they are visual aids only and are not involved in any proofs,
our discussion here will be somewhat informal. 
\begin{defn}
\label{def: The TRIANGLE}Let $C\subset\mathbb{R}^{d}$ be a packing
of $n\geq3$ points. Define $\widetilde{\triangle}\left(C\right)\subset\mathbb{R}^{2}$\nomenclature[Z1 TRIANGLE]{$\widetilde{\triangle}\left(C\right)$}{The set in $\mathbb{R}^{2}$ defined as the convex hull of the following vertices: \\$\left(n-\widetilde{r}\left(C,d\right),\,\widetilde{\delta}\left(p_{r\left(C,n,d\right)}\left(C\right)\right)\right),\ \ldots,\ \left(n-1,\widetilde{\delta}\left(p_{1}\left(C\right)\right)\right),\ \left(n,\delta\left(C\right)\right)$ and \\$\left(n-\widetilde{r}\left(C,d\right),\,\delta\left(S_{n-r\left(C,n,d\right)}^{d}\right)\right),\ \ldots,\ \left(n-1,\delta\left(S_{n-1}^{d}\right)\right),\ \left(n,\delta\left(S_{n}^{d}\right)\right)$}
by 
\[
\widetilde{\triangle}\left(C\right):=\conv\left\{ \begin{array}{llll}
\left(n-\widetilde{r}\left(C,d\right),\,\widetilde{\delta}\left(p_{r\left(C,n,d\right)}\left(C\right)\right)\right), & \ldots, & \left(n-1,\widetilde{\delta}\left(p_{1}\left(C\right)\right)\right), & \left(n,\delta\left(C\right)\right),\\
\\
\left(n-\widetilde{r}\left(C,d\right),\,\delta\left(S_{n-r\left(C,n,d\right)}^{d}\right)\right), & \ldots, & \left(n-1,\delta\left(S_{n-1}^{d}\right)\right), & \left(n,\delta\left(S_{n}^{d}\right)\right)
\end{array}\right\} 
\]
 if $\delta\left(C\right)\geq\delta\left(S_{n}^{d}\right)$ and $\widetilde{\triangle}\left(C\right):=\emptyset$
if $\delta\left(C\right)<\delta\left(S_{n}^{d}\right)$. 
\end{defn}

While $\widetilde{\triangle}\left(C\right)$ is typically a polygon
with $\widetilde{r}\left(C,d\right)+3$ sides, it visually resembles
the right triangle\nomenclature[Z2 TRIANGLE, hat]{$\widehat{\triangle}\left(C\right)$}{The right triangle with vertices $\left(n-\widetilde{r}\left(C,d\right),\,\delta\left(S_{n-r\left(C,n,d\right)}^{d}\right)\right)$, $\left(n,\delta\left(S_{n}^{d}\right)\right)$, and $\left(n,\delta\left(C\right)\right)$}
\[
\widehat{\triangle}\left(C\right):=\conv\left\{ \begin{array}{ll}
 & \left(n,\delta\left(C\right)\right),\\
\left(n-\widetilde{r}\left(C,d\right),\,\delta\left(S_{n-r\left(C,n,d\right)}^{d}\right)\right), & \left(n,\delta\left(S_{n}^{d}\right)\right)
\end{array}\right\} ,
\]
 and in our graphs we will draw $\widehat{\triangle}$ instead of
$\widetilde{\triangle}$. 
\noindent \begin{center}
\begin{figure}[H]
\noindent \begin{centering}
\includegraphics[width=6.24in]{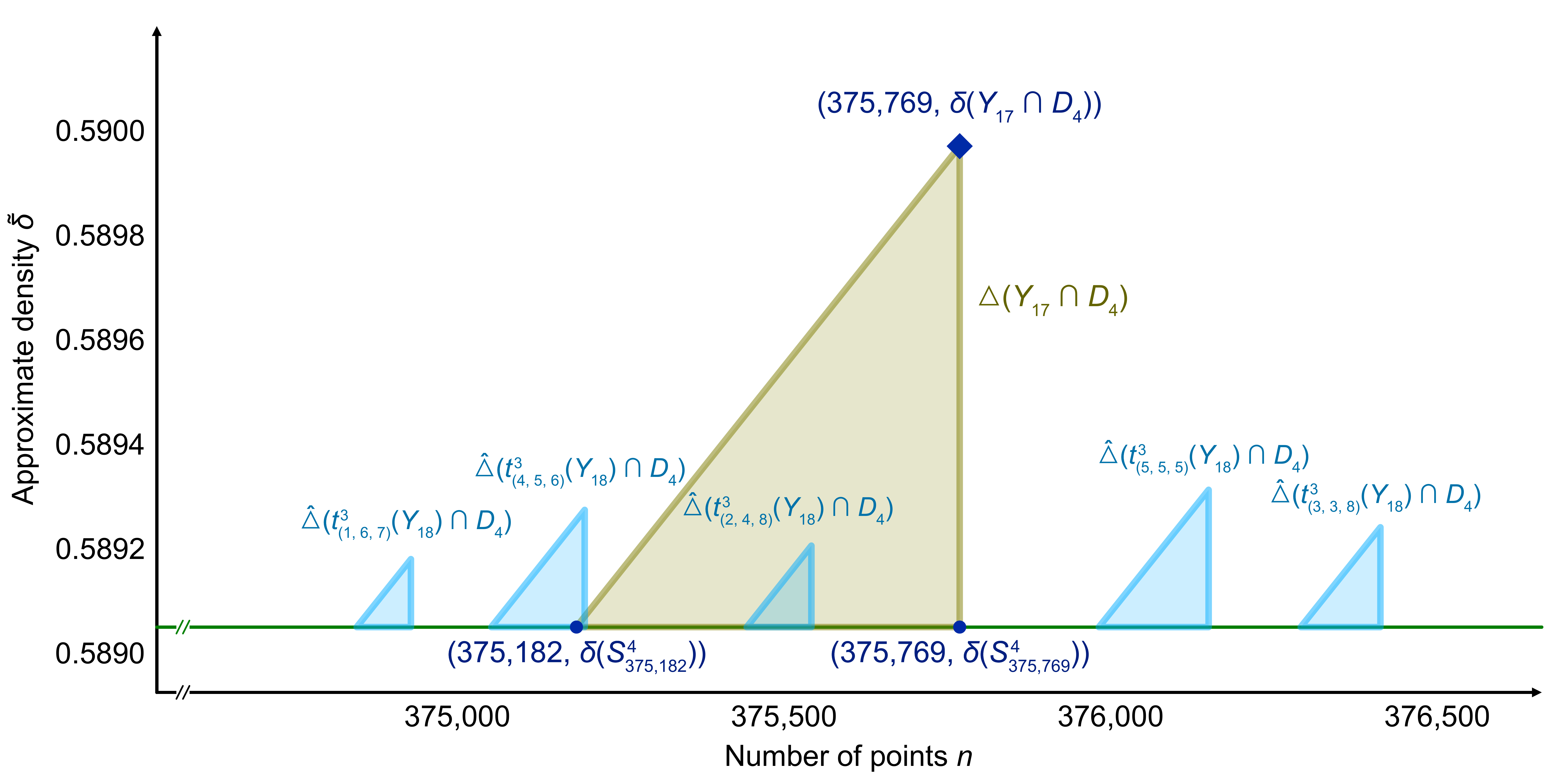}
\par\end{centering}
\caption{\protect\label{fig: The TRIANGLES}The triangles $\widehat{\triangle}\left(Y_{17}\cap D_{4}\right)$
and $\widehat{\triangle}\left(t_{-\mathbf{h}}^{3}\left(Y_{18}\right)\cap D_{4}\right)$
for certain values of $m$ and $\mathbf{h}$. }

The \emph{dark yellow triangle} is $\widehat{\triangle}\left(Y_{17}\cap D_{4}\right)$
and the \emph{light blue triangles} are five distinct $\widehat{\triangle}\left(t_{-\mathbf{h}}^{3}\left(Y_{18}\right)\cap D_{4}\right)$'s.
The smallest and largest values in $\widetilde{L}\left(Y_{17}\cap D_{4}\right)$
are also marked on the diagram: $G\left(Y_{17}\right)=375,\!769$
and $G\left(Y_{17}\right)-\widetilde{r}\left(Y_{17}\cap D_{4},4\right)=375,\!182$.
The \emph{green line} is $y=\delta\left(S_{n}^{4}\right)$.
\end{figure}
\par\end{center}

The set $\widehat{\triangle}\left(C\right)$ is also closely related
to $\widetilde{L}\left(C\right)$; its projection onto the first coordinate
is simply $\widetilde{L}\left(C\right)$, and two packings $C$ and
$C'$ satisfy $\widehat{\triangle}\left(C\right)\cap\widehat{\triangle}\left(C'\right)\neq\emptyset$
if and only if $\widetilde{L}\left(C\right)\cap\widetilde{L}\left(C'\right)\neq\emptyset$.
The following graph is a bird's-eye-view of (the boundaries of) the
triangles $\widehat{\triangle}\left(t_{\mathbf{h}}^{3}\left(Y_{m}\right)\cap D_{4}\right)$
for all $m\in\left\{ 17,\ldots,21\right\} $ and $\mathbf{h}\in\left\{ 0,\ldots,\left\lfloor \frac{m-1}{2}\right\rfloor \right\} ^{3}$. 
\noindent \begin{center}
\begin{figure}[H]
\noindent \begin{centering}
\includegraphics[width=6.24in]{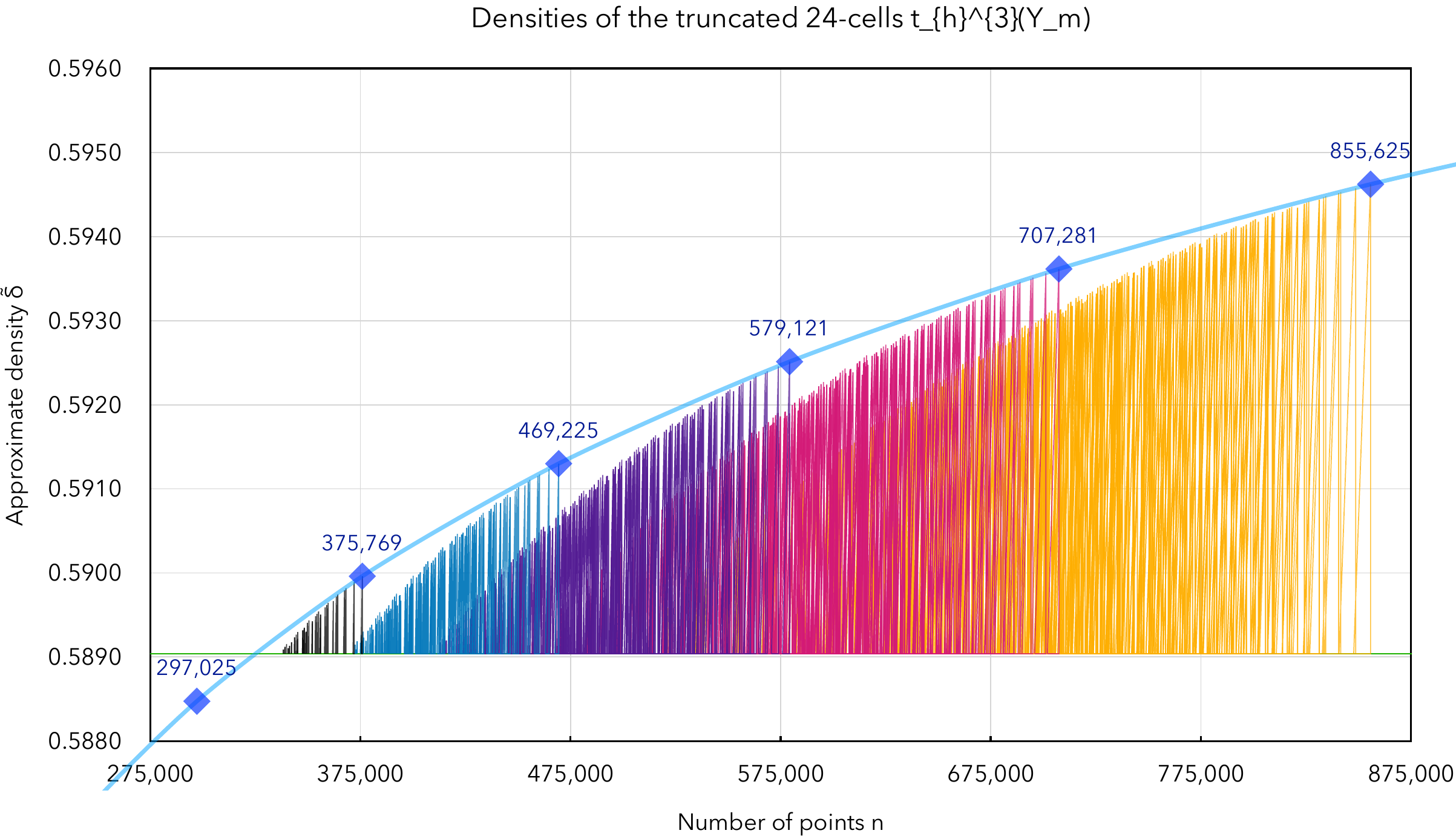}
\par\end{centering}
\caption{\protect\label{fig: A bird's eye view of the triangles}The \textquotedblleft triangles\textquotedblright{}
$\widehat{\triangle}\left(t_{\mathbf{h}}^{3}\left(Y_{m}\right)\cap D_{4}\right)$
for $m\in\left\{ 17,\ldots,21\right\} $ and $\mathbf{h}\in\left\{ 0,\ldots,\left\lfloor \frac{m-1}{2}\right\rfloor \right\} ^{3}$.}

\noindent The \emph{green line} near the bottom is the density $\delta\left(S_{n}^{4}\right)$
of the sausage, the \emph{large blue diamonds} are the points $\left(x,y\right)=\left(G\left(Y_{m}\right),\delta\left(Y_{m}\cap D_{4}\right)\right)$,
and the \emph{pale blue curve} is an interpolation between the points
$\left(G\left(Y_{m}\right),\delta\left(Y_{m}\cap D_{4}\right)\right)$
for $m\in\left\{ 1,\ldots,24\right\} $. The spikes in the middle
are $\widehat{\triangle}\left(t_{\mathbf{h}}^{3}\left(Y_{m}\right)\cap D_{4}\right)$
and are colored \emph{black}, \emph{dark blue}, \emph{purple}, \emph{pink},
and \emph{yellow} for $m=17$, $18$, $19,$$20,$and $21$ respectively. 
\end{figure}
\par\end{center}

In the subsequent graphs below we zoom in on various portions of this
graph, which will show some of the gaps in, and coverings of, the
sets $\widetilde{L}\left(t_{\mathbf{h}}^{3}\left(Y_{m}\right)\cap D_{4}\right)$
for $m\in\left\{ 17,\ldots,21\right\} $ and $h\in\left\{ 0,\ldots,\left\lfloor \frac{m-1}{2}\right\rfloor \right\} ^{3}$.
In particular, the graphs show that 
\[
\left\{ 516,\!837,\ \ldots,\ 516,935\right\} \cap\left(\bigcup_{m=1}^{m'}\bigcup_{h\in\left\{ 0,\ldots,\left\lfloor \frac{m-1}{2}\right\rfloor \right\} ^{3}}\widetilde{L}\left(t_{\mathbf{h}}^{3}\left(Y_{m}\right)\cap D_{4}\right)\right)=\emptyset
\]
 and 
\begin{align*}
\widetilde{L}\left(t_{\left(0,0,1\right)}^{3}\left(Y_{m}\right)\cap D_{4}\right)\cap\widetilde{L}\left(t_{\left(0,0,0\right)}^{3}\left(Y_{m}\right)\cap D_{4}\right) & =\emptyset\qquad\text{for }m\in\left\{ 17,18,19,20\right\} \text{,}\\
\widetilde{L}\left(t_{\left(0,0,1\right)}^{3}\left(Z_{21}\right)\cap D_{4}\right)\cap\widetilde{L}\left(t_{\left(0,0,0\right)}^{3}\left(Z_{21}\right)\cap D_{4}\right) & \neq\emptyset,
\end{align*}
 which illustrate the inequality $N_{4}^{*}\leq516,\!946$. 

In the next subsubsection we enlarge specific regions of Figure \ref{fig: A bird's eye view of the triangles}
for a clearer view. For clarity, we change the colors so that $\delta\left(S_{n}^{d}\right)$
is (usually dark) red while $\widehat{\triangle}\left(t_{\mathbf{h}}^{3}\left(Y_{m}\right)\cap D_{4}\right)$
is light grey or black for odd $m$ and blue for even $m$. Each graph
focuses on $\widehat{\triangle}\left(t_{-\mathbf{h}}^{3}\left(Y_{m}\right)\cap D_{4}\right)$
for a fixed $m$, however, the triangles $\widehat{\triangle}\left(t_{-\mathbf{h}}^{3}\left(Y_{m}\right)\cap D_{4}\right)$
and $\widehat{\triangle}\left(t_{\mathbf{h}}^{3}\left(Y_{m'}\right)\cap D_{4}\right)$
may overlap for $m\neq m'$. We omit any triangle $\widehat{\triangle}\left(t_{\mathbf{h}}^{3}\left(Y_{m}\right)\cap D_{4}\right)$
with $m\geq22$, in part to show the overlap of the triangles $\widehat{\triangle}\left(t_{\left(0,0,0\right)}^{3}\left(Y_{21}\right)\cap D_{4}\right)$
and $\widehat{\triangle}\left(t_{\left(0,0,1\right)}^{3}\left(Y_{21}\right)\cap D_{4}\right)$
in Figure \ref{fig: The triangles for m =00003D 21}. Another graph
(Figure \ref{fig: The triangles for m =00003D 19, zoom}) zooms in
heavily around the interval $\left[516,\!836,\ 516,\!946\right]$,
illustrating the gap in the triangles immediately preceding $n=516,\!946$,
since the intermediate points $n\in\left\{ 516,\!837,\ \ldots,\ 516,\!945\right\} $
are not covered by the union 
\[
\bigcup_{m\in\mathbb{N}}\bigcup_{\mathbf{h}\in\left\{ 0,\ldots,\left\lfloor \frac{m-1}{2}\right\rfloor \right\} ^{3}}\widetilde{L}\left(t_{\mathbf{h}}^{3}\left(Y_{m}\right)\cap D_{4}\right).
\]

\pagebreak{}

\subsubsection{Detailed graphs of $\widehat{\triangle}\left(t_{-\mathbf{h}}^{3}\left(Y_{m}\right)\cap D_{4}\right)$ }
\noindent \begin{center}
\begin{figure}[H]
\noindent \begin{centering}
\includegraphics[width=6.24in]{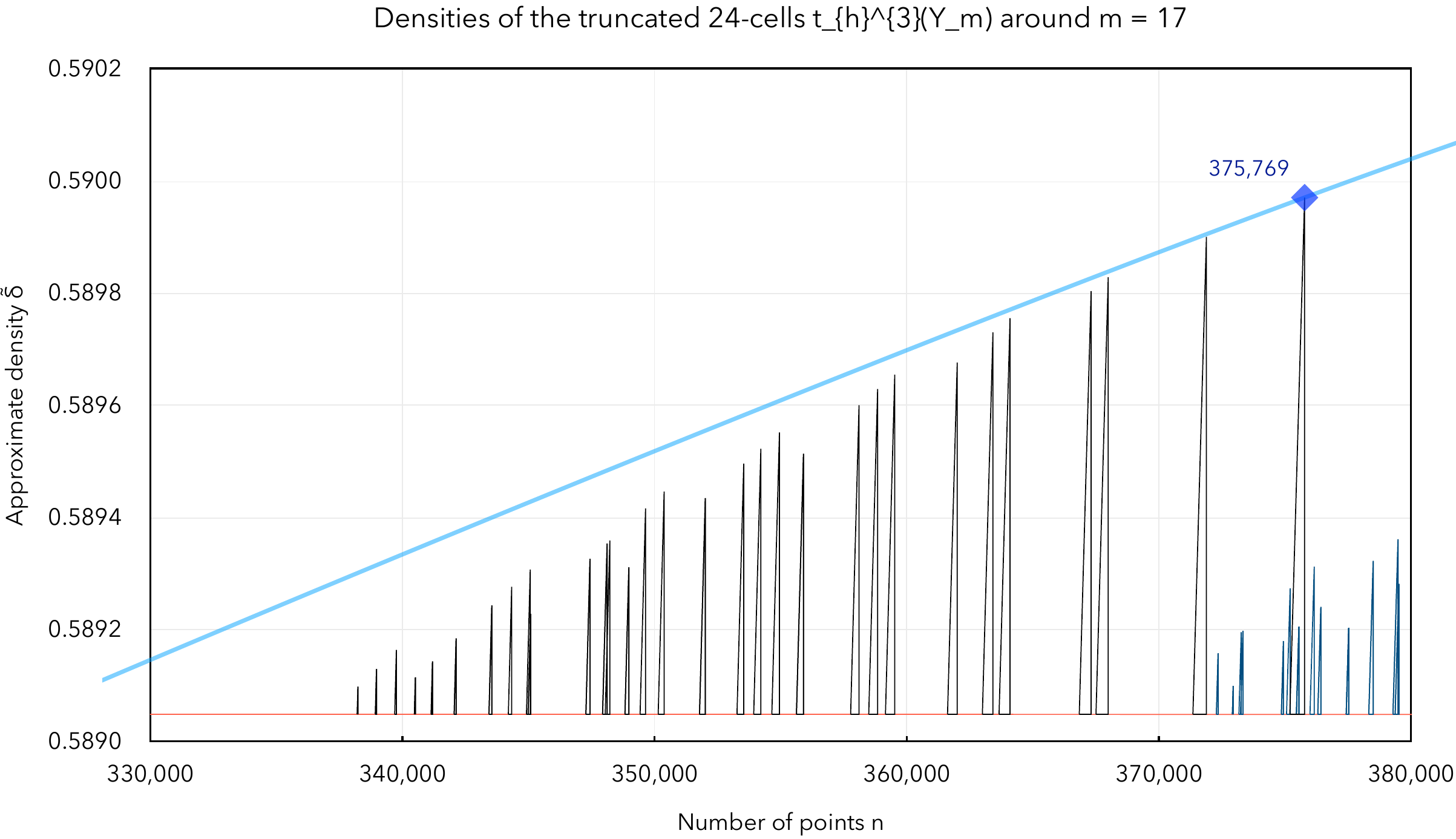}
\par\end{centering}
\caption{\protect\label{fig: The triangles for m =00003D 17}The \textquotedblleft triangles\textquotedblright{}
$\widehat{\triangle}\left(t_{\mathbf{h}}^{3}\left(Y_{17}\right)\cap D_{4}\right)$
for $\mathbf{h}\in\left\{ 0,\ldots,\left\lfloor \frac{m-1}{2}\right\rfloor \right\} ^{3}$.}
\end{figure}
\begin{figure}[H]
\noindent \begin{centering}
\includegraphics[width=6.24in]{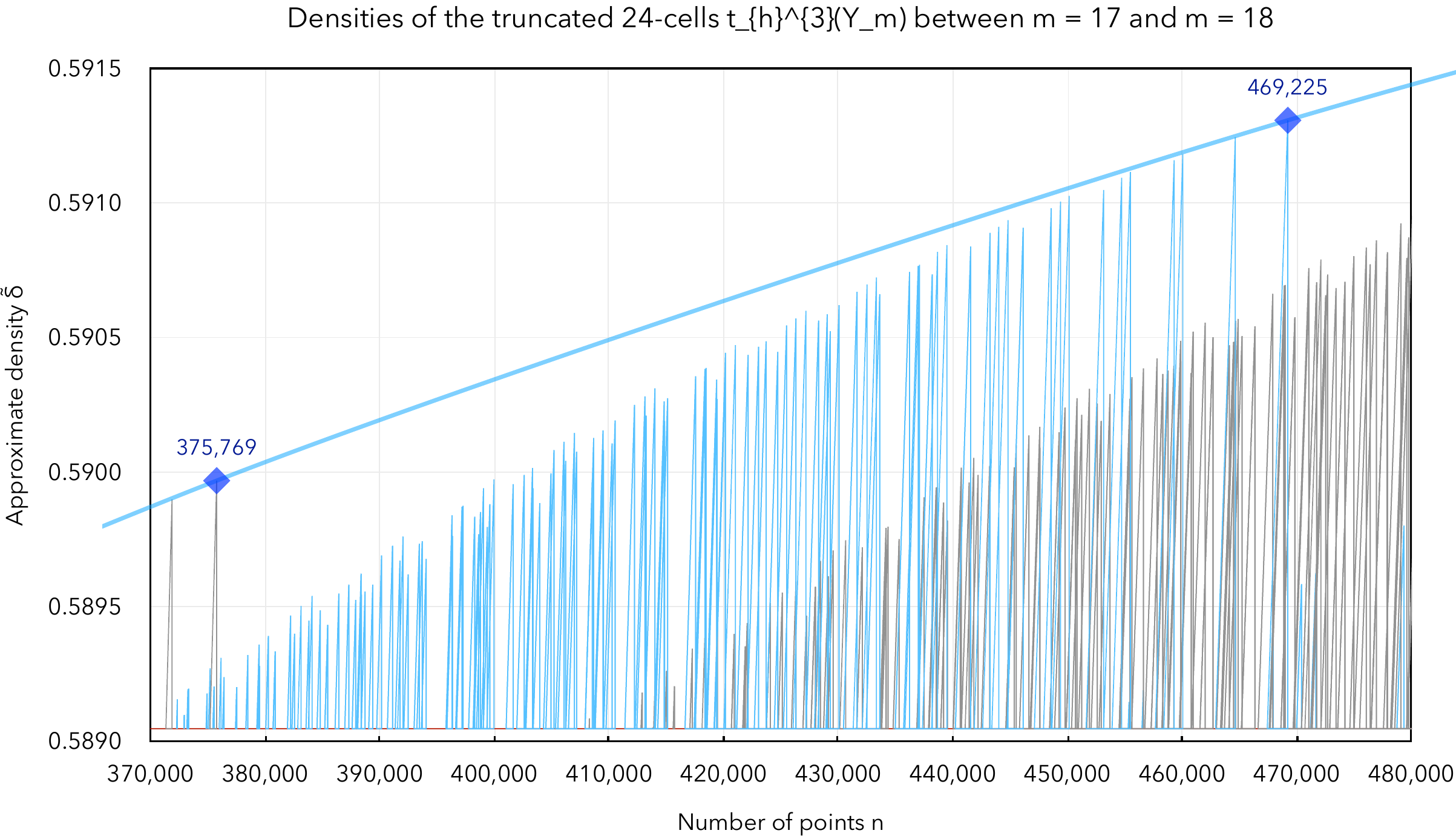}
\par\end{centering}
\caption{\protect\label{fig: The triangles for m =00003D 18}The \textquotedblleft triangles\textquotedblright{}
$\widehat{\triangle}\left(t_{\mathbf{h}}^{3}\left(Y_{18}\right)\cap D_{4}\right)$
for $\mathbf{h}\in\left\{ 0,\ldots,\left\lfloor \frac{m-1}{2}\right\rfloor \right\} ^{3}$.}
\end{figure}
\par\end{center}

\pagebreak{}
\noindent \begin{center}
\begin{figure}[H]
\noindent \begin{centering}
\includegraphics[width=6.24in]{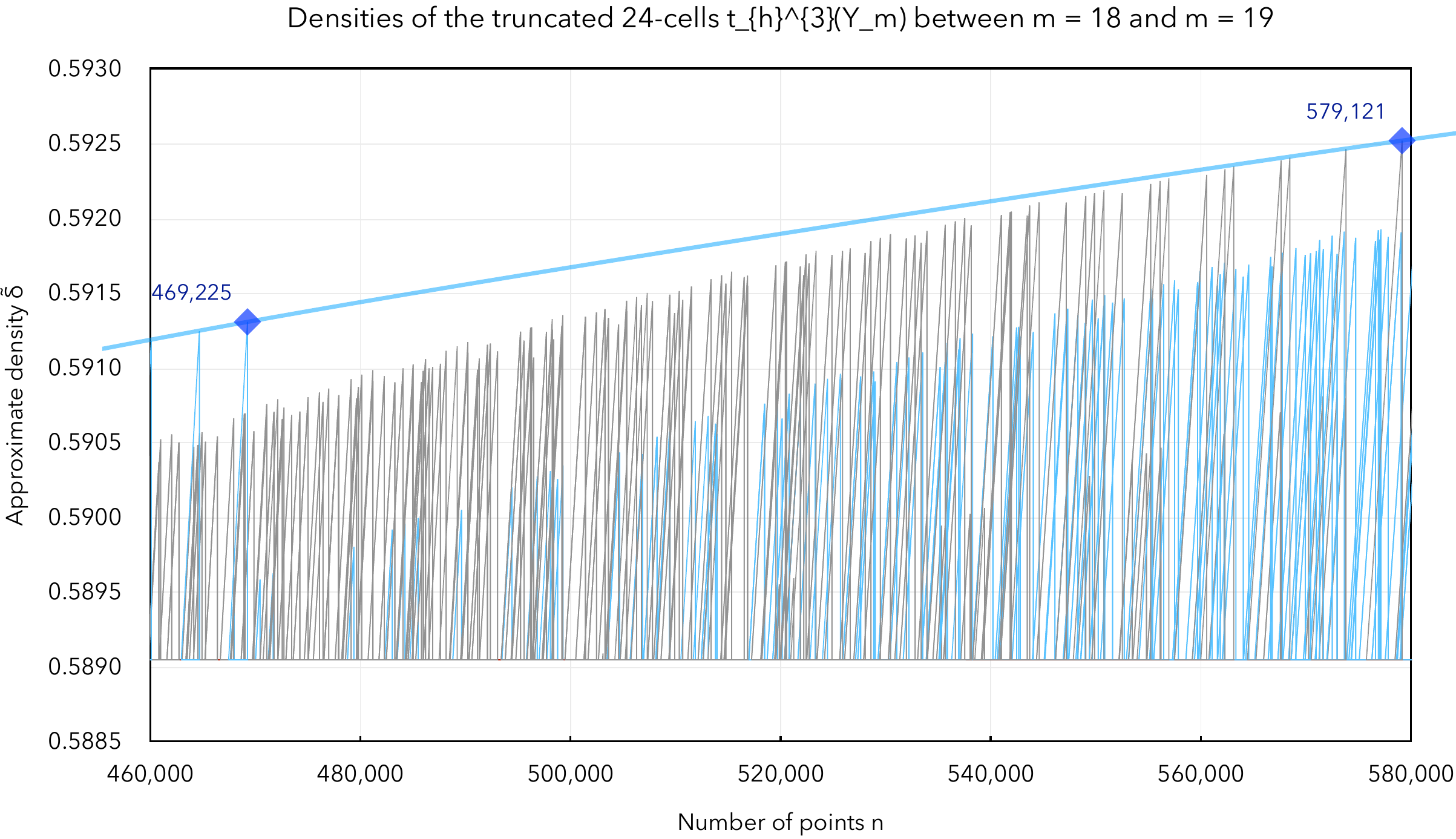}
\par\end{centering}
\caption{\protect\label{fig: The triangles for m =00003D 19}The \textquotedblleft triangles\textquotedblright{}
$\widehat{\triangle}\left(t_{\mathbf{h}}^{3}\left(Y_{19}\right)\cap D_{4}\right)$
for $\mathbf{h}\in\left\{ 0,\ldots,\left\lfloor \frac{m-1}{2}\right\rfloor \right\} ^{3}$.}
\end{figure}
\begin{figure}[H]
\noindent \begin{centering}
\includegraphics[width=6.24in]{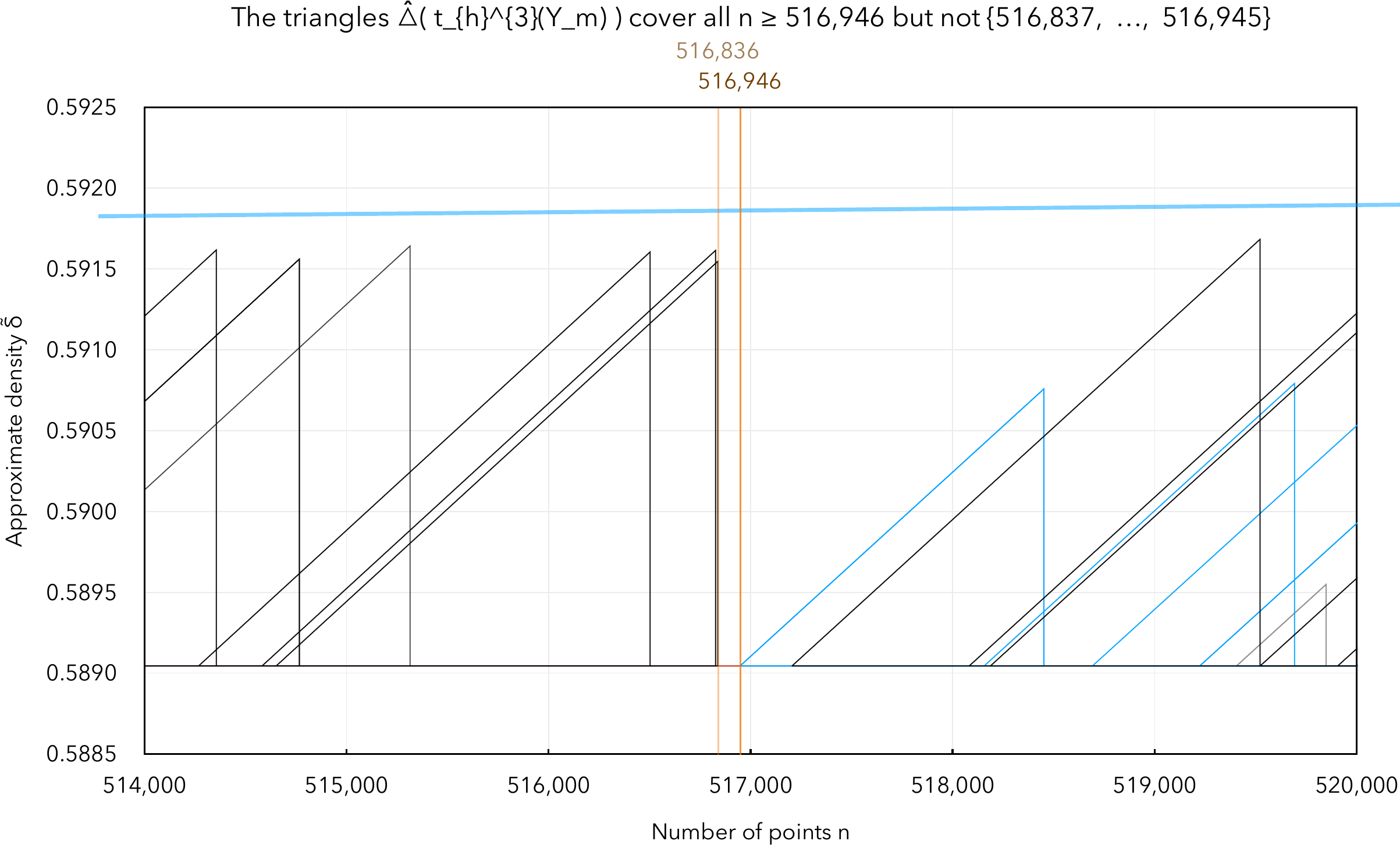}
\par\end{centering}
\caption{\protect\label{fig: The triangles for m =00003D 19, zoom}The \textquotedblleft triangles\textquotedblright{}
$\widehat{\triangle}\left(t_{\mathbf{h}}^{3}\left(Y_{19}\right)\cap D_{4}\right)$
and $\widehat{\triangle}\left(t_{\mathbf{h}}^{3}\left(Y_{20}\right)\cap D_{4}\right)$
for $\mathbf{h}\in\left\{ 0,\ldots,\left\lfloor \frac{m-1}{2}\right\rfloor \right\} ^{3}$
that are near $n=516,\!946$.}

In this graph, the darker orange vertical line is $x=516,\!946=\min\widetilde{L}\left(t_{\left(4,7,9\right)}^{3}\left(Y_{20}\right)\cap D_{4}\right)$
and the pale orange vertical line is $x=516,\!837=G\left(t_{\left(0,1,7\right)}^{3}\left(Y_{19}\right)\right)$.
\end{figure}
\par\end{center}

\pagebreak{}
\noindent \begin{center}
\begin{figure}[H]
\noindent \begin{centering}
\includegraphics[width=6.24in]{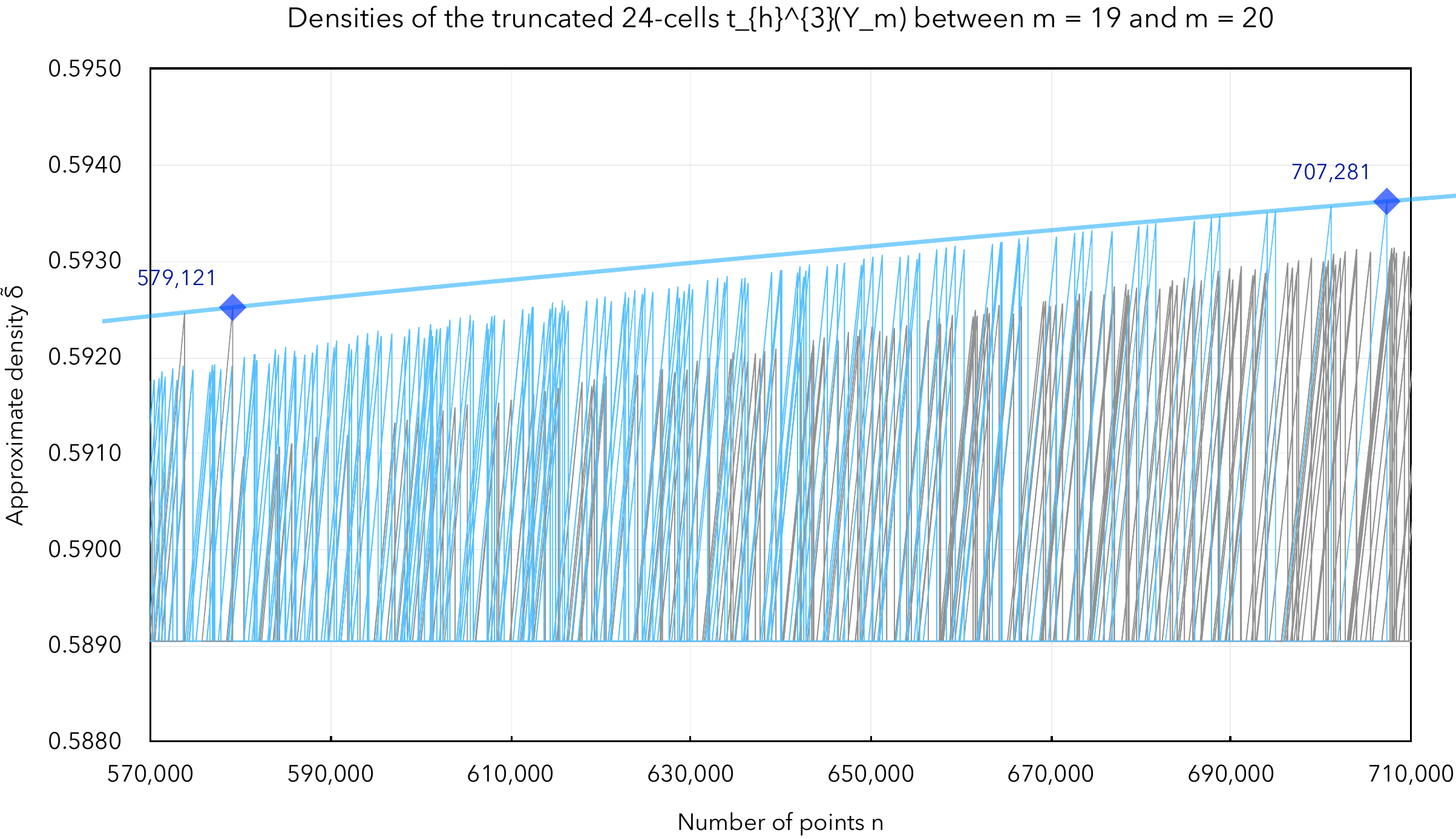}
\par\end{centering}
\caption{\protect\label{fig: The triangles for m =00003D 20}The \textquotedblleft triangles\textquotedblright{}
$\widehat{\triangle}\left(t_{\mathbf{h}}^{3}\left(Y_{20}\right)\cap D_{4}\right)$
for $\mathbf{h}\in\left\{ 0,\ldots,\left\lfloor \frac{m-1}{2}\right\rfloor \right\} ^{3}$.
One can see that $\widetilde{L}\left(Y_{20}\cap D_{4}\right)$ and
$\widetilde{L}\left(t_{\left(0,0,1\right)}^{3}\left(Y_{20}\right)\cap D_{4}\right)$
are disjoint.}
\end{figure}
\begin{figure}[H]
\noindent \begin{centering}
\includegraphics[width=6.24in]{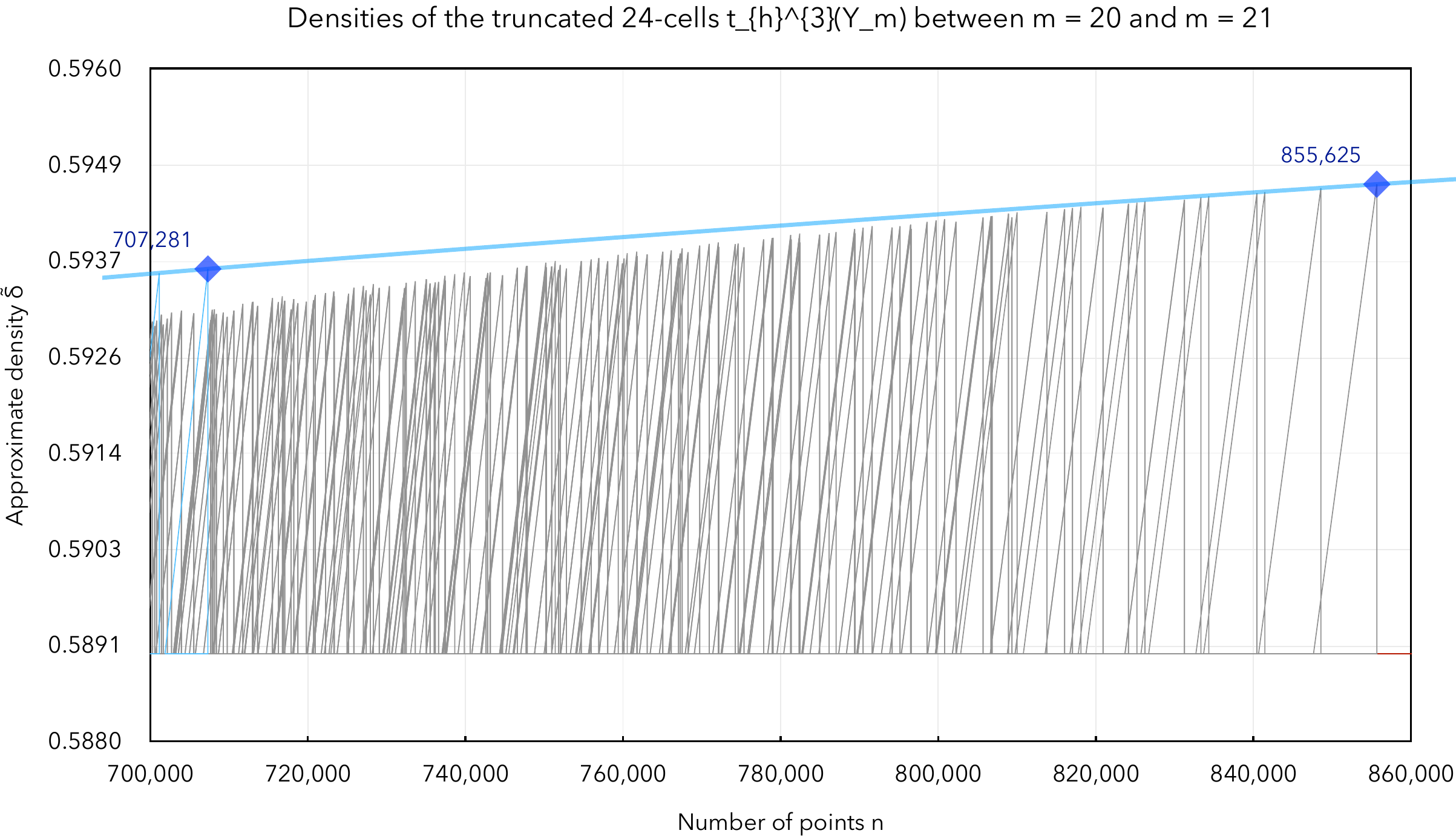}
\par\end{centering}
\caption{\protect\label{fig: The triangles for m =00003D 21}The \textquotedblleft triangles\textquotedblright{}
$\widehat{\triangle}\left(t_{\mathbf{h}}^{3}\left(Y_{21}\right)\cap D_{4}\right)$
for $\mathbf{h}\in\left\{ 0,\ldots,\left\lfloor \frac{m-1}{2}\right\rfloor \right\} ^{3}$.}

The rightmost two triangles in this graph are $\widehat{\triangle}\left(Y_{21}\cap D_{4}\right)$
and $\widehat{\triangle}\left(t_{\left(0,0,1\right)}^{3}\left(Y_{21}\right)\cap D_{4}\right)$
respectively (along with the other permutations $\left(0,1,0\right)$
and $\left(1,0,0\right)$). One can see that $\widetilde{L}\left(Y_{21}\cap D_{4}\right)$
and $\widetilde{L}\left(t_{\left(0,0,1\right)}^{3}\left(Y_{21}\right)\cap D_{4}\right)$
partially overlap. 
\end{figure}
\par\end{center}
\end{document}